\documentclass[11pt,reqno]{amsart}
\usepackage{amsmath,amssymb,latexsym,esint,cite,mathrsfs}
\usepackage{verbatim,wasysym}
\usepackage[left=2.5cm,right=2.5cm,top=3cm,bottom=3cm]{geometry}
\makeatletter \@addtoreset{equation}{section} \makeatother

\usepackage{graphicx}
\usepackage{subfigure}
\usepackage{graphicx,epic,eepic}

\usepackage{color,enumitem,graphicx}
\usepackage[colorlinks=true,urlcolor=blue,
citecolor=red,linkcolor=blue,linktocpage,pdfpagelabels,
bookmarksnumbered,bookmarksopen]{hyperref}
\usepackage[hyperpageref]{backref}
\usepackage[english]{babel}

\numberwithin{equation}{section}
\newtheorem{theorem}{Theorem}[section]

\newtheorem{definition}[theorem]{Definition}
\newtheorem{proposition}[theorem]{Proposition}
\newtheorem{corollary}[theorem]{Corollary}
\newtheorem{remark}[theorem]{Remark}
\numberwithin{equation}{section}

\allowdisplaybreaks

\begin{document}

\title[Nonlocal Sobolev inequality]
{Quantitative stability of a nonlocal Sobolev inequality}

\author[P. Piccione]{Paolo Piccione}
\address{\noindent Paolo Piccione  \newline
Instituto de Matematica e Estatistica,
Universidade de Sao Paulo, Sao Paulo, SP, Brazil}\email{piccione@ime.usp.br}

\author[M. Yang]{Minbo Yang}
\address{\noindent Minbo Yang  \newline
School of Mathematics, Zhejiang Normal University,
Jinhua 321004, Zhejiang, People's Republic of China.}\email{mbyang@zjnu.edu.cn}

\author[S. Zhao]{Shunneng Zhao}
\address{\noindent Shunneng Zhao  \newline
Department of Mathematics, Yunnan Normal University,
Kunming 650500, Yunnan, People's Republic of China.}\email{snzhao@zjnu.edu.cn}



\allowdisplaybreaks

\begin{abstract}
{\small In this paper, we study the quantitative stability of the nonlocal Soblev inequality
	\begin{equation*}
		S_{HL}\left(\int_{\mathbb{R}^N}\big(|x|^{-\mu} \ast |u|^{2_{\mu}^{\ast}}\big)|u|^{2_{\mu}^{\ast}} dx\right)^{\frac{1}{2_{\mu}^{\ast}}}\leq\int_{\mathbb{R}^N}|\nabla u|^2 dx , \quad \forall~u\in \mathcal{D}^{1,2}(\mathbb{R}^N),
	\end{equation*}
where $2_{\mu}^{\ast}=\frac{2N-\mu}{N-2}$ and $S_{HL}$ is a positive constant depending only on $N$ and $\mu$. For $N\geq3$, and $0<\mu<N$, it is well-known that, up to translation and scaling,
	the nonlocal Soblev inequality has a unique extremal function $W[\xi,\lambda]$ which is positive
	and radially symmetric.
	We first prove a result of quantitative stability of the nonlocal Soblev inequality with the level of gradients. Secondly, we also establish
	the stability of profile decomposition to the Euler-Lagrange equation of the above inequality for nonnegative functions. Finally we study the stability of the nonlocal Soblev inequality
	\begin{equation*}
		\Big\|\nabla u-\sum_{i=1}^{\kappa}\nabla W[\xi_i,\lambda_i]\Big\|_{L^2}\leq C\Big\|\Delta u+\left(\frac{1}{|x|^{\mu}}\ast |u|^{2_{\mu}^{\ast}}\right)|u|^{2_{\mu}^{\ast}-2}u\Big\|_{(\mathcal{D}^{1,2}(\mathbb{R}^N))^{-1}}
	\end{equation*}
	with the parameter region $\kappa\geq2$, $3\leq N<6-\mu$,  $\mu\in(0,N)$ satisfying $0<\mu\leq4$, or dimension $N\geq3$ and $\kappa=1$, $\mu\in(0,N)$ satisfying $0<\mu\leq4$.

  }

\end{abstract}

\vspace{3mm}

\maketitle

\section{Introduction and Main results}\label{sectid}
\subsection{The classical Sobolev inequality}
The Sobolev inequality with exponent $2$ states that, for any $N\geq 3$, there exists a dimensional constant $S=S(N)>0$ such that
\begin{equation}\label{bsic}
\|\nabla u\|^2_{L^2}\geq S\|u\|^2_{L^{2^*}},\quad  \forall~u\in \mathcal{D}^{1,2}(\mathbb{R}^N),
\end{equation}
where $2^*=2N/(N-2)$, and $\mathcal{D}^{1,2}(\mathbb{R}^N)$ denotes the closure of $C^\infty_c(\mathbb{R}^N)$ with respect to the norm $\|u\|_{\mathcal{D}^{1,2}(\mathbb{R}^N)}=\|\nabla u\|_{L^2}$. It is well known that the Euler-Lagrange equation associated to (\ref{bsic}) is given by
\begin{equation}\label{bec}
\Delta u+|u|^{2^*-2}u=0\quad \mbox{in}\ \ \mathbb{R}^N.
\end{equation}
The best constant $S$ in the Sobolev inequality is achieved by the Aubin-Talenti bubbles \cite{Ta76} $u=U[\xi,\lambda](x)$ defined by
\[U[\xi,\lambda](x)=[N(N-2)]^{\frac{N-2}{4}}\Big(\frac{\lambda}{1+\lambda^2|x-\xi|^2}\Big)^{\frac{N-2}{2}},\hspace{4mm}\lambda\in\mathbb{R}^{+},\hspace{4mm}\xi\in\mathbb{R}^N.\]
In fact, Caffarelli et al. \cite{CGS89} and Gidas et al. \cite{GNN79} proved that all the positive solutions to equation (\ref{bec}) are the Aubin-Talenti bubbles. In other words, the smooth  manifold of extremal function in \eqref{bsic}
$$\mathcal{M}:=\big\{cU[\xi,\lambda]:~c\in\mathbb{R},~\xi\in\mathbb{R}^N,~\lambda>0\big\}$$
is all nonnegative solution to equation \eqref{bec}. Later on, in \cite{BrL85}, Brezis and Lieb raised the question of stability for Sobolev inequality, whether
a remainder term proportional to the quadratic distance of the function $u$ to be the manifold $\mathcal{M}$ - can be added to the right hand side of (\ref{bsic}).
 This question was first answered by Bianchi and Egnell \cite{BE91} (see also \cite{AGP99}), that is,
 \begin{equation*}
 \inf\limits_{c\in\mathbb{R}, \xi\in\mathbb{R}^N,\lambda\in\mathbb{R}^{+}}\Big\|\nabla\big(u-cU[\xi,\lambda]\big)\Big\|_{L^2}^2\leq C(N)\Big(\big\|\nabla u\big\|_{L^2}^2-S^2\big\|u\big\|_{L^{2^\ast}}^2\Big).
 \end{equation*}

A natural and more challenging problem is to consider the Euler-Lagrange equation \eqref{bec}. Informally, whether a function $u$ almost solves \eqref{bec} must be quantitatively close to Aubin-Talenti bubbles.
In fact, a seminal work of Struwe \cite{Struwe-1984} proved the
well-known stability of profile decompositions of \eqref{bec}, that is $u$ sufficiently close to a sum of weakly-interacting bubbles even if we restrict to nonnegative functions. Starting with \cite{Struwe-1984}, another approach has been developed for studying the stability of the critical of functional inequalities in $\mathcal{D}^{1,2}(\mathbb{R}^N)$, see \cite{Aryan, CFM18, FG20,WW22, DSW21}.

Here we need to point out that Ciraolo et al. \cite{CFM18} obtained the first quantitative stability estimate of equation \eqref{bec} in the case where only one bubble and $N\geq3$. Later, Figalli and Glaudo \cite{FG20} proved that stability of Sobolev inequality in critical point setting for any finite number of bubbles. However, Figalli and Glaudo constructed counter-examples showing that when $N\geq6$ and bubbles $\kappa\geq2$, the $\mathcal{D}^{1,2}(\mathbb{R}^N)$-distance of $u$ from the manifold of
sums of Aubin-Talenti bubbles is much greater than the right side of \eqref{bec}, that is,
\begin{equation*}
 \inf\limits_{\substack{(z_{1},\cdots,z_{\kappa})\in\mathbb{R}^N,\\ \lambda_{1}, \cdots,\lambda_{\kappa}>0}}\Big\|\nabla\big(u-\sum_{i=1}^{\kappa}W[\xi_i,\lambda_i]\big)\Big\|_{L^2}^2\gg \Big\|\Delta u+|u|^{2^*-2}u\Big\|_{(\mathcal{D}^{1,2}(\mathbb{R}^N))^{-1}},
 \end{equation*}
and they also propose some conjectures in higher dimension $N\geq6$.  Recently, Deng et al. proved in \cite{DSW21} sharp quantitative estimates of Struwe's decomposition for Sobolev inequality (\ref{bsic}) by using the finite-dimensional reduction method and  it also completely solves the remaining dimension $N\geq6$.
\subsection{ The Hardy-Littlewood-Sobolev inequality}
As is well known, for a family of functional inequalities called Hardy-Littlewood-Sobolev inequality, which is in a certain sense dual to the family of Sobolev inequalities considered here. While we have consistently used
the latter formulation, it is worthwhile to explain this connection.

  The classical Hardy-Littlewood-Sobolev (HLS) inequality was introduced by Hardy and Littlewood \cite{Hardy-Littlewood-1} on $\mathbb{R}$, and generalised by Sobolev \cite{Sobolev-1} to $\mathbb{R}^N$. By rearrangement and symmetrisation, optimality is reduced to small classical of functions, for instance, to radial functions. This strategy was used by Lieb in \cite{Lieb83} to prove the existence of the extremal function to the HLS inequality with sharp constant, and computed the best constant, which can be stated as follows.
    \begin{proposition}\label{prohlsi}
     Let $\mu\in(0,N)$, $1<r,t<\infty$ and $\frac{1}{r}+\frac{1}{t}+\frac{\mu}{N}=2$. The following inequality holds for all $f\in L^r(\mathbb{R}^N)$ and $g\in L^t(\mathbb{R}^N)$,
    \begin{equation}\label{hlsi}
    \int_{\mathbb{R}^N}\int_{\mathbb{R}^N}f(x)|x-y|^{-\mu} g(y)dxdy\leq C(N,r,t,\mu)\|f\|_{L^r}\|g\|_{L^t}.
    \end{equation}
Moreover, $t=r=\frac{2N}{2N-\mu}$, the best constant is the following form
    \begin{equation}\label{cnu}
    C(N,r,t,\mu)=C(N,\mu)=\frac{\Gamma((N-\mu)/2)\pi^{\mu/2}}{\Gamma(N-\mu/2)}\left(\frac{\Gamma(N)}{\Gamma(N/2)}\right)^{1-\frac{\mu}{N}},
    \end{equation}
   and the equality holds if and only if
    \begin{equation*}
   f(x)=cg(x)=a\Big(\frac{1}{1+\lambda^2|x-x_0|^2}\Big)^{\frac{2N-\mu}{2}}
    \end{equation*}
    for some $a\in \mathbb{C}$, $\lambda\in \mathbb{R}\backslash\{0\}$ and $x_0\in \mathbb{R}^N$.
    \end{proposition}

\begin{remark}
The family of HLS inequalities is a two-parameter
family of inequalities, depending on parameters $\mu\in(0,N)$ and $1<r<\infty$, and inequality \eqref{hlsi} is equivalent to the following Riesz potential estimate for all $f\in L^r(\mathbb{R}^N)$,
\begin{equation}\label{hls-1}
\Big\|f\ast|x|^{-\mu}\Big\|_{L^{t}}\leq C(N,r,t,\mu)\big\|f\big\|_{L^r},\quad\text{with}\quad \frac{1}{r}+\frac{\mu}{N}=1+\frac{1}{t}.
\end{equation}
\end{remark}
In fact, the classical Sobolev inequality can be deduced from the HLS inequality.
According to the theory of differentiation, for any $N$-dimensional unit vector $\omega$,
$$
f(x)=-\int_{0}^{\infty}\frac{\partial}{\partial r}f(x+\omega r)dr,\quad\text{for any}~~ f\in\mathcal{C}_{c}^{\infty}(\mathbb{R}^N).
$$
Integrating on the unit sphere $\mathbb{S}^{N-1}$ yields
$$f(x)=\frac{1}{\omega_{N-1}}\int_{\mathbb{R}^N}\frac{(x-y)\cdot\nabla f(y)}{|x-y|^N}dy,$$
and so
$$|f(x)|=\frac{1}{\omega_{N-1}}\int_{\mathbb{R}^N}\frac{|\nabla f(y)|}{|x-y|^{N-1}}dy,\quad i.e.,\quad |f(x)|\leq C(N)\big||x|^{1-N}\ast|\nabla f|\big|.$$
As a consequence, the HLS inequality \eqref{hls-1} tells us that the classical Sobolev inequality holds
\begin{equation}\label{bsic-1}
\big\|f\big\|_{L^{\frac{Np}{N-p}}}\leq \mathcal{S}(p,N)\big\|\nabla f\big\|_{L^p},\quad\text{for all}~~1<p<N,\quad  \forall~u\in {D}^{1,p}(\mathbb{R}^N).
\end{equation}
On the other hand, the relation between the Sobolev inequalities and the HLS inequality come from \cite{Lieb-Loss} the well-known fact
\begin{equation*}
 \Big\|(-\Delta)^{-\frac{s}{2}}f\Big\|_{L^2}^2=\frac{1}{\pi^{\frac{N}{2}-2s}}\cdot\frac{\Gamma(N-2s)}{\Gamma(2s)}\int_{\mathbb{R}^N}\int_{\mathbb{R}^N}f(x)|x-y|^{-(N-2s)} g(y)dxdy,
\end{equation*}
for any $s\in(0,\frac{N}{2})$ and $f\in L^r(\mathbb{R}^N)$.
 In view of duality, there exists $\mathcal{S}(N,s)>0$ such that the fractional Sobolev inequality  was stated as,
\begin{equation}\label{fractional}
\Big\|(-\Delta)^{\frac{s}{2}}f\Big\|_{L^2}^2\geq\mathcal{S}(N,s)\big\|f\big\|_{L^{2N/(N-2s)}}^2\quad \text{for all}\quad f\in\dot{H}^{s}(\mathbb{R}^N),
\end{equation}
in the equivalent form
\begin{equation*}
\Big\|(-\Delta)^{-\frac{s}{2}}g\Big\|_{L^\frac{2N}{N-2s}}\leq\mathcal{S}^{-\frac{1}{2}}(N,s)\big\|g\big\|_{L^{2}}\quad \text{for all}\quad g\in L^{2}(\mathbb{R}^N).
\end{equation*}
Inequality \eqref{fractional} is valid for function $f$ in homogeneous Sobolev space $\dot{H}^{s}(\mathbb{R}^N)$ of tempered distributions whose Fourier transform
$$\hat{f}\in L^1_{loc}(\mathbb{R}^N)\quad\text{and}\quad\Big\|(-\Delta)^{\frac{s}{2}}f\Big\|_{L^2}^2=
\int_{\mathbb{R}^N}|\xi|^{2s}|\hat{f}|^2<\infty.
$$
In particular, by duality it is straight forward to check that inequality (\ref{hlsi}) corresponding to Sobolev inequality (\ref{bsic}) is given by, for $\mu=N-2$ and $r=t$,
\begin{equation}\label{fractional-0}
\int_{\mathbb{R}^N}f(-\Delta)^{-1}f\leq\frac{1}{\pi N(N-2)}\Big(\frac{\Gamma(N)}{\Gamma(N/2)}\Big)^{\frac{2}{N}}\big\|f\big\|_{L^{2N/(N+2)}}^2.
\end{equation}
 Therefore, it is natural to ask whether a remainder term can be added to the left hand side of inequality (\ref{hlsi}), (\ref{fractional}) or (\ref{fractional-0}). Successfully, Chen, Frank and Weth in \cite{CFW13} gave the affirmative answer regarding the stability of (\ref{fractional}), that is,
   \begin{equation}\label{chen-F}
   \big\|(-\Delta)^{\frac{s}{2}}f\big\|_{L^2}^{2}-\mathcal{S}(N,s)\big\|f\big\|_{L^t}^2\geq
   K_{N,S}\inf\limits_{g\in\mathcal{M}_{N,S}}\big\|(-\Delta)^{\frac{s}{2}}(f-g)\big\|_{L^2}^2\quad\text{for all}~~s\in(0,N/2),
   \end{equation}
where $\mathcal{M}_{N,S}$ is the manifold of all
optimal functions, which is generated from $v(x)=(1+x^2)^{(2s-N)/2}$ by multiplication by a
constant, translations and scalings.
For HLS inequality (\ref{hlsi}) with the special case $t=r=\frac{2N}{2N-\mu}$. Carlen \cite{Ca17} proved for all $u\in L^{\frac{2N}{2N-\mu}}(\mathbb{R}^N)$,
    \begin{equation*}
    C(N,\mu)\big\|f\big\|^2_{L^{\frac{2N}{2N-\mu}}}-\int_{\mathbb{R}^N}\int_{\mathbb{R}^N}\frac{f(x)f(y)}{|x-y|^\mu}dydx
    \geq K_{HL,\mu}\inf_{g\in \mathcal{M}_{HL,\mu}}\big\|f-g\big\|^2_{L^{\frac{2N}{2N-\mu}}},
    \end{equation*}
    where constant $K_{HLS,\mu}>0$ depending only on $N$ and $\mu$, $\mathcal{M}_{HL,\mu}$ is the manifold of all
optimal functions, which is generated from $v(x)=(1+x^2)^{(\mu-2N)/2}$ by multiplication by a
constant, translations and scalings. As a straightforward consequence of the duality approach,  an explicit stability results for HLS inequality are obtained in \cite{Dolbeault-1,Dolbeault-2} by using flow methods.
Recently, Dolbeault and Esteban \cite{DE22} first established a constructive local stability result of (\ref{fractional-0}) in a neighbourhood of the optimal functions, with respect to very strong topology of relative uniform convergence.

\subsection{ The nonlocal Sobolev inequality}
Let $N\geq3$ and $\mu\in(0,N)$. We define the Coulomb space $X_{NL}^{(q)}$ as the vector space of measurable functions $u:\mathbb{R}^N\rightarrow\mathbb{R}$ such that
\begin{equation*}
\|u\|_{X_{NL}^{(q)}}=\Big(\int_{\mathbb{R}^N}\big(\frac{1}{|x|^{\mu}} \ast u^{q}\big)u^{q} dx\Big)^{\frac{1}{2q}}<\infty.
\end{equation*}
 It can be observed that for every measurable function $u\in X_{NL}^{(q)}$ if and only if $|u|^{q}\in X_{NL}^{(1)}$. By the HLS inequality we have
$$
\int_{\mathbb{R}^N}\Big(\frac{1}{|x|^{\mu}} \ast u^{q}\Big)u^{q}\leq C\Big(\int_{\mathbb{R}^N}|u|^{\frac{2Nq}{2N-\mu}}\Big)^{\frac{2N-\mu}{N}}
$$
and thus $L^{\frac{2Nq}{2N-\mu}}(\mathbb{R}^N)\subset X_{NL}^{(q)}$. It is not difficult to see that the Coulomb space  $X_{NL}^{(q)}$ is a Banach space under the norm $
\|\cdot\|_{X_{NL}^{(q)}}$. See \cite{GY18, Mercuri}. By the HLS inequality, the integral
    \begin{equation*}
    \int_{\mathbb{R}^N}\int_{\mathbb{R}^N}\frac{|u(x)|^q|v(y)|^q}{|x-y|^\mu}dydx
    \end{equation*}
    is well-defined if for $u\in L^{qt}(\mathbb{R}^N)$ satisfying
  $\frac{2}{t}+\frac{\mu}{N}=2$.
     Hence, for $u\in H^1(\mathbb{R}^N)$, the  continuous Sobolev embedding
    $H^1(\mathbb{R}^N)\hookrightarrow L^\gamma(\mathbb{R}^N)$ for every $\gamma\in \left[2,\frac{2N}{N-2}\right]$
    implies that
    \begin{equation*}
    \frac{2N-\mu}{N}\leq q\leq \frac{2N-\mu}{N-2}.
    \end{equation*}
    Hence we may call $\frac{2N-\mu}{N}$ the low critical exponent and $\frac{2N-\mu}{N-2}=:2_{\mu}^{\ast}$ the upper critical exponent due to the HLS inequality. Consider the upper critical case, we are concerned with the following nonlocal Sobolev inequality
\begin{equation}\label{Prm}
\int_{\mathbb{R}^N}|\nabla u|^2 dx
\geq S_{HL}\left(\int_{\mathbb{R}^N}(|x|^{-\mu} \ast|u|^{2^*_\mu})|u|^{2^*_\mu} dx\right)^{\frac{1}{2^*_\mu}}, \quad u\in \mathcal{D}^{1,2}(\mathbb{R}^N),
\end{equation}
for some positive constant $S_{HL}$ depending only on $N$ and $\mu$, where $N\geq 3$, $0<\mu<N$.
It is well-know that the optimal constant in \eqref{Prm} is given by
    \begin{equation*}
S_{HL}=S\left[\frac{\Gamma((N-\mu)/2)\pi^{\mu/2}}{\Gamma(N-\mu/2)}\left(\frac{\Gamma(N)}{\Gamma(N/2)}\right)^{1-\frac{\mu}{N}}\right]^{(2-N)/(2N-\mu\mu)},
    \end{equation*}
where $S$ is the best Sobolev constant. What's more, the equality is achieved in \eqref{Prm} if and only if by
    \begin{equation}\label{defU}
    W[\xi,\lambda](x)=S^{\frac{(N-\mu)(2-N)}{4(N-\mu+2)}}[C(N,\mu)]^{\frac{2-N}{2(N-\mu+2)}}[N(N-2)]^{\frac{N-2}{4}}\Big(\frac{\lambda}{1+\lambda^2|x-\xi|^2}\Big)^{\frac{N-2}{2}},\hspace{1mm}\lambda\in\mathbb{R}^{+},\hspace{1mm}\xi\in\mathbb{R}^N,
    \end{equation}
    which satisfies the Euler-Lagrange equation of (\ref{Prm})
    \begin{equation}\label{ele}
    \Delta u+(|x|^{-\mu}\ast |u|^{2^*_\mu})|u|^{2^*_\mu-2}u=0 \quad \mbox{in}\quad \mathbb{R}^N.
    \end{equation}
See \cite{GY18, DY19, GHPS19,Le18}.

As far as we know there is few results about the stability of the nonlocal Sobolev inequality and it is interesting to study the quantitative stability estimate of the nonlocal Sobolev inequality \eqref{Prm}. Very recently, Deng et al. \cite{DSB213} gave a first result that the gradient type remainder term of inequality \eqref{Prm}, that is,
 \[
    B_2{\rm dist}(u,\mathcal{\widetilde{M}})^2
    \geq \int_{\mathbb{R}^N}|\nabla u|^2 dx
    -S_{HL}\left(\int_{\mathbb{R}^N}\big(|x|^{-\mu} \ast u^{2_{\mu}^*}\big)u^{2_{\mu}^*} dx\right)^{\frac{1}{{2_{\mu}^*}}}
    \geq B_1 {\rm dist}(u,\mathcal{\widetilde{M}})^2,
    \]
    where
    $$
    \mathcal{\widetilde{M}}=\big\{cW[\xi,\lambda]:~c\in\mathbb{R},~\xi\in\mathbb{R}^N,~\lambda>0\big\}
    $$
    is $N+2$-dimensional manifold, and ${\rm dist}(u,\mathcal{\widetilde{M} }):=\inf\limits_{c\in\mathbb{R}, \lambda>0, z\in\mathbb{R}^N}\|u-cW[\xi,\lambda]\|_{\mathcal{D}^{1,2}(\mathbb{R}^N)}$.

We first establish the following quantitative estimate for \eqref{Prm} at the level of gradients.
\begin{theorem}\label{thme1}
Fix $N\geq3$ and $\mu\in(0,N)$. There exist constants $K, L>0$, depending only on $N$ and $\mu$, such that for all $u\in \mathcal{D}^{1,2}(\mathbb{R}^N)$ and for any $v\in \mathcal{\widetilde{M}}$ with $\|u\|_{NL}=\|v\|_{NL}$,
\begin{equation*}
\left\|\nabla u-\nabla v\right\|_{L^{2}}^2\leq K(N,\mu)\widetilde{\delta}(u)+L(N,\mu)\left\| u\right\|_{NL}\left\| u-v\right\|_{NL}.
\end{equation*}
Here~$\|\cdot\|_{NL}:=\|\cdot\|_{X_{NL}^{(2_{\mu}^{\ast})}}$ for simplicity, $\widetilde{\delta}(u):=\int_{\mathbb{R}^N}|\nabla u|^2 dx
-S_{HL}\left(\int_{\mathbb{R}^N}(|x|^{-\mu} \ast|u|^{2^*_\mu})|u|^{2^*_\mu} dx\right)^{\frac{1}{2^*_\mu}}$ and
$$
S_{HL}=S\left[\frac{\Gamma((N-\mu)/2)\pi^{\mu/2}}{\Gamma(N-\mu/2)}\left(\frac{\Gamma(N)}{\Gamma(N/2)}\right)^{1-\frac{\mu}{N}}\right]^{(2-N)/(2N-\mu\mu)}.
$$
\end{theorem}
Inspired by the spirit of Struwe in \cite{Struwe-1984}, a nonlocal version of the stability of profile decompositions to \eqref{ele} for nonnegative functions states the following:
\begin{theorem}\label{F4}
   Let $N\geq3$ and $\kappa\geq1$ be positive integers. Let $(u_m)_{m\in\mathbb{N}}\subseteq\mathcal{D}^{1,2}(\mathbb{R}^N)$ be a sequence of nonnegative functions such that $$(\kappa-\frac{1}{2})S_{HL}^{\frac{2N-\mu}{N+2-\mu}}\leq\|u_m|_{\mathcal{D}^{1,2}(\mathbb{R}^N)}^2\leq(\kappa+\frac{1}{2})S_{HL}^{\frac{2N-\mu}{N+2-\mu}}$$ with $S_{HL}=S_{HL}(N,\mu)$ as in \eqref{Prm}, and assume that
    \[
    \Big\|\Delta u_m+\left(\frac{1}{|x|^{\mu}}\ast |u_m|^{2_{\mu}^{\ast}}\right)|u_m|^{2_{\ast}^{\ast}-2}u_m\Big\|_{(\mathcal{D}^{1,2}(\mathbb{R}^N))^{-1}}\rightarrow0\quad \mbox{as}\quad m\rightarrow\infty.
    \]
Then there exist $\kappa$-tuples of points $(\xi_{1}^{(m)},\cdots, \xi_{\kappa}^{(m)})_{m\in\mathbb{N}}$ in $\mathbb{R}^N$ and $\kappa$-tuples of positive real numbers $(\lambda_{1}^{(m)},\cdots,\lambda_{\kappa}^{(m)})_{m\in\mathbb{N}}$  such that
\[
\Big\|\nabla\Big(u_m-\sum_{i=1}^{\kappa}W[\xi_i^{(m)},\lambda_{i}^{(m)}]\Big)\Big\|_{L^2}\rightarrow0\quad \mbox{as}\quad m\rightarrow\infty.
\]
\end{theorem}
\begin{remark}\label{remark0}
Note that,
\[
    \Big\|\Delta u_m+\left(\frac{1}{|x|^{\mu}}\ast |u_m|^{2_{\mu}^{\ast}}\right)|u_m|^{2_{\ast}^{\ast}-2}u_m\Big\|_{(\mathcal{D}^{1,2}(\mathbb{R}^N))^{-1}}\rightarrow0\quad \mbox{as}\quad m\rightarrow\infty.
    \]
    is equivalent to saying that $u_m$ is Palais-Smale sequence for the functiona $E_0$ corresponding to \eqref{ele}
    \begin{equation*}
E_0(u; \mathbb{R}^N):=\frac{1}{2}\int_{\mathbb{R}^N}|\nabla u|^2-
	\frac{1}{2\cdot2_{\mu}^{\ast}}
	\int_{\mathbb{R}^N}\int_{\mathbb{R}^N}\frac{u^{2_{\mu}^\ast
		}(x)u^{2_{\mu}^\ast}(y)}{|x-y|^{\mu}}, \quad\quad u\in H_0^1(\mathbb{R}^N).
    \end{equation*}
\end{remark}

It is much more difficult to study the stability of the nonlocal Sobolev inequality via the Euler-Lagrange equation \eqref{ele}. In this paper, we are going to study the stability of the nonlocal Sobolev equality and give a first attempt to estimate the stability of the nonlocal Sobolev inequality via the Euler-Lagrange equation \eqref{ele}.
Before stating our results, it is necessary to introduce the definition of interaction of Talenti bubbles.
\begin{definition}[\cite{FG20,DSW21}]\label{definition}
Let $W[\xi_i,\lambda_i]$ and $W[\xi_j,\lambda_j]$ be two bubbles. Define the interaction of them by
\begin{equation*}		Q(\xi_i,\xi_j,\lambda_i,\lambda_j)=\min\Big(\frac{\lambda_i}{\lambda_j}+\frac{\lambda_j}{\lambda_i}+\lambda_i\lambda_j|\xi_i-\xi_j|^2\Big)^{-\frac{N-2}{2}}.
	\end{equation*}
Let $(W[\xi_i,\lambda_i])_{1\leq i\leq \kappa}$ be a family of Talenti bubbles. We say that the family is $\delta$-interacting if
\begin{equation}\label{interacting}
\max\big\{Q(\xi_i,\xi_j,\lambda_i,\lambda_j):~i,j=1,\cdots,\kappa\big\}<\delta.
\end{equation}

\end{definition}

Our main result is in the following:
    \begin{theorem}\label{Figalli}
  For any dimension $3\leq N<6-\mu$ and $\kappa\geq2$, $\mu\in(0,N)$ satisfying $0<\mu\leq4$, there exist a small constant $\delta=\delta(N,\kappa)>0$ and a large constant $C=C(N,\kappa)>0$ such that the following statement holds. Let $u\in \mathcal{D}^{1,2}(\mathbb{R}^N)$ be a function such that
  \begin{equation*}
\Big\|\nabla u-\sum_{i=1}^{\kappa}\nabla \widetilde{W}[\xi_i,\lambda_i]\Big\|_{L^2}\leq\delta,
\end{equation*}
where $\big(\widetilde{W}[\xi_i,\lambda_i]\big)_{1\leq i\leq\kappa}$ is a $\delta$-interacting family of Talenti bubbles.
Then there exist $\kappa$ Talenti bubbles  $(W[\xi_1,\lambda_1],\cdots,W[\xi_\kappa,\lambda_\kappa])$ such that
\begin{equation*}
dist_{\mathcal{D}^{1,2}}\big(u,\mathcal{\widetilde{M}}_0^{\kappa}\big)\leq C\Big\|\Delta u+\left(\frac{1}{|x|^{\mu}}\ast |u|^{2_{\mu}^{\ast}}\right)|u|^{2_{\ast}^{\ast}-2}u\Big\|_{(\mathcal{D}^{1,2}(\mathbb{R}^N))^{-1}},
\end{equation*}
where
$$\mathcal{\widetilde{M}}_0^{\kappa}=\Big\{\sum_{i=1}^{\kappa}W[\xi_i,\lambda_i]:  ~\xi_i\in\mathbb{R}^N,~\lambda_i>0\Big\}.$$
    \end{theorem}

As a direct consequence of Theorem \ref{Figalli}, we can obtain the following two corollaries.
 \begin{corollary}\label{Figalli2}
  For any dimension $3\leq N<6-\mu$, $\mu\in(0,N)$ satisfying $0<\mu\leq4$, and $\kappa\in\mathbb{N}$, there exists a constant constant $C=C(N,\kappa)>0$ such that the following statement holds. For any nonnegative function $u\in \mathcal{D}^{1,2}(\mathbb{R}^N)$ such that
\begin{equation*}
\big(\kappa-\frac{1}{2})S_{HL}^{\frac{2N-\mu}{N+2-\mu}}\leq\|u\|_{\mathcal{D}^{1,2}(\mathbb{R}^N)}^2\leq\big(\kappa+\frac{1}{2}\big)S_{HL}^{\frac{2N-\mu}{N+2-\mu}},
\end{equation*}
 there exist $\kappa$ Talenti bubbles  $(W[\xi_1,\lambda_1],\cdots,W[\xi_\kappa,\lambda_\kappa])$ such that
\begin{equation*}
dist_{\mathcal{D}^{1,2}}\big(u,\mathcal{\widetilde{M}}_0^{\kappa}\big)\leq C\Big\|\Delta u+\left(\frac{1}{|x|^{\mu}}\ast |u|^{2_{\mu}^{\ast}}\right)|u|^{2_{\ast}^{\ast}-2}u\Big\|_{(\mathcal{D}^{1,2}(\mathbb{R}^N))^{-1}}.
\end{equation*}
Furthermore, for any $i\neq j$, the interaction between the bubbles can be estimated as
\begin{equation*}
\int_{\mathbb{R}^N}\Big(\frac{1}{|x|^{\mu}}\ast W[\xi_i,\lambda_i]^{2_{\mu}^{*}}\Big)W[\xi_i,\lambda_i]^{2_{\mu}^{*}-1}W[\xi_j,\lambda_j]\leq C\Big\|\Delta u+\left(\frac{1}{|x|^{\mu}}\ast |u|^{2_{\mu}^{\ast}}\right)|u|^{2_\mu^{\ast}-2}u\Big\|_{(\mathcal{D}^{1,2}(\mathbb{R}^N))^{-1}}.
\end{equation*}
    \end{corollary}
    \begin{corollary}\label{Figalli-0}
  For any dimension $N\geq3$ and $\kappa=1$, $\mu\in(0,N)$ satisfying $0<\mu\leq4$, there exists a large constant $C=C(N,\kappa)>0$ such that the following statement holds. Let $u\in \mathcal{D}^{1,2}(\mathbb{R}^N)$ be a nonnegative function such that
  \begin{equation}\label{Fi}
\frac{1}{2}S_{HL}^{\frac{2N-\mu}{N+2-\mu}}\leq\|u\|_{\mathcal{D}^{1,2}(\mathbb{R}^N)}^2\leq\frac{3}{2}S_{HL}^{\frac{2N-\mu}{N+2-\mu}}.
\end{equation}
Then there exists $1$ Talenti bubble $W[\xi,\lambda]$ such that
\begin{equation}\label{Fi1}
dist_{\mathcal{D}^{1,2}}\big(u,\mathcal{\widetilde{M}}_0^{1}\big)\leq C\Big\|\Delta u+\left(\frac{1}{|x|^{\mu}}\ast |u|^{2_{\mu}^{\ast}}\right)|u|^{2_{\ast}^{\ast}-2}u\Big\|_{(\mathcal{D}^{1,2}(\mathbb{R}^N))^{-1}},
\end{equation}
where
$\mathcal{\widetilde{M}}_0^{1}=\big\{W[\xi,\lambda]:  \xi\in\mathbb{R}^N, \lambda>0\big\}.$
    \end{corollary}

To prove the main results in Theorem \ref{Figalli} and Corollaries \ref{Figalli2}-\ref{Figalli-0}, the arguments depend a lot on the nondegeneracy property of
the positive solutions of equation \eqref{ele}. The topic of non-degeneracy of the ground state solution to the semilinear elliptic equation or nonlinear Hartree equation has generated much interest in recent years. It is a key ingredient in the stability analysis of functional inequality and Lyapunov-Schmidt reduction method of constructing  blow-up solutions of the equation, see, for example \cite{Medina}. Recently,
Du and Yang \cite{DY19} proved that if $\mu$ is close to $N$ with $N=3$ or $4$, $W[\xi,\lambda](x)$ as in (\ref{defU}) is nondegenerate in the sense that solutions of the linearized equation
\begin{equation}\label{deflp}
-\Delta \phi=2^*_\mu\left(|x|^{-\mu} \ast (W^{2^*_\mu-1}[\xi,\lambda]\phi)\right)W^{2^*_\mu-1}[\xi,\lambda] +(2^*_\mu-1)\left(|x|^{-\mu} \ast W^{2^*_\mu}[\xi,\lambda]\right)W^{2^*_\mu-2}[\xi,\lambda]\phi
 \end{equation}
for $x\in\mathbb{R}^N$. $\phi\in \mathcal{D}^{1,2}(\mathbb{R}^N)$ are linear combinations of functions $\frac{N-2}{2}W[\xi,\lambda]+x\cdot\nabla W[\xi,\lambda]$ and $\partial_{x_j} W[\xi,\lambda]$, $j=1,\ldots,N$. However, it is hard to investigate the expansion of the nonlocal term by spherical harmonics since the working space is $\mathcal{D}^{1,2}(\mathbb{R}^N)$ for the critical case. So it is necessary to borrow some other ideas to analyze the nondegeneracy property for the critical problem.
Later, Gao et al. \cite{GMYZ22} showed that a nondegeneracy result at $W[\xi,\lambda]$ for \eqref{ele} when $N=6$ and $\mu=4$, and also proposed that the problem is an open within the remaining range of $N,\mu$. Recently, Li et al. in \cite{XLi} gave a complete describtion of the nondegeneracy which can be stated as
follows:

\vskip0.25cm

\begin{proposition}\label{prondgr}
Assume $N\geq 3$, $0<\mu<N$ with $0<\mu\leq4$. Let $W[\xi,\lambda]$ be as in \eqref{defU}.
Then the linearized operator of equation (\ref{ele}) at $W[\xi,\lambda]$ defined by
\begin{equation*}
L[\phi]:=-\Delta \phi-2^*_\mu\left(|x|^{-\mu} \ast (W^{2^*_\mu-1}[\xi,\lambda]\phi)\right)W^{2^*_\mu-1}[\xi,\lambda]
-(2^*_\mu-1)\left(|x|^{-\mu} \ast W^{2^*_\mu}[\xi,\lambda]\right)W^{2^*_\mu-2}[\xi,\lambda]\phi
\end{equation*}
only admits solutions in $\mathcal{D}^{1,2}(\mathbb{R}^N)$ of the form
\begin{equation*}
\phi=aD_\lambda W[\xi,\lambda] +\mathbf{b}\cdot\nabla W[\xi,\lambda],
\end{equation*}
where $a\in\mathbb{R}$, $\mathbf{b}\in\mathbb{R}^N$.
\end{proposition}

The paper is organized as follows. Theorem~\ref{thme1} is proved in Section \ref{sectevp0}. We first establish a stability of profile decompositions and complete the proof of Theorem \ref{F4} in Section \ref{sectrt}.  Later, we describe proof of Theorem \ref{Figalli} and Corollaries \ref{Figalli2}-\ref{Figalli-0} in Section \ref{sectrt}. The proof of Theorem \ref{Figalli} requires some crucial Propositions, the proofs of which are given in Section~\ref{section5}.

\textbf{Notations.} Throughout this paper, $C$ and $\widetilde{C}$ are indiscriminately used to denote various absolutely positive constants. We say that $a\lesssim b$ if $a\leq Cb$, $a\approx b$ if $a\lesssim b$ and $a\gtrsim b$.

\section{Quantitative estimate for nonlocal Sobolev inequality}\label{sectevp0}
In this section, we prove Theorem \ref{thme1} using a version of Clarkson's inequality for vector-valued functions \cite{Neumayer}, which we recall here:
  \begin{proposition}\label{Clarkson}
Let $X,Y:\mathbb{R}^N\rightarrow\mathbb{R}^N$ with $|X|,|Y|\in L^2(\mathbb{R}^N)$. Then
\begin{equation*}
\Big\|\frac{X+Y}{2}\Big\|_{L^2}^2+\Big\|\frac{X-Y}{2}\Big\|_{L^2}^2\leq\frac{1}{2}\big\|X\big\|_{L^2}^2+
\frac{1}{2}\big\|Y\big\|_{L^2}^2.
\end{equation*}
 \end{proposition}
\textbf{The proof of Theorem~\ref{thme1}.} Applying the Clarkson's inequality in Proposition with $X=\nabla u$ and $Y=\nabla v$, we have
\begin{equation*}
\begin{split}
\Big\|\frac{\nabla u-\nabla v}{2}\Big\|_{L^2}^2&\leq\frac{1}{2}\big\|\nabla u\big\|_{L^2}^2+
\frac{1}{2}\big\|\nabla v\big\|_{L^2}^2-\Big\|\frac{\nabla u+\nabla v}{2}\Big\|_{L^2}^2.\\
\end{split}
\end{equation*}
Combining $u,v$ satisfy the nonlocal Sobolev inequality \eqref{Prm}, we find
\begin{equation*}
\big\|\nabla v\big\|_{L^2}^2\leq\big\|\nabla u\big\|_{L^2}^2
\end{equation*}
and
\begin{equation*}
\Big\|\nabla u+\nabla v\Big\|_{L^2}^2\geq S_{HL}\left(\int_{\mathbb{R}^N}\big(|x|^{-\mu} \ast|u+v|^{2^*_\mu}\big)|u+v|^{2^*_\mu}\right)^{\frac{1}{2^*_\mu}}.
\end{equation*}
Hence we obtain
\begin{equation*}
\begin{split}
\Big\|\frac{\nabla u-\nabla v}{2}\Big\|_{L^2}^2&\leq\Big\|\nabla u\Big\|_{L^2}^2-S_{HL}\left(\int_{\mathbb{R}^N}\Big(|x|^{-\mu} \ast\Big|\frac{u+v}{2}\Big|^{2^*_\mu}\Big)\Big|\frac{u+v}{2}\Big|^{2^*_\mu} \right)^{\frac{1}{2^*_\mu}}.\\
\end{split}
\end{equation*}
By the semigroup property of the Riesz potential, we omit the coefficients of for convenience and find that
\begin{equation*}
\int_{\mathbb{R}^N}\Big(|x|^{-\mu} \ast\Big|\frac{u+v}{2}\Big|^{2^*_\mu}\Big)\Big|\frac{u+v}{2}\Big|^{2^*_\mu}=\int_{\mathbb{R}^N}\left(|x|^{-\frac{N+\mu}{2}} \ast\Big|\frac{u+v}{2}\Big|^{2^*_\mu}\right)^2
\end{equation*}
and
\begin{equation*}
\int_{\mathbb{R}^N}\Big(|x|^{-\mu} \ast\Big|\frac{u-v}{2}\Big|^{2^*_\mu}\Big)\Big|\frac{u-v}{2}\Big|^{2^*_\mu}=\int_{\mathbb{R}^N}\left(|x|^{-\frac{N+\mu}{2}} \ast\Big|\frac{u-v}{2}\Big|^{2^*_\mu}\right)^2.
\end{equation*}
Combining Minkowski's inequality, then we have
\begin{equation*}
\begin{split}
\left(\int_{\mathbb{R}^N}\frac{|u(y)| ^{2^*_\mu}}{|x-y|^{\frac{N+\mu}{2}}}\right)^2&=\left(\int_{\mathbb{R}^N}\bigg|\frac{\frac{u+v}{2}}{|x-y|^{\frac{N+\mu}{2}\cdot 2_{\mu}^{\ast}}}+\frac{\frac{u-v}{2}}{|x-y|^{\frac{N+\mu}{2}\cdot 2_{\mu}^{\ast}}}\bigg|^{{2_\mu^{\ast}}}\right)^{\frac{1}{2^*_\mu}\cdot2\cdot2_\mu^{\ast}}\\&
\leq \bigg(\bigg\|\frac{\frac{u+v}{2}}{|x-y|^{\frac{N+\mu}{2}\cdot 2_{\mu}^{\ast}}}\bigg\|_{L^{2^*_\mu}}+\bigg\|\frac{\frac{u-v}{2}}{|x-y|^{\frac{N+\mu}{2}\cdot 2_{\mu}^{\ast}}}\bigg\|_{L^{2^*_\mu}}\bigg)^{2\cdot2_{\mu}^{\ast}}.
\end{split}
\end{equation*}
By Minkowski's inequality again, we find that
\begin{equation*}
\begin{split}
\left(\int_{\mathbb{R}^N}\Big(|x-y|^{-\frac{N+\mu}{2}}\ast|u(y)| ^{2^*_\mu}\Big)^2\right)^{\frac{1}{2\cdot2_{\mu}^{\ast}}}&
\leq \left\|\big|x-y\big|^{-\frac{N+\mu}{2}}\ast\big|\frac{u+v}{2}\big| ^{2^*_\mu}\right\|^{\frac{1}{2_{\mu}^{\ast}}}_{L^2}\\&+ \left\|\big|x-y\big|^{-\frac{N+\mu}{2}}\ast\big|\frac{u-v}{2}\big| ^{2^*_\mu}\right\|^{\frac{1}{2_{\mu}^{\ast}}}_{L^2}.
\end{split}
\end{equation*}
Therefore, we get
\begin{equation*}
\begin{split}
\left(\int_{\mathbb{R}^N}\Big(|x|^{-\mu} \ast\Big|\frac{u+v}{2}\Big|^{2^*_\mu}\Big)\Big|\frac{u+v}{2}\Big|^{2^*_\mu} \right)^{\frac{1}{2^*_\mu}}\geq\Big(\big\|u\big\|_{NL}- \Big\|\frac{u-v}{2}\Big\|_{NL}\Big)^2.
\end{split}
\end{equation*}
What's more, by the convexity of the function $f(x)=|x|^2$, $f(x+y)\geq f(x)+f^{\prime}(x)y$, and so
\begin{equation*}
\Big(\big\|u\big\|_{NL}- \Big\|\frac{u-v}{2}\Big\|_{NL}\Big)^2\geq\big\|u\big\|_{NL}^2-2\big\|u\big\|_{NL}\Big\|\frac{u-v}{2}\Big\|_{NL}.
\end{equation*}
These two inequalities imply that
\begin{equation*}
\begin{split}
\Big\|\frac{\nabla u-\nabla v}{2}\Big\|_{L^2}^2&\leq\big\|\nabla u\big\|_{L^2}^2-S_{HL}\big\|u\big\|_{NL}^2+2S_{HL}\big\|u\big\|_{NL}\Big\|\frac{u-v}{2}\Big\|_{NL},
\end{split}
\end{equation*}
and the conclusion follows.

\section{Spectrum of the linear operator}\label{sectevp1}

    For the simplicity of notations, we write $W$ instead of $W[\xi,\lambda]$ defined in (\ref{defU}).
Firstly, we study the following eigenvalue problem
    \begin{equation}\label{Pwhlep}
    \begin{split}
    &-\Delta \omega+\big(|x|^{-\mu} \ast W^{2_\mu^*}\big)W^{2_\mu^*-2}\omega
    =\bar{\nu}\left[\Big(|x|^{-\mu} \ast (W^{2_\mu^*-1}\omega)\Big)W^{2_\mu^*-1}
    +\Big(|x|^{-\mu} \ast W^{2_\mu^*}\Big)W^{2_\mu^*-2}\omega\right],\\&\omega\in \mathcal{D}^{1,2}(\mathbb{R}^N).
    \end{split}
    \end{equation}
    By a straightforward computation, for all $\omega\in \mathcal{D}^{1,2}(\mathbb{R}^N)$, we get that
    \begin{equation}\label{pfvlnv}
    \begin{split}
    \int_{\mathbb{R}^N}\Big(|x|^{-\mu} \ast (W^{2_\mu^*-1}\omega)\Big)W^{2_\mu^*-1}\omega
    \leq & C(N,\mu) \|W\|^{2(2_{\mu}^\ast-1)}_{L^{2^*}}\|\omega\|^2_{L^{2^\ast}}
    \leq \|\omega\|^2_{\mathcal{D}^{1,2}(\mathbb{R}^N)}.
    \end{split}
    \end{equation}
Analogously, we have that
    \begin{equation}\label{pfvlnv2}
    \begin{split}
    \int_{\mathbb{R}^N}\Big(|x|^{-\mu} \ast W^{2_\mu^*}\Big)W^{2_\mu^*-2}\omega^2
    \leq  \|\omega\|^2_{\mathcal{D}^{1,2}(\mathbb{R}^N)}.
    \end{split}
    \end{equation}
Following the work of Servadei and Valdinoci \cite{SV13} for non-local operators, the eigenvalues of problem (\ref{Pwhlep}) can be defined as follows:

    \begin{definition}\label{defevp}
    The Rayleigh quotient characterization of first eigenvalue implies
    \begin{equation}\label{deffev1}
  \bar{\nu}_1:=\inf_{v\in \mathcal{D}^{1,2}(\mathbb{R}^N)\backslash\{0\}}
    \frac{\int_{\mathbb{R}^N}|\nabla \omega|^2+\int_{\mathbb{R}^N}(|x|^{-\mu} \ast W^{2_\mu^*})\omega^2}
    {\int_{\mathbb{R}^N}(|x|^{-\mu} \ast (W^{2_\mu^*-1}\omega))W^{2_\mu^*-1}\omega+ \int_{\mathbb{R}^N}(|x|^{-\mu} \ast W^{2_\mu^*})\omega^2}
    .
    \end{equation}
In addition, for any $l\in\mathbb{N}$ the eigenvalues can be characterized as follows:
    \begin{equation}\label{deffevk}
   \bar{\nu}_{l+1}:=\inf_{v\in \mathbb{W}_{l+1}\backslash\{0\}}
    \frac{\int_{\mathbb{R}^N}|\nabla \omega|^2+\int_{\mathbb{R}^N}(|x|^{-\mu} \ast W^{2_\mu^*})\omega^2}
    {\int_{\mathbb{R}^N}(|x|^{-\mu} \ast (W^{2_\mu^*-1}\omega))W^{2_\mu^*-1}\omega+ \int_{\mathbb{R}^N}(|x|^{-\mu} \ast W^{2_\mu^*-1})\omega^2}
    ,
    \end{equation}
    where
    \begin{equation}\label{defczs}
    \mathbb{W}_{l+1}:=\left\{\omega\in \mathcal{D}^{1,2}(\mathbb{R}^N): \int_{\mathbb{R}^N}\nabla \omega \cdot \nabla v_j=0,\quad \mbox{for all}\quad j=1,\ldots,l.\right\},
    \end{equation}
    and $v_j$ is the corresponding eigenfunction to $\bar{\nu}_j$.
    \end{definition}

    Choosing $\omega=W$ in (\ref{deffev1}), then we get that
    \begin{equation*}
    \begin{split}
   \bar{\nu}_1
    \leq &  \frac{\int_{\mathbb{R}^N}|\nabla W|^2 +\int_{\mathbb{R}^N}(|x|^{-\mu} \ast W^{2_\mu^*})W^{2_\mu^*}}
    {\int_{\mathbb{R}^N}(|x|^{-\mu} \ast W^{2_\mu^*})W^{2_\mu^*}
    + \int_{\mathbb{R}^N}(|x|^{-\mu} \ast W^{2_\mu^*})W^{2_\mu^*}}=1.
    \end{split}
    \end{equation*}
    Moreover, in view of (\ref{pfvlnv2}), we have that
    \begin{equation*}
    \begin{split}
    \int_{\mathbb{R}^N}(|x|^{-\mu} \ast (W^{2_\mu^*-1}\omega))W^{2_\mu^*-1}\omega
    \leq \|\omega\|^2_{\mathcal{D}^{1,2}(\mathbb{R}^N)},\quad\text{for all}\quad\omega\in \mathcal{D}^{1,2}(\mathbb{R}^N),
    \end{split}
    \end{equation*}
    and the equality holds if and only if $\omega=\alpha W$ with $\alpha\in\mathbb{R}$, which means that $\bar{\nu}_1\geq 1$. Then we obtain that $\bar{\nu}_1=1$ and the corresponding eigenfunction is $\alpha W$ with $\alpha\in\mathbb{R}$, that is $\omega_1=W$.
     Therefore, from Proposition \ref{prondgr} involving the nondegeneracy of $W$, we have the following discrete spectral information of operator $\mathcal{L}[u]$.

    \begin{proposition}\label{propep}
    Let $\bar{\nu}_j$, $j=1,2,\ldots,$ denote the eigenvalues of (\ref{Pwhlep}) in increasing order as in Definition \ref{defevp}. Then operator $\mathcal{L}[u]$ from \eqref{defanndg} has a discrete spectrum $\{\bar{\nu}_j\}_{1}^{\infty}$, with $0<\bar{\nu}_j<\bar{\nu}_{j+1}$ for all $j$, and
     \begin{equation*}
     \begin{split}
     &\bar{\nu}_1=1, \quad \quad X_1=span\left\{W\right\},\\
     \bar{\nu}_2&=2_\mu^*, \quad \quad X_2=span
    \left\{\partial x_1W,\ldots,\partial x_NW,\quad x\cdot\nabla W+\frac{N-2}{2}W\right\},
    \end{split}
     \end{equation*}
     where $X_j$ denotes the eigenfunction space corresponding to $ \bar{\nu}_j$.  Furthermore, $T_w\widetilde{M}=span\left\{X_1\cup X_2\right\}$.
    \end{proposition}

Let $\rho=u-\sigma$ where $\sigma=\sum\limits_{i=1}^{\kappa}\alpha_{i}W[\xi_i,\lambda_i]$ be the linear combination of Talenti bubbles that is closet to $u$ in the $\mathcal{D}^{1,2}$-norm, such that,
    \begin{equation*}
    \big\|\nabla u-\nabla\sigma\big\|_{L^2}=\min_{\substack{\bar{\alpha}_1,\dots,\bar{\alpha}_\kappa\in\mathbb{R},\\ \bar{\xi}_1,\dots,\bar{\xi}_\kappa\in\mathbb{R}^N,\\ \bar{\lambda}_1,\dots,\bar{\lambda}_\kappa\in\mathbb{R}}}\bigg\|\nabla u-\nabla\bigg(\sum_{i=1}^{\kappa}\bar{\alpha}_iW[\bar{\xi}_i,\bar{\lambda}_{i}]\bigg)\bigg\|_{L^2}.
    \end{equation*}
Here we note that the family together with the coefficients $\alpha_1,\cdots,\alpha_{\kappa}\in\mathbb{R}$ is $\delta$-interacting if (\ref{interacting}) holds, and we have that
\begin{equation*}
\max\limits_{1\leq i\leq\kappa}|\alpha_i-1|\leq\delta.
\end{equation*}
Moreover, for any $1\leq i\leq\kappa$, $\rho$ also satisfies the following orthogonality conditions:
\begin{equation*}
\int_{\mathbb{R}^N}\nabla\rho\nabla W[\xi_i,\lambda_i]=0,\quad\int_{\mathbb{R}^N}\nabla\rho\nabla \frac{\partial W[\xi_i,\lambda_i]}{\partial\lambda}=0,\quad\int_{\mathbb{R}^N}\nabla\rho\nabla \frac{\partial W[\xi_i,\lambda_i]}{\partial\xi_i}=0 \quad \mbox{for any}\quad 1\leq i\leq N.
\end{equation*}
In the sequel, we write $W_i$ instead of $W[\xi_i,\lambda_i]$ for simplicity. In virtue of the eigenvalue problem, we know that the functions $W_i$, $\frac{\partial W_i}{\partial\lambda}$ and $\frac{\partial W_i}{\partial\xi_i}$ are eigenfunctions for the eigenvalue problem
\begin{eqnarray}\label{defanndg}
\left\{ \arraycolsep=0.0pt
\begin{array}{ll}
& \mathcal{L}[u]=\bar{\nu}\mathcal{R}[u],\quad\mbox{in}\quad\mathbb{R}^N,\\[3mm]
& u\in \mathcal{D}^{1,2}(\mathbb{R}^N),
\end{array}
\right.
\end{eqnarray}
where we denote
\begin{equation*}
\mathcal{L}[u]:=-\Delta u+\Big(|x|^{-\mu} \ast W_i^{2_\mu^*}\Big)W_i^{2_\mu^*-2}u
\end{equation*}
and
\begin{equation*}
\mathcal{R}[u]:=\Big(|x|^{-\mu} \ast \big(W_i^{2_\mu^*-1}u\big)\Big)W_i^{2_\mu^*-1}
 +\Big(|x|^{-\mu} \ast W_i^{2_\mu^*}\Big)W_i^{2_\mu^*-2}u.
\end{equation*}
Then the orthogonal conditions are equivalent to the following
\begin{equation}\label{EP1}
\int_{\mathbb{R}^N}\Big(|x|^{-\mu} \ast \big(W_i^{2_\mu^*-1}\rho\big)\Big)W_i^{2_\mu^*}=0,
\end{equation}
\begin{equation}\label{EP2}
   \begin{split}
&2_\mu^*\int_{\mathbb{R}^N}\Big(|x|^{-\mu} \ast \big(W_{i}^{2_\mu^*-1}\rho\big)\Big)W_i^{2_\mu^*-1}\frac{\partial W_i}{\partial\lambda}
+(2_\mu^*-1)\int_{\mathbb{R}^N}\Big(|x|^{-\mu} \ast W_{i}^{2_\mu^*}\Big)W_{i}^{2_\mu^*-2}\frac{\partial W_{i}}{\partial\lambda}\rho=0
   \end{split}
\end{equation}
and
\begin{equation}\label{EP3}
2_\mu^*\int_{\mathbb{R}^N}\Big(|x|^{-\mu} \ast \big(W_i^{2_\mu^*-1}\rho\big)\Big)W_i^{2_\mu^*-1}\frac{\partial W_i}{\partial\xi_i}+(2_\mu^*-1)\int_{\mathbb{R}^N}\Big(|x|^{-\mu} \ast W_i^{2_\mu^*}\Big)W_i^{2_\mu^*-2}\frac{\partial W_i}{\partial x_i}\rho=0,
\end{equation}
for any $1\leq i\leq N$.

The following inequality  is an immediate consequence of Proposition \ref{propep}, which will be used in the proof of Proposition \ref{estimate1}.
\begin{proposition}\label{pro-4.1}
For every $\varphi\in (\mathcal{\widetilde{Z}})^{\perp}$ is almost orthogonal to the functions $W_i$, $\frac{\partial W_i}{\partial\lambda}$ and $\frac{\partial W_i}{\partial\xi_i}$, we have
\begin{equation*}
\begin{split}
\int_{\mathbb{R}^N}\Big[\Big(|x|^{-\mu} \ast \big(W_i^{2_\mu^*-1}\varphi\big)\Big)W_i^{2_\mu^*-1}\varphi
 &+\Big(|x|^{-\mu} \ast W_i^{2_\mu^*}\Big)W_i^{2_\mu^*-2}\varphi^2\Big]
 \\&\leq\frac{1}{\hbar}\Big[\int_{\mathbb{R}^N}|\nabla\varphi|^2+\int_{\mathbb{R}^N}\Big(|x|^{-\mu} \ast W_i^{2_\mu^*}\Big)W_i^{2_\mu^*-2}\varphi^2\Big],
\end{split}
\end{equation*}
where $\hbar>2_{\mu}^{*}$ is the largest eigenvalue of $\frac{\mathcal{L}[u]}{\mathcal{R}[u]}$ and the set $\mathcal{\widetilde{Z}}$ is the eigenfunction space associated to eigenvalues of $\frac{\mathcal{L}[u]}{\mathcal{R}[u]}$.
\end{proposition}

\section{Stability through Euler-Lagrange equation}\label{sectrt}

In this section we prove Theorems \ref{F4} and \ref{Figalli}.

\subsection{ A stability of profile decompositions}
Inspired by the well-known results of profile decompositions to (\ref{bec}) by Struwe in \cite{Struwe-1984}, we are in a position to prove a profile decomposition of the nonlocal Hartree type equation \eqref{ele} for nonnegative functions.

\textbf{Proof the Theorem~\ref{F4}}.
We define the concentration function by
	$$
	Q_{m}(r):=\sup_{y\in\mathbb{R}^N}\int_{B_{r}(y)}|w_{m}(x)|^{2}.
	$$
Since
$$(\kappa-\frac{1}{2})S_{HL}^{\frac{2N-\mu}{N+2-\mu}}\leq\|w_m|_{\mathcal{D}^{1,2}(\mathbb{R}^N)}^2\leq(\kappa+\frac{1}{2})S_{HL}^{\frac{2N-\mu}{N+2-\mu}}$$
for some $\kappa\in\mathbb{N}$, we choose $r_m>0$, and $y_m\in\mathbb{R}^N$ such that
$$
	Q_{m}(r)=\int_{B_{r_m}(y_m)}|w_{m}|^{2}=\frac{1}{2L_{R_{\epsilon}}}S_{HL}^{\frac{2N-\mu}{N+2-\mu}},
	$$
where $L_{R_{\epsilon}}$ is the number such that the ball $B_{2R_{\varepsilon}}(0)$ is covered by $L_{R_{\epsilon}}$ balls with radius $R_{\epsilon}$.
Let
$$w_m\mapsto \tilde{w}_m(x)=\Big(\frac{r_n}{R_{\epsilon}}\Big)^{(N-2)/2}w_m\Big(\frac{r_n}{R_\epsilon}x\Big)$$
such that
\begin{equation}\label{Q-1}
\widetilde{Q}_{m}(R_{\epsilon}):=\sup_{y\in\mathbb{R}^N}\int_{B_{R_{\epsilon}}(y)}|\nabla\tilde{w}_{m}|^{2}=\int_{B_{R_\epsilon}(r_m^{-1}R_{\epsilon}y_m)}|\nabla\tilde{w}_{m}|^{2}=\frac{1}{2L_{R_{\epsilon}}}S_{HL}^{\frac{2N-\mu}{N+2-\mu}}.
\end{equation}
By invariance of the $\mathcal{D}^{1,2}(\mathbb{R}^N)$ norms under translation and dilation, we get that
$$\|\tilde{w}_m\|_{\mathcal{D}^{1,2}(\mathbb{R}^N)}=\|w_m\|_{\mathcal{D}^{1,2}(\mathbb{R}^N)}$$
and
\begin{equation*} \int_{\mathbb{R}^N}\int_{\mathbb{R}^N}\frac{|\tilde{w}_{m}(x)|^{2_{\mu}^{\ast}}|\tilde{w}_{m}(y)|^{2_{\mu}^{\ast}}}
	{|x-y|^{\mu}}=\int_{\mathbb{R}^N}\int_{\mathbb{R}^N}\frac{|w_{m}(x)|^{2_{\mu}^{\ast}}
		|w_{m}(y)|^{2_{\mu}^{\ast}}}
	{|x-y|^{\mu}}.
\end{equation*}
Hence we may assume that there exists $w^0\in \mathcal{D}^{1,2}(\mathbb{R}^N)$ such that
	$$
\tilde{w}_m\rightharpoonup w^0,~~~\text{weakly in}~\in \mathcal{D}^{1,2}(\mathbb{R}^N) \quad\text{as}~~m\rightarrow\infty.
	$$
For any $y\in\mathbb{R}^N$, let $\psi_{m}$ be given by
\begin{equation*}
		\psi_{m}=(\tilde{w}_m-w^{0})\tilde{\xi},
\end{equation*}
where $\tilde{\xi}\in \mathcal{C}_{c}^{\infty}$ with $\tilde{\xi}\equiv1$ for $x\in B_{R_{\epsilon}}(y)$ and $\tilde{\xi}\equiv0$ in $B_{2R_{\epsilon}}^{c}(y)$. Then we have
\begin{equation*}
\|\psi_m\|_{\mathcal{D}^{1,2}(\mathbb{R}^N)}^2\leq C\|\tilde{w}_m- w^0\|_{\mathcal{D}^{1,2}(\mathbb{R}^N)}^2+C\int_{B_{2R_{\epsilon}}(y)\setminus B_{R_{\epsilon}}(y)}|\tilde{w}_m- w^0|^2\leq C
\end{equation*}
and
\begin{equation*}
\int_{B_{2R_{\epsilon}}(y)\setminus B_{R_{\epsilon}}(y)}|\tilde{w}_m-
 w^0|^2\rightarrow0\quad\text{as}~~m\rightarrow\infty.
\end{equation*}
 Combining $DE_0(w_m; \mathbb{R}^N)\rightarrow0$ in $(\mathcal{D}^{1,2}(\mathbb{R}^N))^{-1}$ as $m\rightarrow\infty$, $DE_0(w^0; \mathbb{R}^N)=0$, the Brezis-Lieb lemma and HLS inequality we deduce that
\begin{equation}\label{Q-2}
\begin{split}
o_m(1)&=\left\langle\psi_m\xi, DE_0(\tilde{w}_m; \mathbb{R}^N)-DE_0(w^0; \mathbb{R}^N) \right\rangle\\&
=\int_{\mathbb{R}^N}\nabla(\tilde{w}_m-w^0)\nabla(\psi_{m}\tilde{\xi})-\int_{\mathbb{R}^N}\int_{\mathbb{R}^N}\Big[\frac{|\tilde{w}_m(y)|^{2_{\mu}^{\ast}}
		|\tilde{w}_m(x)|^{2_{\mu}^{\ast}-1}\psi_{m}\tilde{\xi}} {|x-y|^{\mu}}-\frac{|w^0(y)|^{2_{\mu}^{\ast}}
		|w^0(x)|^{2_{\mu}^{\ast}-1}\psi_{m}\tilde{\xi}} {|x-y|^{\mu}}\Big]\\&
=\int_{\mathbb{R}^N}|\nabla\psi_{m}|^2-\int_{\mathbb{R}^N}\int_{\mathbb{R}^N}\frac{|(\tilde{w}_m-w^0)(y)|^{2_{\mu}^{\ast}}|(\tilde{w}_m-w^0)(x)|^{2_{\mu}^{\ast}-2}|\psi_{m}|^2}
	{|x-y|^{\mu}}+o_m(1)\\&
=\|\psi_{m}\|_{\mathcal{D}^{1,2}(\mathbb{R}^N)}^2-\int_{B_{R_{\epsilon}}}\int_{B_{R_{\epsilon}}}\frac{|(\tilde{w}_m-w^0)(y)|^{2_{\mu}^{\ast}}|(\tilde{w}_m-w^0)(x)|^{2_{\mu}^{\ast}-2}|\psi_{m}|^2}
	{|x-y|^{\mu}}+o_m(1)\\&
\geq\|\psi_{m}\|_{\mathcal{D}^{1,2}(\mathbb{R}^N)}^2-\int_{\mathbb{R}^N}\int_{\mathbb{R}^N}\frac{|\psi_{m}(y)|^{2_{\mu}^{\ast}}|\psi_{m}(x)|^{2_{\mu}^{\ast}}}
	{|x-y|^{\mu}}+o_m(1)
\end{split}
\end{equation}
as $m\rightarrow\infty$. Moreover,  HLS and Sobolev inequalities give that
\begin{equation}\label{Q-3} \|\psi_{m}\|_{\mathcal{D}^{1,2}(\mathbb{R}^N)}^2-\int_{\mathbb{R}^N}\int_{\mathbb{R}^N}\frac{|\psi_{m}(y)|^{2_{\mu}^{\ast}}|\psi_{m}(x)|^{2_{\mu}^{\ast}}}
	{|x-y|^{\mu}} \geq\big\|\psi_{m}\big\|_{\mathcal{D}^{1,2}(\mathbb{R}^N)}^{2}\Big(1-S_{HL}^{-\frac{2N-\mu}{N-2}}\big\|\psi_{m}\big\|_{\mathcal{D}^{1,2}(\mathbb{R}^N)}^{\frac{2(N-\mu+2)}{N-2}}\Big).
\end{equation}
On the other hand, we note that
$$
\int_{\mathbb{R}^N}|\nabla\psi_{m}|^2=\int_{B_{R_{\epsilon}}}|\nabla(\tilde{w}_m-w^{0})|^2+o_{m}(1)\leq \int_{B_{2R_{\epsilon}}}|\nabla\tilde{w}_m|^2.
$$
It follows from (\ref{Q-1}), (\ref{Q-2}) and (\ref{Q-3}) that $\psi_{m}\rightarrow0$ in $\mathcal{D}^{1,2}(\mathbb{R}^N)$, that is, $\tilde{w}_m\rightarrow w^0$ in $\mathcal{D}^{1,2}(B_{R_{0}}(y))$ for any $y\in\mathbb{R}^N$. Hence by standard covering argument, we obtain
$\tilde{w}_m\rightarrow w^{0}$ in $\mathcal{D}^{1,2}_{loc}(\mathbb{R}^N)$ as $m\rightarrow\infty$. Clearly, we have that
$$
\int_{B_{R_{\epsilon}}(x)}|w^{0}|^{2}=\frac{1}{2L_{R_{\epsilon}}}S_{HL}^{\frac{2N-\mu}{N+2-\mu}}>0,
$$
hence $w^{0}\not\equiv0$.
Since $w^{0}\geq0$, we have $w^{0}(x)=W[\xi,\lambda](x)$, which means that
 $$w_m\rightharpoonup \Big(\frac{r_m}{R_\epsilon}\Big)^{\frac{N-2}{2}}W[\xi,\lambda]\Big(\frac{r_m}{R_\epsilon}x\Big)$$  in~$\in \mathcal{D}^{1,2}(\mathbb{R}^N)$ as $m\rightarrow\infty$.
Hence for some $r_{m,1}>$ and $\epsilon_1>0$, we have that
\begin{equation*}
w_m\rightharpoonup \Big(\frac{r_m}{R_\epsilon}\Big)^{\frac{N-2}{2}}W[\xi,\lambda]\Big(\frac{r_m}{R_\epsilon}x\Big)+\Big(\frac{r_mr_{m,1}}{R_\epsilon R_{\epsilon_1}}\Big)^{\frac{N-2}{2}}W[\xi,\lambda]\Big(\frac{r_mr_{m,1}}{R_{\epsilon_1}R_\epsilon}x\Big),
\end{equation*}
weakly in~$\in \mathcal{D}^{1,2}(\mathbb{R}^N)$ as $m\rightarrow\infty$.
 In view of $W$ is the unique nonnegative solution of \eqref{ele} and iterating the above process $\kappa$ times, then the result easily follows.

\qed
\subsection{A quantitative version of the stability}
The proof of Theorem~\ref{Figalli} relies on the following several Propositions. The proof of Propositions \ref{estimate1} and \ref{estimate2} will be postponed to Section \ref{section5}.
\begin{proposition}\label{FPU1} Let $N\geq3$, $W[\xi_i,\lambda_i]$ and $W[\xi_j,\lambda_j]$ be two bubbles. Then, for any fixed $\varepsilon>0$ and any nonnegative exponents such that $p+q=2^\ast$, it holds that
\begin{equation*}
\int_{\mathbb{R}^N}W[\xi_i,\lambda_i]^{p}W[\xi_j,\lambda_j]^q\approx
\begin{cases}
Q^{\min(p,q)},\quad\mbox{if}\quad|p-q|\geq\varepsilon,\\
Q^{\frac{N}{N-2}}\log(\frac{1}{Q}),\quad\mbox{if}\quad p=q,
\end{cases}
\end{equation*}
where the quantity
\begin{equation*}
		Q:=Q(\xi_i,\xi_j,\lambda_i,\lambda_j)=\min\Big(\frac{\lambda_i}{\lambda_j}+\frac{\lambda_j}{\lambda_i}+\lambda_i\lambda_j|\xi_i-\xi_j|^2\Big)^{-\frac{N-2}{2}}.
	\end{equation*}
\end{proposition}
\begin{proof}
It is similar to that of Proposition B.2 in \cite{FG20}, so is omitted.
\end{proof}
\begin{proposition}\label{estimate1}
Let $N\geq3$ and $\kappa\in\mathbb{N}$. There exists a positive constant $\delta=\delta(N,\kappa)>0$ such that if $\sigma=\sum\limits_{i=1}^{\kappa}\alpha_iW_i$ is a linear combination of $\delta$-interacting Talenti bubbles and $\rho\in\mathcal{D}^{1,2}(\mathbb{R}^N)$ satisfies \eqref{EP1}, \eqref{EP2} and \eqref{EP3}.
Then we have that
\begin{equation*}
\begin{split}
(2_{\mu}^{\ast}-1)\sum\limits_{i=1}^ {\kappa}\big|\alpha_i\big|^{2_{\mu}^{\ast}}&\int_{\mathbb{R}^N}\Big(\frac{1}{|x|^{\mu}}\ast W_i^{2_{\mu}^{\ast}}\Big)\sigma^{2_{\mu}^{\ast}-2}\rho^2\\&
+2_{\mu}^{\ast}\sum\limits_{i=1}^{\kappa}\big|\alpha_i\big|^{2_{\mu}^{\ast}-1}\int_{\mathbb{R}^N}\Big(\frac{1}{|x|^{\mu}}\ast\big(\sigma^{2_{\mu}^{\ast}-1}\rho\big)\Big)W_i^{2_{\mu}^{\ast}-1}\rho
\leq\nu\int_{\mathbb{R}^N}|\nabla\rho|^2,
\end{split}
\end{equation*}
where $\nu$ is a constant strictly less than $1$ depending only on $N$ and $\kappa$.
\end{proposition}
\begin{proposition}\label{estimate2}
Let $N\geq3$ and $\kappa\in\mathbb{N}$. For any $\epsilon>0$ there exists $\delta=\delta(N,\kappa,\epsilon)>0$ such that the following statement holds. Let $u=\sum\limits_{i=1}^\kappa\alpha_iW_i+\rho$, where the family $(\alpha_i,W_i)_{1\leq i\leq\kappa}$ is $\delta$-interacting, and $\rho$ satisfies both the orthogonality conditions satisfies \eqref{EP1}, \eqref{EP2}, \eqref{EP3} and the bound $\|\nabla\rho\|_{L^2}\leq1$. Then, for any $1\leq i\leq\kappa$, it holds that
\begin{equation}\label{www}
|\alpha_i-1|\lesssim \epsilon\big\|\nabla\rho\big\|_{L^2}+\Big\|\Delta u+\Big(\frac{1}{|x|^{\mu}}\ast |u|^{2_{\mu}^{\ast}}\Big)|u|^{2_{\mu}^{\ast}-2}u\Big\|_{(\mathcal{D}^{1,2}(\mathbb{R}^N))^{-1}}+\Big\|\nabla\rho\Big\|_{L^2}^{\min\big(2,\frac{N-\mu+2}{N-2}\big)},
\end{equation}
and for any pair of indices $i\neq j$ it holds
\begin{equation*}
\int_{\mathbb{R}^N}\Big(\frac{1}{|x|^{\mu}}\ast W_i^{2_{\mu}^{\ast}}\Big)W_i^{2_{\mu}^{\ast}-1}W_j\lesssim \epsilon\big\|\nabla\rho\big\|_{L^2}+\Big\|\Delta u+\Big(\frac{1}{|x|^{\mu}}\ast |u|^{2_{\mu}^{\ast}}\Big)|u|^{2_{\mu}^{\ast}-2}u\Big\|_{(\mathcal{D}^{1,2}(\mathbb{R}^N))^{-1}}+\Big\|\nabla\rho\Big\|_{L^2}^{\min\big(2,\frac{N-\mu+2}{N-2}\big)}.
\end{equation*}
\end{proposition}

Now, we are ready to prove the main results.

\subsection{Proof of Theorem \ref{Figalli}.} We first assume Propositions \ref{estimate1} and \ref{estimate2} and prove Theorem \ref{Figalli}.
The Propositions \ref{estimate2} tells us that the optimal coefficients of the linear combination of Talenti bubbles are approximate to $1$. Let $$\mathcal{\widetilde{M}}_{\kappa}=\Big\{\sum_{i=1}^{\kappa}\widetilde{W}[\xi_i,\lambda_i]:  \xi_i\in\mathbb{R}^N, \lambda_i>0\Big\}.$$
The fact that $dist_{\mathcal{D}^{1,2}}\big(u,\mathcal{\widetilde{M}}_{\kappa}\big)\leq\delta$, which yields $\|\rho\|_{L^2}\leq\delta$.
Furthermore, since $(\widetilde{W}[\xi_i,\lambda_i])_{1\leq i\leq\kappa}$ is a $\delta$-interacting family of Talenti bubbles, then we can take $\tilde{\delta}\leq\delta$ so that $(W[\xi_i,\lambda_i])_{1\leq i\leq\kappa}$ is a $\tilde{\delta}$-interacting family of Talenti bubbles.

Now we evaluate the left side of (\ref{ele}) in the $(\mathcal{D}^{1,2}(\mathbb{R}^N))^{-1}$-norm should control the $\mathcal{D}^{1,2}(\mathbb{R}^N)$-distance of $u$ from the manifold of sums of Talenti bubbles. By the orthogonality condition, we have
    \begin{equation}\label{Sta1}
    \begin{split}
  \int_{\mathbb{R}^N}|\nabla \rho|^2&=\int_{\mathbb{R}^N}\nabla u\nabla \rho=  \int_{\mathbb{R}^N}\left(\frac{1}{|x|^{\mu}}\ast |u|^{2_{\mu}^{\ast}}\right)|u|^{2_{\mu}^{\ast}-2}u\rho- \int_{\mathbb{R}^N}\Big[\Delta u+\left(\frac{1}{|x|^{\mu}}\ast |u|^{2_{\mu}^{\ast}}\right)|u|^{2_{\mu}^{\ast}-2}u\Big]\rho\\&
  \leq\int_{\mathbb{R}^N}\left(\frac{1}{|x|^{\mu}}\ast |u|^{2_{\mu}^{\ast}}\right)|u|^{2_{\mu}^{\ast}-2}u\rho+\big\|\nabla\rho\big\|_{L^2}\Big\|\Delta u+\left(\frac{1}{|x|^{\mu}}\ast |u|^{2_{\mu}^{\ast}}\right)|u|^{2_{\mu}^{\ast}-2}u\Big\|_{(\mathcal{D}^{1,2}(\mathbb{R}^N))^{-1}}.
    \end{split}
    \end{equation}
Calculating the first term on the right-hand side, we know
\begin{equation*}
\begin{split}
\int_{\mathbb{R}^N}\left(\frac{1}{|x|^{\mu}}\ast |u|^{2_{\mu}^{\ast}}\right)&|u|^{2_{\mu}^{\ast}-2}u\rho\leq\int_{\mathbb{R}^N}\left(\frac{1}{|x|^{\mu}}\ast \Big||u|^{2_{\mu}^{\ast}}-|\sigma|^{2_{\mu}^{\ast}}\Big| \right)\Big||u|^{2_{\mu}^{\ast}-2}u-|\sigma|^{2_{\mu}^{\ast}-2}\sigma\Big|\rho
\\+&
\int_{\mathbb{R}^N}\left(\frac{1}{|x|^{\mu}}\ast \Big||u|^{2_{\mu}^{\ast}}-|\sigma|^{2_{\mu}^{\ast}}\Big| \right)\bigg|\Big|\sum\limits_{i=1}^{\kappa}\alpha_iW_i\Big|^{2_{\mu}^{\ast}-2}\sum\limits_{i=1}^{\kappa}\alpha_iW_i
-\sum\limits_{i=1}^{\kappa}\big|\alpha_i\big|^{2_{\mu}^{\ast}-2}\alpha_i\big|W_i\big|^{2_{\mu}^{\ast}-2}W_i\bigg|\rho
\\+&
\int_{\mathbb{R}^N}\left(\frac{1}{|x|^{\mu}}\ast \bigg|\Big|\sum\limits_{i=1}^{\kappa}\alpha_iW_i\Big|^{2_{\mu}^{\ast}}
-\sum\limits_{i=1}^{\kappa}\big|\alpha_i\big|^{2_{\mu}^{\ast}}\big|W_i\big|^{2_{\mu}^{\ast}}\bigg| \right)\Big||u|^{2_{\mu}^{\ast}-2}u-|\sigma|^{2_{\mu}^{\ast}-2}\sigma\Big|\rho
\\+&
\int_{\mathbb{R}^N}\left(\frac{1}{|x|^{\mu}}\ast \bigg|\Big|\sum\limits_{i=1}^{\kappa}\alpha_iW_i\Big|^{2_{\mu}^{\ast}}
-\sum\limits_{i=1}^{\kappa}\big|\alpha_i\big|^{2_{\mu}^{\ast}}\big|W_i\big|^{2_{\mu}^{\ast}}\bigg| \right)\bigg|\Big|\sum\limits_{i=1}^{\kappa}\alpha_iW_i\Big|^{2_{\mu}^{\ast}-2}\sum\limits_{i=1}^{\kappa}\alpha_iW_i
\\-&
\sum\limits_{i=1}^{\kappa}\big|\alpha_i\big|^{2_{\mu}^{\ast}-2}\alpha_i\big|W_i\big|^{2_{\mu}^{\ast}-2}W_i\bigg|\rho
+\sum\limits_{i=1}^{\kappa}\int_{\mathbb{R}^N}\left(\frac{1}{|x|^{\mu}}\ast \Big||u|^{2_{\mu}^{\ast}}-|\sigma|^{2_{\mu}^{\ast}}\Big| \right)\big|\alpha_i\big|^{2_{\mu}^{\ast}-2}\alpha_i\big|W_i\big|^{2_{\mu}^{\ast}-2}W_i\rho
\\+&
\sum\limits_{i=1}^{\kappa}\int_{\mathbb{R}^N}\left(\frac{1}{|x|^{\mu}}\ast \bigg|\Big|\sum\limits_{i=1}^{\kappa}\alpha_iW_i\Big|^{2_{\mu}^{\ast}}
-\sum\limits_{i=1}^{\kappa}\big|\alpha_i\big|^{2_{\mu}^{\ast}}\big|W_i\big|^{2_{\mu}^{\ast}}\bigg| \right)\big|\alpha_i\big|^{2_{\mu}^{\ast}-2}\alpha_i\big|W_i\big|^{2_{\mu}^{\ast}-2}W_i\rho
\\+&
\sum\limits_{i=1}^{\kappa}\int_{\mathbb{R}^N}\left(\frac{1}{|x|^{\mu}}\ast \Big|\alpha_iW_i\Big|^{2_{\mu}^{\ast}} \right)\Big||u|^{2_{\mu}^{\ast}-2}u-|\sigma|^{2_{\mu}^{\ast}-2}\sigma\Big|\rho
\\+&
\sum\limits_{i=1}^{\kappa}\int_{\mathbb{R}^N}\left(\frac{1}{|x|^{\mu}}\ast \Big|\alpha_iW_i\Big|^{2_{\mu}^{\ast}} \right)\bigg|\Big|\sum\limits_{i=1}^{\kappa}\alpha_iW_i\Big|^{2_{\mu}^{\ast}-2}\sum\limits_{i=1}^{\kappa}\alpha_iW_i
-\sum\limits_{i=1}^{\kappa}\big|\alpha_i\big|^{2_{\mu}^{\ast}-2}\alpha_i\big|W_i\big|^{2_{\mu}^{\ast}-2}W_i\bigg|\rho
\\+&
\sum\limits_{i=1}^{\kappa}\int_{\mathbb{R}^N}\left(\frac{1}{|x|^{\mu}}\ast \Big|\sum\limits_{i=1}^{\kappa}\alpha_iW_i\Big|^{2_{\mu}^{\ast}} \right)\big|\alpha_i\big|^{2_{\mu}^{\ast}-2}\alpha_i\big|W_i\big|^{2_{\mu}^{\ast}-2}W_i\rho.
\end{split}
\end{equation*}
We divide our argument into several cases.\\		
$\mathbf{Case\ 1.}$
 The exponents $\mu$ and $N$ satisfy
 \begin{equation*}
\left\lbrace
\begin{aligned}
N&=3,\hspace{2mm}\mu\in(0,N),\\
N&=4, \hspace{2mm}\mu\in(0,2],\\
N&\geq5,\hspace{2mm}\mu\in(0,N)\hspace{2mm}\mbox{and}\hspace{2mm}\mu\in(0,1].
\end{aligned}
\right.
\end{equation*}
Thus we have that
\begin{equation}\label{pmr1}
\Big||u|^{2_{\mu}^{\ast}}-|\sigma|^{2_{\mu}^{\ast}}\Big|\leq 2_{\mu}^{\ast}|\sigma|^{2_{\mu}^{\ast}-1}|\rho|+\widetilde{C}_1\big(|\sigma|^{2_{\mu}^{\ast}-2}|\rho|^2+|\rho|^{2_{\mu}^{\ast}}\big), \quad\text{for}~~2_{\mu}^\ast\geq2,
\end{equation}
and
\begin{equation}\label{pmr2}
\Big||u|^{2_{\mu}^{\ast}-2}u-|\sigma|^{2_{\mu}^{\ast}-2}\sigma\Big|\leq (2_{\mu}^{\ast}-1)|\sigma|^{2_{\mu}^{\ast}-2}|\rho|+\widetilde{C}_2\big(|\sigma|^{2_{\mu}^{\ast}-3}|\rho|^2+|\rho|^{2_{\mu}^{\ast}-1}\big),\quad\text{for}~~2_{\mu}^\ast-1\geq2.
\end{equation}
Combining the elementary inequalities
\begin{equation}\label{pmr3}
\bigg|\Big|\sum\limits_{i=1}^{\kappa}\alpha_iW_i\Big|^{2_{\mu}^{\ast}}
-\sum\limits_{i=1}^{\kappa}\big|\alpha_i\big|^{2_{\mu}^{\ast}}\big|W_i\big|^{2_{\mu}^{\ast}}\bigg|\lesssim \sum\limits_{1\leq i\neq j\leq\kappa}\big|\alpha_i\big|^{2_{\mu}^{\ast}-1}\alpha_j\big|W_i\big|^{2_{\mu}^{\ast}-1}W_j,
\end{equation}
and
\begin{equation}\label{pmr4}
\bigg|\Big|\sum\limits_{i=1}^{\kappa}\alpha_iW_i\Big|^{2_{\mu}^{\ast}-2}\sum\limits_{i=1}^{\kappa}\alpha_iW_i
-\sum\limits_{i=1}^{\kappa}\big|\alpha_i\big|^{2_{\mu}^{\ast}-2}\alpha_i\big|W_i\big|^{2_{\mu}^{\ast}-2}W_i\bigg|\lesssim \sum\limits_{1\leq i\neq j\leq\kappa}\big|\alpha_i\big|^{2_{\mu}^{\ast}-2}\alpha_j\big|W_i\big|^{2_{\mu}^{\ast}-2}W_j,
\end{equation}
then we can evaluate separately the following various terms by using the basic inequalities.  We first have that
\begin{equation*}
\aligned
\int_{\mathbb{R}^N}\Big(\frac{1}{|x|^{\mu}}\ast\big(|\sigma|^{2_{\mu}^{\ast}-1}|\rho|\big)\Big)|\sigma|^{2_{\mu}^{\ast}-2}|\rho|^2&\leq
\big\||\sigma|^{2_{\mu}^{\ast}-1}|\rho|\big\|_{L^{\frac{2N}{2N-\mu}}}\big\||\sigma|^{2_{\mu}^{\ast}-2}|\rho|^2\big\|_{L^{\frac{2N}{2N-\mu}}}\\&
\leq\big\|\sigma \big\|_{L^{2^*}}^{\frac{N-2\mu+6}{N-2}}\big\|\rho\big\|_{L^{2^*}}^3\lesssim \big\|\nabla\rho\big\|_{L^{2}}^3.
\endaligned
\end{equation*}
We similarly compute and get
\begin{equation*}
\aligned
\int_{\mathbb{R}^N}\Big(\frac{1}{|x|^{\mu}}\ast\big(|\sigma|^{2_{\mu}^{\ast}-2}|\rho|^2+|\rho|^{2_{\mu}^{\ast}}\big)\Big)|\sigma|^{2_{\mu}^{\ast}-2}|\rho|^2&\leq
\big\|\sigma \big\|_{L^{2^*}}^{\frac{8-2\mu}{N-2}}\big\|\rho\big\|_{L^{2^*}}^4
+\big\|\sigma \big\|_{L^{2^*}}^{\frac{4-\mu}{N-2}}\big\|\rho\big\|_{L^{2^*}}^\frac{4N-\mu-4}{N-2}\\&\lesssim \big\|\nabla\rho\big\|_{L^{2}}^4+\big\|\nabla\rho\big\|_{L^{2}}^\frac{4N-\mu-4}{N-2}
\endaligned
\end{equation*}
and
\begin{equation*}
\aligned
\int_{\mathbb{R}^N}&\bigg[\frac{1}{|x|^{\mu}}\ast\Big(|\sigma|^{2_{\mu}^{\ast}-1}|\rho|+|\sigma|^{2_{\mu}^{\ast}-2}|\rho|^2+|\rho|^{2_{\mu}^{\ast}}\Big)\bigg]\big(|\sigma|^{2_{\mu}^{\ast}-3}|\rho|^3+|\rho|^{2_{\mu}^{\ast}}\big)\vspace{4mm}\\&\leq
\big\|\sigma \big\|_{L^{2^*}}^{\frac{8-2\mu}{N-2}}\big\|\rho\big\|_{L^{2^*}}^4
+\big\|\sigma \big\|_{L^{2^*}}^{\frac{10-N-2\mu}{N-2}}\big\|\rho\big\|_{L^{2^*}}^5+\big\|\sigma \big\|_{L^{2^*}}^{\frac{N+6-2\mu}{N-2}}\big\|\rho\big\|_{L^{2^*}}^3\vspace{4mm}\\&
+\big\|\sigma \big\|_{L^{2^*}}^{\frac{N-\mu+2}{N-2}}\big\|\rho\big\|_{L^{2^*}}^\frac{3N-\mu-2}{N-2}
+\big\|\sigma \big\|_{L^{2^*}}^{\frac{4-\mu}{N-2}}\big\|\rho\big\|_{L^{2^*}}^\frac{4N-\mu-4}{N-2}+\big\|\rho\big\|_{L^{2^*}}^\frac{4N-2\mu}{N-2}\vspace{4mm}\\&\lesssim \big\|\nabla\rho\big\|_{L^{2}}^3+\big\|\nabla\rho\big\|_{L^{2}}^4+\big\|\nabla\rho\big\|_{L^{2}}^5+\big\|\nabla\rho\big\|_{L^{2}}^\frac{3N-\mu-2}{N-2}
+\big\|\nabla\rho\big\|_{L^{2}}^\frac{4N-\mu-4}{N-2}+\big\|\nabla\rho\big\|_{L^{2}}^\frac{4N-2\mu}{N-2}.
\endaligned
\end{equation*}
By a direct computation, we find that
\begin{equation}\label{pmr5}
\aligned
\int_{\mathbb{R}^N}&\left(\frac{1}{|x|^{\mu}}\ast \Big||u|^{2_{\mu}^{\ast}}-|\sigma|^{2_{\mu}^{\ast}}\Big| \right)\bigg|\Big|\sum\limits_{i=1}^{\kappa}\alpha_iW_i\Big|^{2_{\mu}^{\ast}-2}\sum\limits_{i=1}^{\kappa}\alpha_iW_i
-\sum\limits_{i=1}^{\kappa}\big|\alpha_i\big|^{2_{\mu}^{\ast}-2}\alpha_i\big|W_i\big|^{2_{\mu}^{\ast}-2}W_i\bigg|\rho
\\&\lesssim
\sum\limits_{1\leq i\neq j\leq\kappa}\int_{\mathbb{R}^N}\bigg[\frac{1}{|x|^{\mu}}\ast\Big(|\sigma|^{2_{\mu}^{\ast}-1}|\rho|+|\sigma|^{2_{\mu}^{\ast}-2}|\rho|^2+|\rho|^{2_{\mu}^{\ast}}\Big)\bigg]W_i^{2_{\mu}^{\ast}-2}W_j|\rho|\\&\leq
\sum\limits_{1\leq i\neq j\leq\kappa}\Big(\big\|\sigma \big\|_{L^{2^*}}^{\frac{N-\mu+2}{N-2}}\big\|\rho\big\|_{L^{2^*}}^2
+\big\|\sigma \big\|_{L^{2^*}}^{\frac{4-\mu}{N-2}}\big\|\rho\big\|_{L^{2^*}}^3+\big\|\rho \big\|_{L^{2^*}}^{\frac{3N-\mu-2}{N-2}}\Big)\big\|W_i^{2_{\mu}^{\ast}-2}W_j\big\|_{L^{\frac{2N}{N-\mu+2}}}\\&\lesssim \sum\limits_{1\leq i\neq j\leq\kappa}\Big(\big\|\nabla\rho\big\|_{L^{2}}^2+\big\|\nabla\rho\big\|_{L^{2}}^3+\big\|\nabla\rho\big\|_{L^{2}}^\frac{3N-\mu-2}{N-2}\Big)
\big\|W_i^{2_{\mu}^{\ast}-2}W_j\big\|_{L^{\frac{2N}{N-\mu+2}}}=:(I_1).
\endaligned
\end{equation}
Similar to the calculation of \eqref{pmr5}, we also get that
\begin{equation*}
\aligned
\int_{\mathbb{R}^N}&\left(\frac{1}{|x|^{\mu}}\ast \bigg|\Big|\sum\limits_{i=1}^{\kappa}\alpha_iW_i\Big|^{2_{\mu}^{\ast}}
-\sum\limits_{i=1}^{\kappa}\big|\alpha_i\big|^{2_{\mu}^{\ast}}\big|W_i\big|^{2_{\mu}^{\ast}}\bigg| \right)\Big||u|^{2_{\mu}^{\ast}-2}u-|\sigma|^{2_{\mu}^{\ast}-2}\sigma\Big|\rho\\&
\hspace{6mm}\lesssim
\sum\limits_{1\leq i\neq j\leq\kappa}\int_{\mathbb{R}^N}\bigg[\frac{1}{|x|^{\mu}}\ast\Big(\big|W_i\big|^{2_{\mu}^{\ast}-1}W_j\Big)\bigg]\Big(|\sigma|^{2_{\mu}^{\ast}-2}|\rho|^2+|\sigma|^{2_{\mu}^{\ast}-3}|\rho|^3+|\rho|^{2_{\mu}^{\ast}}\Big)
\\&\hspace{6mm}\lesssim \sum\limits_{1\leq i\neq j\leq\kappa}\Big(\big\|\nabla\rho\big\|_{L^{2}}^2+\big\|\nabla\rho\big\|_{L^{2}}^3+\big\|\nabla\rho\big\|_{L^{2}}^\frac{2N-\mu}{N-2}\Big)
\big\|W_i^{2_{\mu}^{\ast}-1}W_j\big\|_{L^{\frac{2N}{2N-\mu}}}=:(I_2),
\endaligned
\end{equation*}
\begin{equation*}
\aligned
\sum\limits_{i=1}^{\kappa}&\int_{\mathbb{R}^N}\left(\frac{1}{|x|^{\mu}}\ast \bigg|\Big|\sum\limits_{i=1}^{\kappa}\alpha_iW_i\Big|^{2_{\mu}^{\ast}}
-\sum\limits_{i=1}^{\kappa}\big|\alpha_i\big|^{2_{\mu}^{\ast}}\big|W_i\big|^{2_{\mu}^{\ast}}\bigg| \right)\big|\alpha_i\big|^{2_{\mu}^{\ast}-2}\alpha_i\big|W_i\big|^{2_{\mu}^{\ast}-2}W_i\rho\\&
\leq \widetilde{C}_{3}
\sum\limits_{1\leq i\neq j\leq\kappa}\int_{\mathbb{R}^N}\Big[\frac{1}{|x|^{\mu}}\ast\Big(\big|W_i\big|^{2_{\mu}^{\ast}-1}W_j\Big)\Big]\sum\limits_{ i=1}^{\kappa}W_i^{2_{\mu}^{\ast}-1}|\rho|\lesssim \sum\limits_{1\leq i\neq j\leq\kappa}\big\|\nabla\rho\big\|_{L^{2}}
\big\|W_i^{2_{\mu}^{\ast}-1}W_j\big\|_{L^{\frac{2N}{2N-\mu}}}=:(I_3),
\endaligned
\end{equation*}
\begin{equation*}
\aligned
\sum\limits_{i=1}^{\kappa}&\int_{\mathbb{R}^N}\left(\frac{1}{|x|^{\mu}}\ast \Big|\alpha_iW_i\Big|^{2_{\mu}^{\ast}} \right)\bigg|\Big|\sum\limits_{i=1}^{\kappa}\alpha_iW_i\Big|^{2_{\mu}^{\ast}-2}\sum\limits_{i=1}^{\kappa}\alpha_iW_i
-\sum\limits_{i=1}^{\kappa}\big|\alpha_i\big|^{2_{\mu}^{\ast}-2}\alpha_i\big|W_i\big|^{2_{\mu}^{\ast}-2}W_i\bigg|\rho\\&
\leq \widetilde{C}_{4}
\sum\limits_{i=1}^{\kappa}\int_{\mathbb{R}^N}\Big(\frac{1}{|x|^{\mu}}\ast\big|W_i\big|^{2_{\mu}^{\ast}}\Big)\sum\limits_{ 1\leq i\neq j\leq\kappa}W_i^{2_{\mu}^{\ast}-2}W_j|\rho|\lesssim \sum\limits_{1\leq i\neq j\leq\kappa}\big\|\nabla\rho\big\|_{L^{2}}
\big\|W_i^{2_{\mu}^{\ast}-2}W_j\big\|_{L^{\frac{2N}{N-\mu+2}}}=:(I_4),
\endaligned
\end{equation*}
and
\begin{equation*}
\aligned
\int_{\mathbb{R}^N}&\left(\frac{1}{|x|^{\mu}}\ast \bigg|\Big|\sum\limits_{i=1}^{\kappa}\alpha_iW_i\Big|^{2_{\mu}^{\ast}}
-\sum\limits_{i=1}^{\kappa}\big|\alpha_i\big|^{2_{\mu}^{\ast}}\big|W_i\big|^{2_{\mu}^{\ast}}\bigg| \right)\bigg|\Big|\sum\limits_{i=1}^{\kappa}\alpha_iW_i\Big|^{2_{\mu}^{\ast}-2}\sum\limits_{i=1}^{\kappa}\alpha_iW_i-
\sum\limits_{i=1}^{\kappa}\big|\alpha_i\big|^{2_{\mu}^{\ast}-2}\alpha_i\big|W_i\big|^{2_{\mu}^{\ast}-2}W_i\bigg|\rho\\&
\hspace{16mm}\lesssim
\sum\limits_{1\leq i\neq j\leq\kappa}\int_{\mathbb{R}^N}\Big[\frac{1}{|x|^{\mu}}\ast\Big(\big|W_i\big|^{2_{\mu}^{\ast}-1}W_j\Big)\Big]\sum\limits_{1\leq i\neq j\leq\kappa}\big|W_i\big|^{2_{\mu}^{\ast}-2}W_j|\rho|\\&\hspace{16mm}\lesssim \sum\limits_{1\leq i\neq j\leq\kappa}\big\|\nabla\rho\big\|_{L^{2}}\big\|W_i^{2_{\mu}^{\ast}-2}W_j\big\|_{L^{\frac{2N}{N-\mu+2}}}
\big\|W_i^{2_{\mu}^{\ast}-1}W_j\big\|_{L^{\frac{2N}{2N-\mu}}}=:(I_5).
\endaligned
\end{equation*}
By straight computations show that
\begin{equation*}
\aligned
\sum\limits_{i=1}^{\kappa}\int_{\mathbb{R}^N}&\left(\frac{1}{|x|^{\mu}}\ast \Big|\alpha_iW_i\Big|^{2_{\mu}^{\ast}} \right)\Big||u|^{2_{\mu}^{\ast}-2}u-|\sigma|^{2_{\mu}^{\ast}-2}\sigma\Big|\rho\\&
\leq
\sum\limits_{i=1}^ {\kappa}\big|\alpha_i\big|^{2_{\mu}^{\ast}}\int_{\mathbb{R}^N}\Big(\frac{1}{|x|^{\mu}}\ast\big|W_i\big|^{2_{\mu}^{\ast}}\Big)\Big((2_{\mu}^{\ast}-1)|\sigma|^{2_{\mu}^{\ast}-2}|\rho|^2+\widetilde{C}_2\big(|\sigma|^{2_{\mu}^{\ast}-3}|\rho|^3+|\rho|^{2_{\mu}^{\ast}}\big)\Big)
\\&\leq \widetilde{C}_5\Big(\big\|\nabla\rho\big\|_{L^{2}}^3+\big\|\nabla\rho\big\|_{L^{2}}^\frac{2N-\mu}{N-2}\Big)
+(2_{\mu}^{\ast}-1)\sum\limits_{i=1}^ {\kappa}\big|\alpha_i\big|^{2_{\mu}^{\ast}}\int_{\mathbb{R}^N}\Big(\frac{1}{|x|^{\mu}}\ast\big|W_i\big|^{2_{\mu}^{\ast}}\Big)|\sigma|^{2_{\mu}^{\ast}-2}|\rho|^2,
\endaligned
\end{equation*}
and analogously
\begin{equation*}
\aligned
\sum\limits_{i=1}^{\kappa}\int_{\mathbb{R}^N}&\left(\frac{1}{|x|^{\mu}}\ast \Big||u|^{2_{\mu}^{\ast}}-|\sigma|^{2_{\mu}^{\ast}}\Big| \right)\big|\alpha_i\big|^{2_{\mu}^{\ast}-2}\alpha_i\big|W_i\big|^{2_{\mu}^{\ast}-2}W_i\rho\\&
\leq
\sum\limits_{i=1}^{\kappa}\big|\alpha_i\big|^{2_{\mu}^{\ast}-1}\int_{\mathbb{R}^N}\bigg[\frac{1}{|x|^{\mu}}\ast\Big(2_{\mu}^{\ast}|\sigma|^{2_{\mu}^{\ast}-1}|\rho|+\widetilde{C}_1\big(|\sigma|^{2_{\mu}^{\ast}-2}|\rho|^2+|\rho|^{2_{\mu}^{\ast}}\big)\Big)\bigg]W_i^{2_{\mu}^{\ast}-1}|\rho|
\\&\leq \widetilde{C}_6\Big(\big\|\nabla\rho\big\|_{L^{2}}^3+\big\|\nabla\rho\big\|_{L^{2}}^\frac{3N-\mu-2}{N-2}\Big)
+2_{\mu}^{\ast}\sum\limits_{i=1}^{\kappa}\big|\alpha_i\big|^{2_{\mu}^{\ast}-1}\int_{\mathbb{R}^N}\Big(\frac{1}{|x|^{\mu}}\ast\big(|\sigma|^{2_{\mu}^{\ast}-1}|\rho|\big)\Big)W_i^{2_{\mu}^{\ast}-1}|\rho|.
\endaligned
\end{equation*}
Using \eqref{EP1} we have
\begin{equation*}
\aligned
\sum\limits_{i=1}^{\kappa}\int_{\mathbb{R}^N}\Big(\frac{1}{|x|^{\mu}}\ast\Big|\sum\limits_{i=1}^{\kappa}\alpha_iW_i\Big|^{2_{\mu}^{\ast}}\Big)|\alpha_i|^{2_{\mu}^{\ast}-2}\alpha_i|W_i|^{2_{\mu}^{\ast}-2}W_i\rho=0.
\endaligned
\end{equation*}
Substituting these estimates into the first term on the right-hand side of \eqref{Sta1}, we have that
\begin{equation}\label{up}
\aligned
&\int_{\mathbb{R}^N}\left(\frac{1}{|x|^{\mu}}\ast |u|^{2_{\mu}^{\ast}}\right)|u|^{2_{\mu}^{\ast}-2}u\rho\\ \leq&
(2_{\mu}^{\ast}-1)\sum\limits_{i=1}^ {\kappa}\big|\alpha_i\big|^{2_{\mu}^{\ast}}\int_{\mathbb{R}^N}\Big(\frac{1}{|x|^{\mu}}\ast\big|W_i\big|^{2_{\mu}^{\ast}}\Big)|\sigma|^{2_{\mu}^{\ast}-2}|\rho|^2
+2_{\mu}^{\ast}\sum\limits_{i=1}^ {\kappa}\big|\alpha_i\big|^{2_{\mu}^{\ast}-1}\int_{\mathbb{R}^N}\Big(\frac{1}{|x|^{\mu}}\ast\big(|\sigma|^{2_{\mu}^{\ast}-1}|\rho|\big)\Big)W_i^{2_{\mu}^{\ast}-1}|\rho|
\\+&\widetilde{C}_{N,1}\Big(\big\|\nabla\rho\big\|_{L^{2}}^3+\big\|\nabla\rho\big\|_{L^{2}}^4+\big\|\nabla\rho\big\|_{L^{2}}^5+\big\|\nabla\rho\big\|_{L^{2}}^\frac{3N-\mu-2}{N-2}
+\big\|\nabla\rho\big\|_{L^{2}}^\frac{4N-\mu-4}{N-2}+\big\|\nabla\rho\big\|_{L^{2}}^\frac{4N-2\mu}{N-2}+\big\|\nabla\rho\big\|_{L^{2}}^\frac{2N-\mu}{N-2}\Big)\\+&
\widetilde{C}_{N,2}\Big((I_1)+(I_2)+(I_3)+(I_4)+(I_5)\Big).
\endaligned
\end{equation}

We consider $(I_1)$.
For any $i\neq j$, $3\leq N<6-\mu$, by Proposition~\ref{FPU1}, it is clear that
\begin{equation}\label{I11}
\aligned
\big\|W_i^{2_{\mu}^{\ast}-2}W_j\big\|_{L^{\frac{2N}{N-\mu+2}}}
&=\Big(\int_{\mathbb{R}^N}W_i^{\frac{2N(2_{\mu}^{\ast}-2)}{N-\mu+2}}W_j^{\frac{2N}{N-\mu+2}}\Big)^{\frac{N-\mu+2}{2N}}\\&\approx
\min\Big(\frac{\lambda_i}{\lambda_j},\frac{\lambda_j}{\lambda_i},\frac{1}{\lambda_i\lambda_j|\xi_i-\xi_j|^2}\Big)^{\frac{N-2}{2}}
=\int_{\mathbb{R}^N}W_i^{2^{\ast}-1}W_j.
\endaligned
\end{equation}
Combining this estimate we obtain
\begin{equation}\label{I1}
\aligned (I_1) \lesssim\sum\limits_{1\leq i\neq j\leq\kappa}\bigg(\big\|\nabla\rho\big\|_{L^{2}}^2+\big\|\nabla\rho\big\|_{L^{2}}^3+\big\|\nabla\rho\big\|_{L^{2}}^\frac{3N-\mu-2}{N-2}\bigg)
\int_{\mathbb{R}^N}W_i^{2^{\ast}-1}W_j
\endaligned
\end{equation}
and
\begin{equation*}
\aligned
(I_4)\lesssim \sum\limits_{1\leq i\neq j\leq\kappa}\big\|\nabla\rho\big\|_{L^{2}}
\int_{\mathbb{R}^N}W_i^{2^{\ast}-1}W_j.
\endaligned
\end{equation*}
Similar to the above argument, for any $i\neq j$, it holds that
\begin{equation}\label{I2}
\aligned
\big\|W_i^{2_{\mu}^{\ast}-1}W_j\big\|_{L^{\frac{2N}{2N-\mu}}}\approx\int_{\mathbb{R}^N}W_i^{2^{\ast}-1}W_j.
\endaligned
\end{equation}
So
\begin{equation*}
\aligned
(I_2)\lesssim \sum\limits_{1\leq i\neq j\leq\kappa}\Big(\big\|\nabla\rho\big\|_{L^{2}}^2+\big\|\nabla\rho\big\|_{L^{2}}^3+\big\|\nabla\rho\big\|_{L^{2}}^\frac{2N-\mu}{N-2}\Big)
\int_{\mathbb{R}^N}W_i^{2^{\ast}-1}W_j
\endaligned
\end{equation*}
and
\begin{equation*}
\aligned
(I_3)\lesssim \sum\limits_{1\leq i\neq j\leq\kappa}\big\|\nabla\rho\big\|_{L^{2}}
\int_{\mathbb{R}^N}W_i^{2^{\ast}-1}W_j.
\endaligned
\end{equation*}
Then, we find from \eqref{I11} and \eqref{I2} that
\begin{equation*}
\aligned
(I_5)\lesssim \sum\limits_{1\leq i\neq j\leq\kappa}\big\|\nabla\rho\big\|_{L^{2}}\Big(\int_{\mathbb{R}^N}W_i^{2^{\ast}-1}W_j\Big)^2.
\endaligned
\end{equation*}
Now note that the identity
\begin{equation}\label{eq3.18}
\int_{\mathbb{R}^N}\frac{1}{|x-y|^{2\gamma}}\Big(\frac{1}{1+|y|^{2}}\Big)^{N-\gamma}dy
=I(\gamma)\Big(\frac{1}{1+|x|^{2}}\Big)^{\gamma},\ \ 0 <\gamma < \frac{N}{2},
\end{equation}
where
$$
I(\gamma)=\frac{\pi^{\frac{N}{2}}\Gamma(\frac{N-2\gamma}{2})}{\Gamma(N-\gamma)} \ \ \mbox{with}\ \Gamma(\gamma)=\int_0^{+\infty} x^{\gamma-1}e^{-x}\,dx, \quad \gamma>0,
$$
then we get
\begin{equation*}
|x|^{-\mu}\ast |W_i|^{2_{\mu}^{\ast}}
=\int_{\mathbb{R}^N}\frac{W_i^{2_{\mu}^{\ast}}(y)}{|x-y|^{\mu}}dy
=\mathcal{\widetilde{Q}}W_i^{2^{\ast}-2_{\mu}^{\ast}}(x),
\end{equation*}
where  $$\mathcal{\widetilde{Q}}=I(\gamma)S^{\frac{(N-\mu)(2-N)}{4(N-\mu+2)}}[C(N,\mu)]^{\frac{2-N}{2(N-\mu+2)}}[N(N-2)]^{\frac{N-2}{4}}.$$
Therefore, by applying the Proposition \ref{estimate2} implies that
\begin{equation*}
\begin{split}
&\sum\limits_{1\leq i\neq j\leq\kappa}\bigg(\big\|\nabla\rho\big\|_{L^{2}}^2+\big\|\nabla\rho\big\|_{L^{2}}^3+\big\|\nabla\rho\big\|_{L^{2}}^\frac{2N-\mu}{N-2}+\big\|\nabla\rho\big\|_{L^{2}}^\frac{3N-\mu-2}{N-2}\bigg)
\int_{\mathbb{R}^N}W_i^{2^{\ast}-1}W_j\\ \lesssim &\epsilon\bigg(\big\|\nabla\rho\big\|_{L^{2}}^2+\big\|\nabla\rho\big\|_{L^{2}}^3+\big\|\nabla\rho\big\|_{L^{2}}^\frac{2N-\mu}{N-2}+\big\|\nabla\rho\big\|_{L^{2}}^\frac{3N-\mu-2}{N-2}\bigg)\\
+&\bigg(\big\|\nabla\rho\big\|_{L^{2}}^2+\big\|\nabla\rho\big\|_{L^{2}}^3+\big\|\nabla\rho\big\|_{L^{2}}^\frac{2N-\mu}{N-2}+\big\|\nabla\rho\big\|_{L^{2}}^\frac{3N-\mu-2}{N-2}\bigg)\Big\|\Delta u+\Big(\frac{1}{|x|^{\mu}}\ast |u|^{2_{\mu}^{\ast}}\Big)|u|^{2_{\mu}^{\ast}-2}u\Big\|_{(\mathcal{D}^{1,2}(\mathbb{R}^N))^{-1}}\\
+&\Big\|\nabla\rho\Big\|_{L^2}^{\min\big(4,\frac{3N-\mu+2}{N-2}\big)}+\Big\|\nabla\rho\Big\|_{L^2}^{\min\big(5,\frac{4N-\mu-4}{N-2}\big)}+\Big\|\nabla\rho\Big\|_{L^2}^{\min\big(\frac{4N-\mu-4}{N-2},\frac{3N-2\mu+2}{N-2}\big)}
+\Big\|\nabla\rho\Big\|_{L^2}^{\min\big(\frac{4}{5N-\mu-2},\frac{4N-2\mu+4}{N-2}\big)},
\end{split}
\end{equation*}
using also that
\begin{equation*}
\begin{split}
\sum\limits_{1\leq i\neq j\leq\kappa}&\big\|\nabla\rho\big\|_{L^{2}}\Big(\int_{\mathbb{R}^N}W_i^{2^{\ast}-1}W_j\Big)^2\\
&\lesssim \epsilon^2\big\|\nabla\rho\big\|_{L^2}^3+\big\|\nabla\rho\big\|_{L^{2}}\Big\|\Delta u+\Big(\frac{1}{|x|^{\mu}}\ast |u|^{2_{\mu}^{\ast}}\Big)|u|^{2_{\mu}^{\ast}-2}u\Big\|_{(\mathcal{D}^{1,2}(\mathbb{R}^N))^{-1}}^2
+\Big\|\nabla\rho\Big\|_{L^2}^{\min\big(5,\frac{3N-2\mu+2}{N-2}\big)},
\end{split}
\end{equation*}
and
\begin{equation*}
\begin{split}
\sum\limits_{1\leq i\neq j\leq\kappa}&\big\|\nabla\rho\big\|_{L^{2}}
\int_{\mathbb{R}^N}W_i^{2^{\ast}-1}W_j \\&\lesssim \epsilon\big\|\nabla\rho\big\|_{L^2}^2+\big\|\nabla\rho\big\|_{L^{2}}\Big\|\Delta u+\Big(\frac{1}{|x|^{\mu}}\ast |u|^{2_{\mu}^{\ast}}\Big)|u|^{2_{\mu}^{\ast}-2}u\Big\|_{(\mathcal{D}^{1,2}(\mathbb{R}^N))^{-1}}+\Big\|\nabla\rho\Big\|_{L^2}^{\min\big(3,\frac{2N-\mu}{N-2}\big)}.
\end{split}
\end{equation*}

Now we are in a position to prove Theorem \ref{Figalli}. To this end, choosing $\epsilon>0$ such that
$$
\nu(N,\kappa)+\epsilon\widetilde{C}_{N,3}\widetilde{C}_{N,5}+\epsilon\widetilde{C}_{N,4}\widetilde{C}_{N,5}<1,
$$
we are able to conclude that
\begin{equation*}
\begin{split}
\int_{\mathbb{R}^N}\left(\frac{1}{|x|^{\mu}}\ast |u|^{2_{\mu}^{\ast}}\right)|u|^{2_{\mu}^{\ast}-2}u\rho &\leq\Big[\nu(N,\kappa)+\epsilon\widetilde{C}_{N,3}\widetilde{C}_{N,5}+\epsilon\widetilde{C}_{N,4}\widetilde{C}_{N,5}\Big]\big\|\nabla\rho\big\|_{L^2}^2
\\&+(J_1)+(J_2)+(J_3)+(J_4)+(J_5)+(J_6)+(J_7),
\end{split}
\end{equation*}
where
\begin{equation*}
\begin{split}
&(J_1):=
\widetilde{C}_{N,6}\big\|\nabla\rho\big\|_{L^{2}}
\Big\|\Delta u+\Big(\frac{1}{|x|^{\mu}}\ast |u|^{2_{\mu}^{\ast}}\Big)|u|^{2_{\mu}^{\ast}-2}u\Big\|_{(\mathcal{D}^{1,2}(\mathbb{R}^N))^{-1}},
\\&
(J_2):=
\widetilde{C}_{N,7}\big\|\nabla\rho\big\|_{L^{2}}\Big\|\Delta u+\Big(\frac{1}{|x|^{\mu}}\ast |u|^{2_{\mu}^{\ast}}\Big)|u|^{2_{\mu}^{\ast}-2}u\Big\|_{(\mathcal{D}^{1,2}(\mathbb{R}^N))^{-1}}^2,
\\&
(J_3):=
\widetilde{C}_{N,8}\bigg(\big\|\nabla\rho\big\|_{L^{2}}^2+\big\|\nabla\rho\big\|_{L^{2}}^3\bigg)
\Big\|\Delta u+\Big(\frac{1}{|x|^{\mu}}\ast |u|^{2_{\mu}^{\ast}}\Big)|u|^{2_{\mu}^{\ast}-2}u\Big\|_{(\mathcal{D}^{1,2})^{-1}(\mathbb{R}^N)},\\&
(J_4):=
\widetilde{C}_{N,9}\bigg(\big\|\nabla\rho\big\|_{L^{2}}^{\frac{2N-\mu}{N-2}}+\big\|\nabla\rho\big\|_{L^{2}}^{\frac{3N-\mu-2}{N-2}}\bigg)
\Big\|\Delta u+\Big(\frac{1}{|x|^{\mu}}\ast |u|^{2_{\mu}^{\ast}}\Big)|u|^{2_{\mu}^{\ast}-2}u\Big\|_{(\mathcal{D}^{1,2})^{-1}(\mathbb{R}^N)},\\&
(J_5):=\widetilde{C}_{N,10}\bigg(\big\|\nabla\rho\big\|_{L^{2}}^3+\big\|\nabla\rho\big\|_{L^{2}}^4+\big\|\nabla\rho\big\|_{L^{2}}^5+\big\|\nabla\rho\big\|_{L^{2}}^\frac{3N-\mu-2}{N-2}
+\big\|\nabla\rho\big\|_{L^{2}}^\frac{4N-\mu-4}{N-2}+\big\|\nabla\rho\big\|_{L^{2}}^\frac{4N-2\mu}{N-2}+\big\|\nabla\rho\big\|_{L^{2}}^\frac{2N-\mu}{N-2}\bigg)
,\\&
(J_6):=\widetilde{C}_{N,11}\bigg(\Big\|\nabla\rho\Big\|_{L^2}^{\min\big(5,\frac{3N-2\mu+2}{N-2}\big)}+\Big\|\nabla\rho\Big\|_{L^2}^{\min\big(4,\frac{3N-\mu+2}{N-2}\big)}+\Big\|\nabla\rho\Big\|_{L^2}^{\min\big(5,\frac{4N-\mu-4}{N-2}\big)}
+\Big\|\nabla\rho\Big\|_{L^2}^{\min\big(5,\frac{3N-2\mu+2}{N-2}\big)}\bigg),\\&
(J_7):=\widetilde{C}_{N,12}\bigg(\Big\|\nabla\rho\Big\|_{L^2}^{\min\big(\frac{4N-\mu-4}{N-2},\frac{3N-2\mu+2}{N-2}\big)}
+\Big\|\nabla\rho\Big\|_{L^2}^{\min\big(\frac{4}{5N-\mu-2},\frac{4N-2\mu+4}{N-2}\big)}\bigg).
\end{split}
\end{equation*}
As a consequence, together with \eqref{Sta1} and $\|\nabla\rho\big\|_{L^{2}}\ll1$,
 \begin{equation*}
\begin{split}
\Big[1-\nu(N,\kappa)&-\epsilon\widetilde{C}_{N,3}\widetilde{C}_{N,5}-\epsilon\widetilde{C}_{N,4}\widetilde{C}_{N,5}\Big]\big\|\nabla\rho\big\|_{L^2}^2
\\&\lesssim \Big\|\nabla\rho\Big\|_{L^{2}}
\Big\|\Delta u+\Big(\frac{1}{|x|^{\mu}}\ast |u|^{2_{\mu}^{\ast}}\Big)|u|^{2_{\mu}^{\ast}-2}u\Big\|_{(\mathcal{D}^{1,2}(\mathbb{R}^N))^{-1}}+(J_4)+(J_5)+(J_6)+(J_7).
\end{split}
\end{equation*}
Combining this inequality with $\|\nabla\rho\big\|_{L^{2}}\ll1$ yields the conclusion
\begin{equation*}
\begin{split}
\big\|\nabla\rho\big\|_{L^2}
\lesssim
\Big\|\Delta u+\Big(\frac{1}{|x|^{\mu}}\ast |u|^{2_{\mu}^{\ast}}\Big)|u|^{2_{\mu}^{\ast}-2}u\Big\|_{(\mathcal{D}^{1,2}(\mathbb{R}^N))^{-1}}.
\end{split}
\end{equation*}
On the other hand, by \eqref{www} of the Proposition \ref{estimate2}, the conclusion follows.

$\mathbf{Case\ 2.}$
 The exponents $\mu$ and $N$ satisfy
 \begin{equation*}
\left\lbrace
\begin{aligned}
N&=4, \hspace{2mm}\mu\in(2,N),\\
N&\geq5,\hspace{2mm}\mu\in(0,N)\hspace{2mm}\mbox{and}\hspace{2mm}\mu\in(1,4].
\end{aligned}
\right.
\end{equation*}
We find the elementary inequality
\begin{equation*}
\Big|u|u|^{2_{\mu}^{\ast}-2}-\sigma|\sigma|^{2_{\mu}^{\ast}-2}\Big|\leq (2_{\mu}^{\ast}-1)|\sigma|^{2_{\mu}^{\ast}-2}|\rho|+\widetilde{C}_3|\rho|^{2_{\mu}^{\ast}-1},\quad\text{for}~~1<2_{\mu}^\ast-1\leq2.
\end{equation*}
The argument of the statement is identical to the proof of case 1. Combining this equality with \eqref{pmr1} yields the conclusion.

    \qed

\begin{remark}
In the proof of Theorem \ref{Figalli}, it is pointed out that the evaluations when $N\geq6-\mu$ is not sufficient to obtain the desired conclusion. In fact, for $N>6-\mu$, we have that
\begin{equation*}
\aligned
\big\|W_i^{2_{\mu}^{\ast}-2}W_j\big\|_{L^{\frac{2N}{N-\mu+2}}}
&=\Big(\int_{\mathbb{R}^N}W_i^{\frac{2N(2_{\mu}^{\ast}-2)}{N-\mu+2}}W_j^{\frac{2N}{N-\mu+2}}\Big)^{\frac{N-\mu+2}{2N}}\\&
\approx\Big(\int_{\mathbb{R}^N}W_i^{2^{\ast}-1}W_j\Big)^{2_{\mu}^{\ast}-2}\gg\int_{\mathbb{R}^N}W_i^{2^{\ast}-1}W_j.
\endaligned
\end{equation*}
For $N=6-\mu$, we get that
\begin{equation*}
\aligned
\big\|W_i^{2_{\mu}^{\ast}-2}W_j\big\|_{L^{\frac{2N}{N-\mu+2}}}
\approx\Big(\int_{\mathbb{R}^N}W_i^{2^{\ast}-1}W_j\Big)^{\frac{3(8-\mu)}{24}}\Big|\log\Big(\int_{\mathbb{R}^N}W_i^{2^{\ast}-1}W_j\Big)\Big|^{\frac{8-\mu}{12}}\gg\int_{\mathbb{R}^N}W_i^{2^{\ast}-1}W_j.
\endaligned
\end{equation*}
\end{remark}

\textbf{Proof of  Corollary~\ref{Figalli2}.}
We can now prove Corollary~\ref{Figalli2} by using Theorems \ref{Figalli} and \ref{F4}.
In view of Theorem~\ref{F4} there exists $\epsilon>0$ such that
$$\Big\|\Delta u+\left(\frac{1}{|x|^{\mu}}\ast |u|^{2_{\mu}^{\ast}}\right)|u|^{2_\mu^{\ast}-2}u\Big\|_{(\mathcal{D}^{1,2}(\mathbb{R}^N))^{-1}}\leq\epsilon.$$
Then, for any $\delta>0$, we have that
$$
\Big\|\nabla u-\sum_{i=1}^{\kappa}\nabla W[\xi_i,\lambda_i]\Big\|_{L^2}\leq\delta,
$$
where $(W[\xi_i,\lambda_i])_{1\leq i\leq\kappa}$ is a $\delta$-interacting family of Talenti bubbles.
Combining Theorem~\ref{Figalli} and the conclusion follows.

    \qed

\textbf{Proof of Corollary \ref{Figalli-0}.} For dimension $N\geq3$ and a single bubble case, combining the arguments of Theorem \ref{Figalli} and Corollary \ref{Figalli2}, the result easily follows.

    \qed

\section{Proofs of Propositions 4.2 and 4.3}\label{section5}
We complete here the proof of Theorem \ref{Figalli} by proving Propositions~\ref{estimate1} and \ref{estimate2}.
The following two Propositions are the analog of Lemmas 3.8-3.9 of \cite{FG20}, which are very useful in the sequel.
\begin{proposition}\label{fai1}
Let $N\geq1$. Given a point $\tilde{x}\in\mathbb{R}^N$ and two radii $0<r<R$, there exists a Lipschitz bump function $\psi=\psi_{\tilde{x},r,R}:\mathbb{R}^N\rightarrow[0,1]$ such that $\psi\equiv1$ in $B_r(\tilde{x})$, $\psi\equiv$ in $B_{R}(\tilde{x})$ , and
\begin{equation*}
\int_{\mathbb{R}^N}|\nabla\psi|^N\lesssim\log\Big(\frac{R}{r}\Big)^{1-N}.
\end{equation*}
\end{proposition}

\begin{proposition}\label{fai2}
For any $N\geq3$, $\kappa\in\mathbb{N}$, and $\epsilon>0$, there exists $\delta=\delta(N,\kappa,\epsilon)>0$ such that if $(W[\xi_i,\lambda_i])_{1\leq i\leq\kappa}$ is a $\delta$-interacting family of Talenti bubbles. Then for any $1\leq i\leq\kappa$ there exists a Lipschitz bump function $\Psi_i:\mathbb{R}^N\rightarrow[0,1]$ such that the following hold:
\begin{enumerate}
		\item Almost all mass of $\big(|x|^{-\mu}\ast W_i^{2_{\mu}^*}\big)W_i^{2_{\mu}^*}$ is in the region $\{\Psi_i=1\}$, that is
$$\int_{\{\Psi_i=1\}}\Big(|x|^{-\mu}\ast W_i^{2_{\mu}^*}\Big)W_i^{2_{\mu}^*}=\mathcal{\widetilde{Q}}_0S^{\frac{(N-\mu)(2-N)}{4(N-\mu+2)}}\int_{\{\Psi_i=1\}}W_i^{2^{\ast}}\geq\mathcal{\widetilde{Q}}_0(1-\epsilon)S^{\frac{(N-\mu)(2-N)}{4(N-\mu+2)}+N},$$
where  $$\mathcal{\widetilde{Q}}_0=I(\gamma)[C(N,\mu)]^{\frac{2-N}{2(N-\mu+2)}}[N(N-2)]^{\frac{N-2}{4}},$$

$$I(\gamma)=\frac{\pi^{\frac{N}{2}}\Gamma(\frac{N-2\gamma}{2})}{\Gamma(N-\gamma)},\ \Gamma(\gamma)=\int_0^{+\infty} x^{\gamma-1}e^{-x}\,dx ,\ \ \forall \gamma>0,$$  $C(N,\mu)$ is the sharp constant of HLS inequality given as in (\ref{cnu}).
\item In the region $\{\Psi_i>0\}$ it holds $\epsilon W_i>W_j$ for any $j\neq i$.
\item The $L^N$-norm of the gradient is small, that is
$$\big\|\nabla\Psi_i\big\|_{L^N}\leq\epsilon.$$
\item
For any $j\neq i$ such that $\lambda_j\leq \lambda_i$, it holds that
\begin{equation*}
\frac{\sup_{\{\Psi_i>0\}}W_j}{\inf_{\{\Psi_i>0\}}W_j}\leq1+\epsilon.
\end{equation*}
\end{enumerate}
\end{proposition}
\begin{proof}
The proof of this proposition is very similar to that of Lemma 3.9 of \cite{FG20}. Therefore, we will just sketch the proof of $(1)$. Without loss of generality, we only consider one of the indices, i.e. $W:=W_\kappa$.
Assume that $j\in\{1,\cdots,\kappa-1\}$ such that $\lambda_j>1$ and $|\xi_j|<2r$. Let $r_j\in(0,\infty)$ be the only positive real number such that
$$\epsilon\Big(\frac{1}{1+r^2}\Big)\frac{N-2}{2}=\Big(\frac{\lambda_j}{1+\lambda_j^2|r_j|^2}\Big)^{\frac{N-2}{2}}.$$
Then $\delta$ is sufficiently small to tell us that $r_j\leq\epsilon^2$ for any $j\in\{1,\cdots,\kappa-1\}$. Notice that
\begin{equation*}
|x|^{-\mu}\ast |W(x)|^{2_{\mu}^{\ast}}
=\mathcal{\widetilde{Q}}S^{\frac{(N-\mu)(2-N)}{4(N-\mu+2)}}W^{2^{\ast}-2_{\mu}^{\ast}}(x),
\end{equation*}
where  $$\mathcal{\widetilde{Q}}_0=I(\gamma)[C(N,\mu)]^{\frac{2-N}{2(N-\mu+2)}}[N(N-2)]^{\frac{N-2}{4}}~~\text{with}~~I(\gamma)=\frac{\pi^{\frac{N}{2}}\Gamma(\frac{N-2\gamma}{2})}{\Gamma(N-\gamma)}.$$
Then we denote $\Psi=\Psi_\kappa$ and imply that
\begin{equation*}
\begin{split}
\int_{\{\Psi<1\}}\Big(|x|^{-\mu}\ast W^{2_{\mu}^{\ast}}\Big)W^{2_{\mu}^{\ast}}&=\mathcal{\widetilde{Q}}_0S^{\frac{(N-\mu)(2-N)}{4(N-\mu+2)}}\int_{\{\Psi<1\}}W^{2^{\ast}}(x)\leq\mathcal{\widetilde{Q}}_0S^{\frac{(N-\mu)(2-N)}{4(N-\mu+2)}}\int_{B_{\varepsilon r}^{c}(0)}W^{2^{\ast}}\\&+\mathcal{\widetilde{Q}}_0S^{\frac{(N-\mu)(2-N)}{4(N-\mu+2)}}\sum\limits_{j}\int_{B_{\epsilon^{-1}r_j}(\xi_j)}W^{2^{\ast}}\\&
\leq \mathcal{\widetilde{Q}}_0S^{\frac{(N-\mu)(2-N)}{4(N-\mu+2)}}\int_{B_{\varepsilon r}^{c}(0)}W^{2^{\ast}}+\widetilde{C}_{N}\mathcal{\widetilde{Q}}_0S^{\frac{(N-\mu)(2-N)}{4(N-\mu+2)}}\kappa\epsilon^N\leq\mathcal{\widetilde{Q}}_0\epsilon S^{\frac{(N-\mu)(2-N)}{4(N-\mu+2)}+N},
\end{split}
\end{equation*}
by $\epsilon$ is sufficient small and choosing $r=r(\epsilon)$ large enough. Thus we find that
\begin{equation*}
\begin{split}
\int_{\{\Psi=1\}}\Big(|x|^{-\mu}\ast W^{2_{\mu}^*}\Big)W^{2_{\mu}^*}&=\mathcal{\widetilde{Q}}_0S^{\frac{(N-\mu)(2-N)}{4(N-\mu+2)}}\int_{\{\Psi=1\}}W^{2^{\ast}}\\&\geq\mathcal{\widetilde{Q}}_0S^{\frac{(N-\mu)(2-N)}{4(N-\mu+2)}}\int_{\{\Psi=1\}}W^{2^{\ast}}
+\mathcal{\widetilde{Q}}_0S^{\frac{(N-\mu)(2-N)}{4(N-\mu+2)}}\Big(\int_{\{\Psi<1\}}W^{2^{\ast}}-\varepsilon\mathcal{S}^N\Big).
\end{split}
\end{equation*}
The
conclusion follows.

\end{proof}

\subsection{Proof of Proposition \ref{estimate1}.}
Proposition \ref{fai2}-(2) can be used to deduce that for choosing $\epsilon=o(1)$,
    \begin{equation}\label{fai3}
    \begin{split}
(2_{\mu}^{\ast}-1)\sum\limits_{i=1}^ {\kappa}\big|\alpha_i\big|^{2_{\mu}^{\ast}}&\int_{\mathbb{R}^N}\Big(\frac{1}{|x|^{\mu}}\ast W_j^{2_{\mu}^{\ast}}\Big)\sigma^{2_{\mu}^{\ast}-2}\rho^2\\& \leq\big(1+o(1)\big)(2_{\mu}^{\ast}-1)\sum_{i=1}^{\kappa}\int_{\mathbb{R}^N}\Big(\frac{1}{|x|^{\mu}}\ast \big(\sum_{j=1}^{\kappa}W_j^{2_{\mu}^{\ast}}\big)\Big)W_i^{2_{\mu}^{\ast}-2}(\Psi_i\rho)^2\\&\hspace{8mm}+(2_{\mu}^{\ast}-1)\sum_{j=1}^{\kappa}\int_{\{\sum\Psi_i<1\}}\Big(\frac{1}{|x|^{\mu}}\ast W_j^{2_{\mu}^{\ast}}\Big)\sigma^{2_{\mu}^{\ast}-2}\rho^2.
\end{split}
\end{equation}
In virtue of Proposition \ref{fai2}-(2), we get that
\begin{equation*}
\begin{split}
\sum_{i=1}^{\kappa}\int_{\mathbb{R}^N}\Big(\frac{1}{|x|^{\mu}}\ast W_1^{2_{\mu}^{\ast}}\Big)W_i^{2_{\mu}^{\ast}-2}(\Psi_i\rho)^2&\leq\big(1+o(1)\big)\int_{\mathbb{R}^N}\Big(\frac{1}{|x|^{\mu}}\ast W_1^{2_{\mu}^{\ast}}\Big)W_1^{2_{\mu}^{\ast}-2}(\Psi_1\rho)^2,\\&
\cdots,
\end{split}
\end{equation*}
\begin{equation*}
\sum_{i=1}^{\nu}\int_{\mathbb{R}^N}\Big(\frac{1}{|x|^{\mu}}\ast W_{i-1}^{2_{\mu}^{\ast}}\Big)W_i^{2_{\mu}^{\ast}-2}(\Psi_i\rho)^2\leq\big(1+o(1)\big)\int_{\mathbb{R}^N}\Big(\frac{1}{|x|^{\mu}}\ast W_{i-1}^{2_{\mu}^{\ast}}\Big)W_{i-1}^{2_{\mu}^{\ast}-2}(\Psi_{i-1}\rho)^2,
\end{equation*}
\begin{equation*}
\begin{split}
\sum_{i=1}^{\kappa}\int_{\mathbb{R}^N}\Big(\frac{1}{|x|^{\mu}}\ast W_{i+1}^{2_{\mu}^{\ast}}\Big)W_i^{2_{\mu}^{\ast}-2}(\Psi_i\rho)^2&\leq\big(1+o(1)\big)\int_{\mathbb{R}^N}\Big(\frac{1}{|x|^{\mu}}\ast W_{i+1}^{2_{\mu}^{\ast}}\Big)W_{i+1}^{2_{\mu}^{\ast}-2}(\Psi_{i+1}\rho)^2,\\&
\cdots,
\end{split}
\end{equation*}
\begin{equation*}
\sum_{i=1}^{\kappa}\int_{\mathbb{R}^N}\Big(\frac{1}{|x|^{\mu}}\ast W_{\kappa}^{2_{\mu}^{\ast}}\Big)W_i^{2_{\mu}^{\ast}-2}(\Psi_i\rho)^2\leq\big(1+o(1)\big)\int_{\mathbb{R}^N}\Big(\frac{1}{|x|^{\mu}}\ast W_{\kappa}^{2_{\mu}^{\ast}}\Big)W_{\kappa}^{2_{\mu}^{\ast}-2}(\Psi_{\kappa}\rho)^2.
\end{equation*}
Hence we are able to conclude that
\begin{equation}\label{conclude}
\begin{split}
   \big(1+o(1)\big)(2_{\mu}^{\ast}-1)\sum_{i=1}^{\kappa}\int_{\mathbb{R}^N}&\Big(\frac{1}{|x|^{\mu}}\ast \big(\sum_{j=1}^{\kappa}W_j^{2_{\mu}^{\ast}}\big)\Big)W_i^{2_{\mu}^{\ast}-2}(\Psi_i\rho)^2
   \\&\leq\big(1+o(1)\big)(2_{\mu}^{\ast}-1)\sum_{i=1}^{\kappa}\int_{\mathbb{R}^N}\Big(\frac{1}{|x|^{\mu}}\ast W_i^{2_{\mu}^{\ast}}\Big)W_i^{2_{\mu}^{\ast}-2}(\Psi_i\rho)^2.
   \end{split}
\end{equation}
Then, we similarly compute and get
 \begin{equation*}
 \begin{split}
\big(1+o(1)\big)2_{\mu}^{\ast}\sum_{i=1}^{\kappa}\int_{\mathbb{R}^N}&\Big(\frac{1}{|x|^{\mu}}\ast \big(W_i^{2_{\mu}^{\ast}-1}\Psi_i\rho\big)\Big)\big(\sum_{j=1}^{\kappa}W_j^{2_{\mu}^{\ast}-1}\big)\Psi_i\rho\\ &\leq
\big(1+o(1)\big)2_{\mu}^{\ast}\sum_{i=1}^{\kappa}\int_{\mathbb{R}^N}\Big(\frac{1}{|x|^{\mu}}\ast \big(W_i^{2_{\mu}^{\ast}-1}\Psi_i\rho\big)\Big)W_i^{2_{\mu}^{\ast}-1}\Psi_i\rho.
\end{split}
\end{equation*}
Thus application of Proposition \ref{fai2}-(2) ensures that
    \begin{equation}\label{rtbdlb}
    \begin{split}
    2_{\mu}^{\ast}&\sum\limits_{j=1}^{\kappa}\big|\alpha_i\big|^{2_{\mu}^{\ast}-1}\int_{\mathbb{R}^N}\Big(\frac{1}{|x|^{\mu}}\ast\big(\sigma^{2_{\mu}^{\ast}-1}\rho\big)\Big)W_j^{2_{\mu}^{\ast}-1}\rho\\ &\leq\big(1+o(1)\big)2_{\mu}^{\ast}\sum_{i=1}^{\kappa}\int_{\mathbb{R}^N}\Big(\frac{1}{|x|^{\mu}}\ast \big(W_i^{2_{\mu}^{\ast}-1}\Psi_i\rho\big)\Big)W_i^{2_{\mu}^{\ast}-1}\Psi_i\rho\\
    &\hspace{4mm}+
    2_{\mu}^{\ast}\sum_{i=1}^{\kappa}\int_{\mathbb{R}^N}\int_{\{\sum\Psi_i<1\}}\frac{\sigma^{2_{\mu}^{\ast}-1}\rho(y)W_i^{2_{\mu}^{\ast}-1}(x)\Psi_i(x)\rho(x)}{|x-y|^{\mu}}
    \\&\hspace{4mm}+2_{\mu}^{\ast}\sum_{i=1}^{\kappa}\int_{\{\sum\Psi_i<1\}}\int_{\mathbb{R}^N}\frac{W_i^{2_{\mu}^{\ast}-1}(x)\Psi_i(x)\rho(x)W_i^{2_{\mu}^{\ast}-1}(y)\rho(y)}{|x-y|^{\mu}}\\&\hspace{4mm}+
    2_{\mu}^{\ast}\sum_{i=1}^{\kappa}\int_{\{\sum\Psi_i<1\}}\int_{\{\sum\Psi_i<1\}}\frac{\sigma^{2_{\mu}^{\ast}-1}\rho(y)W_i^{2_{\mu}^{\ast}-1}(x)\rho(x)}{|x-y|^{\mu}}.
    \end{split}
    \end{equation}
Combined with Proposition \ref{fai2}-(1), we compute and get
  \begin{equation}\label{f01}
  \begin{split}
    \sum_{j=1}^{\kappa}\int_{\{\sum\Psi_i<1\}}\Big(\frac{1}{|x|^{\mu}}\ast W_j^{2_{\mu}^{\ast}}\Big)\sigma^{2_{\mu}^{\ast}-2}\rho^2 &\leq\sum_{j=1}^{\kappa}\Big(\int_{\mathbb{R}^N}W_j^{2^{\ast}}\Big)^{\frac{2N-\mu}{2N}}
    \Big(\int_{\{\sum\Psi_i<1\}}\sigma^{2^{\ast}}\Big)^{\frac{4-\mu}{2N}}\big\|\nabla\rho\big\|_{L^2}^2 \\&\leq o(1)\big\|\nabla\rho\big\|_{L^2}^2,
    \end{split}
      \end{equation}
      and
 \begin{equation}\label{f0}
  \begin{split}
  \sum_{i=1}^{\kappa}\int_{\mathbb{R}^N}\int_{\{\sum\Psi_i<1\}}\frac{\sigma^{2_{\mu}^{\ast}-1}\rho(y)W_i^{2_{\mu}^{\ast}-1}(x)\Psi_i(x)\rho(x)}{|x-y|^{\mu}}
  &\leq\sum_{i=1}^{\kappa}\Big(\int_{\mathbb{R}^N}W_i^{2^{\ast}}\Big)^{\frac{N-\mu+2}{2N}}
    \Big(\int_{\{\sum\Psi_i<1\}}\sigma^{2^{\ast}}\Big)^{\frac{N-\mu+2}{2N}}\big\|\nabla\rho\big\|_{L^2}^2 \\&\leq o(1)\big\|\nabla\rho\big\|_{L^2}^2.
 \end{split}
      \end{equation}
Similar arguments give that
 \begin{equation}\label{f2}
  \begin{split}
 &\sum_{i=1}^{\kappa}\int_{\{\sum\Psi_i<1\}}\int_{\mathbb{R}^N}\frac{W_i^{2_{\mu}^{\ast}-1}(x)\Psi_i(x)\rho(x)W_i^{2_{\mu}^{\ast}-1}(y)\rho(y)}{|x-y|^{\mu}}\leq o(1)\big\|\nabla\rho\big\|_{L^2}^2,
 \end{split}
      \end{equation}
and
\begin{equation}\label{f3}
  \begin{split}
\sum_{i=1}^{\kappa}\int_{\{\sum\Psi_i<1\}}\int_{\{\sum\Psi_i<1\}}\frac{\sigma^{2_{\mu}^{\ast}-1}\rho(y)W_i^{2_{\mu}^{\ast}-1}(x)\rho(x)}{|x-y|^{\mu}}\leq o(1)\big\|\nabla\rho\big\|_{L^2}^2.
 \end{split}
      \end{equation}
We claim that there exists a constant $\nu(N,\kappa)$ less than $1$ such that
    \begin{equation}\label{rt}
    \begin{split}
  (2_{\mu}^{\ast}-1)\int_{\mathbb{R}^N}&\Big(\frac{1}{|x|^{\mu}}\ast \big(W_i^{2_{\mu}^{\ast}}\big)\Big)W_i^{2_{\mu}^{\ast}-2}(\Psi_i\rho)^2+2_{\mu}^{\ast}\int_{\mathbb{R}^N}\Big(\frac{1}{|x|^{\mu}}\ast \big(W_i^{2_{\mu}^{\ast}-1}\Psi_i\rho\big)\Big)W_i^{2_{\mu}^{\ast}-1}\Psi_i\rho\\&\leq\nu(N,\kappa)\int_{\mathbb{R}^N}|\nabla(\Psi_i\rho)|^2+o(1)\|\nabla\rho\|_{L^{2}}^2.
    \end{split}
    \end{equation}
To verify \eqref{rt},  our goal is to show that $\Psi_i\rho$ almost satisfies the orthogonality conditions. Let us first define $\widehat{f}: \mathbb{R}^N\rightarrow\mathbb{R}$ be, up to scaling, the functions $W_i$. Then
    \begin{equation}\label{rt1}
    \begin{split}
\bigg|\int_{\mathbb{R}^N}\Big(\frac{1}{|x|^{\mu} }\ast W_i^{2_\mu^*}\Big)W_i^{2_\mu^*-2}(\Psi_i\rho) \widehat{f}\bigg|&=\bigg|\int_{\mathbb{R}^N}\Big(\frac{1}{|x|^{\mu} } \ast W_i^{2_\mu^*}\Big)W_i^{2_\mu^*-2}\rho\widehat{f}\big(1-\Psi_i\big)\bigg|\\&\leq\bigg|\int_{\{\Psi_i<1\}}\Big(\frac{1}{|x|^{\mu} } \ast W_i^{2_\mu^*}\Big)W_i^{2_\mu^*-2}\rho\widehat{f}\bigg|\\&
\leq\Big(\int_{\mathbb{R}^N}W_i^{2^{\ast}}\Big)^{\frac{2N-\mu}{2N}}\Big(\int_{\{\Psi_i<1\}}W_i^{2^{\ast}}\Big)^{\frac{4-\mu}{2N}}\|\widehat{f}\|_{L^{2^\ast}}\|\rho\|_{L^{2^\ast}}\leq o(1)\|\nabla\rho\|_{L^{2}},
\end{split}
\end{equation}
by the orthogonality conditions \eqref{EP1}.
Now, we define $\widehat{g}: \mathbb{R}^N\rightarrow\mathbb{R}$ be, up to scaling, the functions  $\frac{\partial W_i}{\partial\lambda}$. We have that
\begin{equation}\label{rt2}
   \begin{split}
   &\bigg|2_\mu^*\int_{\mathbb{R}^N}\Big(\frac{1}{|x|^{\mu} } \ast \big(W_{i}^{2_\mu^*-1}\Psi_i\rho\big)\Big)W_i^{2_\mu^*-1}\widehat{g}
+(2_\mu^*-1)\int_{\mathbb{R}^N}\Big(\frac{1}{|x|^{\mu} }\ast W_{i}^{2_\mu^*}\Big)W_{i}^{2_\mu^*-2}\widehat{g}\Psi_i\rho\bigg|
\\=&\bigg|2_\mu^*\int_{\mathbb{R}^N}\Big(\frac{1}{|x|^{\mu} } \ast \big(W_{i}^{2_\mu^*-1}\rho(1-\Psi_i)\big)\Big)W_i^{2_\mu^*-1}\widehat{g}
+(2_\mu^*-1)\int_{\mathbb{R}^N}\Big(\frac{1}{|x|^{\mu} } \ast W_{i}^{2_\mu^*}\Big)W_{i}^{2_\mu^*-2}\widehat{g}\rho(1-\Psi_i)\bigg|\\ \leq&
2_\mu^*\Big(\int_{\{\Psi_i<1\}}W_i^{2^{\ast}}\Big)^{\frac{N-\mu+2}{2N}}\|W_i\|_{L^{2^\ast}}^{2_{\mu}^{*}-1}\|\widehat{g}\|_{L^{2^\ast}}\|\rho\|_{L^{2^\ast}}+
(2_\mu^*-1)\Big(\int_{\{\Psi_i<1\}}W_i^{2^{\ast}}\Big)^{\frac{4-\mu}{2N}}\|W_i\|_{L^{2^\ast}}^{2_{\mu}^{*}}\|\widehat{g}\|_{L^{2^\ast}}\|\rho\|_{L^{2^\ast}}
\\ \leq& o(1)\|\nabla\rho\|_{L^{2}},
   \end{split}
\end{equation}
by the orthogonality conditions \eqref{EP2}.
Analogously, let $\widehat{h}: \mathbb{R}^N\rightarrow\mathbb{R}$ be, up to scaling, the functions $\frac{\partial W_i}{\partial\xi_i}$. Then we clearly have that
\begin{equation}\label{rt3}
 \begin{split}
   &\bigg|2_\mu^*\int_{\mathbb{R}^N}\Big(\frac{1}{|x|^{\mu} } \ast \big(W_{i}^{2_\mu^*-1}\Psi_i\rho\big)\Big)W_i^{2_\mu^*-1}\widehat{h}
+(2_\mu^*-1)\int_{\mathbb{R}^N}\Big(\frac{1}{|x|^{\mu} } \ast W_{i}^{2_\mu^*}\Big)W_{i}^{2_\mu^*-2}\widehat{h}\Psi_i\rho\bigg|
\\ \leq&
2_\mu^*\Big(\int_{\{\Psi_i<1\}}W_i^{2^{\ast}}\Big)^{\frac{N-\mu+2}{2N}}\|W_i\|_{L^{2^\ast}}^{2_{\mu}^{*}-1}\|\widehat{h}\|_{L^{2^\ast}}\|\rho\|_{L^{2^\ast}}+
(2_\mu^*-1)\Big(\int_{\{\Psi_i<1\}}W_i^{2^{\ast}}\Big)^{\frac{4-\mu}{2N}}\|W_i\|_{L^{2^\ast}}^{2_{\mu}^{*}}\|\widehat{h}\|_{L^{2^\ast}}\|\rho\|_{L^{2^\ast}}
\\ \leq& o(1)\|\nabla\rho\|_{L^{2}},
   \end{split}
\end{equation}
by the orthogonality conditions \eqref{EP3}.
As a consequence, \eqref{rt1}, \eqref{rt2} and \eqref{rt3} tell us that $\Psi_i\rho$ almost orthogonal to $\widehat{f}$, $\widehat{g}$ and $\widehat{h}$, as desired.

Combining Proposition \ref{pro-4.1} and the fact that an orthogonal basis of the eigenfunctions space of $\frac{\mathcal{L}[u]}{\mathcal{R}[u]}$ composed of the functions $W_i$, $\frac{\partial W_i}{\partial\lambda}$ and $\frac{\partial W_i}{\partial\xi_i}$, so that we get that
\begin{equation*}
    \begin{split}
  \int_{\mathbb{R}^N}\Big(\frac{1}{|x|^{\mu} }\ast W_i^{2_{\mu}^{\ast}}\Big)&W_i^{2_{\mu}^{\ast}-2}(\Psi_i\rho)^2+\int_{\mathbb{R}^N}\Big(\frac{1}{|x|^{\mu} }\ast \big(W_i^{2_{\mu}^{\ast}-1}\Psi_i\rho\big)\Big)W_i^{2_{\mu}^{\ast}-1}\Psi_i\rho\\&\leq
  \frac{1}{\hbar}\bigg[\int_{\mathbb{R}^N}|\nabla(\Psi_i\rho)|^2+\int_{\mathbb{R}^N}\Big(\frac{1}{|x|^{\mu}} \ast W_i^{2_\mu^*}\Big)W_i^{2_\mu^*-2}(\Psi_i\rho)^2\bigg]+o(1)\|\nabla\rho\|_{L^{2}}^2.
    \end{split}
    \end{equation*}
This yields
\begin{equation*}
    \begin{split}
  \int_{\mathbb{R}^N}\bigg[(2_{\mu}^{\ast}-1)\Big(\frac{1}{|x|^{\mu}}\ast W_i^{2_{\mu}^{\ast}}\Big)W_i^{2_{\mu}^{\ast}-2}(\Psi_i\rho)^2+2_{\mu}^{\ast}\Big(\frac{1}{|x|^{\mu}}\ast \big(W_i^{2_{\mu}^{\ast}-1}\Psi_i\rho\big)\Big)W_i^{2_{\mu}^{\ast}-1}\Psi_i\rho\bigg]\leq
\nu&\int_{\mathbb{R}^N}|\nabla(\Psi_i\rho)|^2\\&+o(1)\|\nabla\rho\|_{L^{2}}^2,
    \end{split}
    \end{equation*}
yielding the result.

Now, combining H\"{o}lder and Sobolev inequalities, one can infer from Proposition \ref{fai2} that
\begin{equation}\label{conclude-1}
\begin{split}
\sum_{i=1}^{\kappa}\int_{\mathbb{R}^N}|\nabla(\Psi_i\rho)|^2&\leq\sum_{i=1}^{\kappa}\int_{\mathbb{R}^N}|\nabla \rho|^2\Psi_i^2+\sum_{i=1}^{\kappa}\|\nabla\Psi_i\|_{L^{N}}^2\|\rho\|_{L^{2^{\ast}}}^{2}+\sum_{i=1}^{\nu}\|\nabla\Psi_i\|_{L^{N}}\|\Psi_i\|_{L^{\infty}}\|\rho\|_{L^{2^{\ast}}}\|\nabla\rho\|_{L^{2}}\\&
\leq\int_{\mathbb{R}^N}|\nabla \rho|^2+o(1)\|\nabla\rho\|_{L^{2}}^2.
\end{split}
 \end{equation}
Therefore, putting \eqref{fai3}-\eqref{conclude-1} together it follows that
\begin{equation*}
\begin{split}
&(2_{\mu}^{\ast}-1)\sum\limits_{i=1}^ {\kappa}\big|\alpha_i\big|^{2_{\mu}^{\ast}}\int_{\mathbb{R}^N}\Big(\frac{1}{|x|^{\mu}}\ast W_j^{2_{\mu}^{\ast}}\Big)\sigma^{2_{\mu}^{\ast}-2}\rho^2
+2_{\mu}^{\ast}\sum\limits_{j=1}^{\kappa}\big|\alpha_i\big|^{2_{\mu}^{\ast}-1}\int_{\mathbb{R}^N}\Big(\frac{1}{|x|^{\mu}}\ast\big(\sigma^{2_{\mu}^{\ast}-1}\rho\big)\Big)W_j^{2_{\mu}^{\ast}-1}\rho \\ \leq& \big(1+o(1)\big)\bigg[(2_{\mu}^{\ast}-1)\sum_{i=1}^{\kappa}\int_{\mathbb{R}^N}\Big(\frac{1}{|x|^{\mu}}\ast W_i^{2_{\mu}^{\ast}}\Big)W_i^{2_{\mu}^{\ast}-2}(\Psi_i\rho)^2+2_{\mu}^{\ast}\sum_{i=1}^{\kappa}\int_{\mathbb{R}^N}\Big(\frac{1}{|x|^{\mu}}\ast \big(W_i^{2_{\mu}^{\ast}-1}\Psi_i\rho\big)\Big)W_i^{2_{\mu}^{\ast}-1}\Psi_i\rho\bigg]\\+&o(1)\|\nabla\rho\|_{L^{2}}^2\\ \leq&\nu(N,\kappa)\int_{\mathbb{R}^N}|\nabla\rho|^2.
\end{split}
\end{equation*}
Hence the result easily follow.

    \qed
\subsection{Proof of Proposition \ref{estimate2}.}
For simplicity of notations, we write $\mathcal{W}$ instead of $W_{i}$, $\alpha$ instead of $\alpha_i$ and $\mathcal{Z}$ instead of $\sum\limits_{j\neq i}\alpha_jW_{j}$ respectively and let $\Psi=\Psi_i$ be the bump function from Proposition~\ref{fai2} if there is no possibility of confusion.

Firstly,  we have the identity
\begin{equation}\label{a1}
     \begin{split}
     &\big(\alpha-\alpha^{2\cdot2_{\mu}^{\ast}-1}\big)\Big(\frac{1}{|x|^{\mu}}\ast \mathcal{W}^{2_{\mu}^{\ast}}\Big)\mathcal{W}^{2_{\mu}^{\ast}-1}-2_{\mu}^*\alpha^{2(2_{\mu^*}-1)}\Big(\frac{1}{|x|^{\mu}}\ast \mathcal{W}^{2_{\mu}^{\ast}-1}\mathcal{Z}\Big)\mathcal{W}^{2_{\mu}^{\ast}-1}\\&\hspace{4mm}-\big(2_{\mu}^*-1\big)\alpha^{2(2_{\mu^*}-1)}\Big(\frac{1}{|x|^{\mu}}\ast \mathcal{W}^{2_{\mu}^{\ast}}\Big)\mathcal{W}^{2_{\mu}^{\ast}-2}\mathcal{Z}\\=&
\Delta\rho+\bigg[-\Delta u-\Big(\frac{1}{|x|^{\mu}}\ast |u|^{2_{\mu}^{\ast}}\Big)|u|^{2_{\mu}^{\ast}-2}u\bigg]-\sum\alpha_i\Big(\frac{1}{|x|^{\mu}}\ast W_i^{2_{\mu}^{\ast}}\Big)W_i^{2_{\mu}^{\ast}-1}\\&
+2_{\mu}^{\ast}
\Big(\frac{1}{|x|^{\mu}}\ast \big((\alpha \mathcal{W})^{2_{\mu}^{\ast}-1}\rho\big)\Big)\big(\alpha \mathcal{W}\big)^{2_{\mu}^{\ast}-1}+\big(2_{\mu}^{\ast}-1\big)\Big(\frac{1}{|x|^{\mu}}\ast \big(\alpha \mathcal{W}\big)^{2_{\mu}^{\ast}}\Big)\big(\alpha \mathcal{W}\big)^{2_{\mu}^{\ast}-2}\rho
\\&
+\bigg[\Big(\frac{1}{|x|^{\mu}}\ast |\sigma+\rho|^{2_{\mu}^{\ast}}\Big)|\sigma+\rho|^{2_{\mu}^{\ast}-2}(\sigma+\rho
)-\Big(\frac{1}{|x|^{\mu}}\ast \sigma^{2_{\mu}^{\ast}}\Big)\sigma^{2_{\mu}^{\ast}-1}\bigg]
\\&+\bigg[-2_{\mu}^{\ast}
\Big(\frac{1}{|x|^{\mu}}\ast \big(\sigma^{2_{\mu}^{\ast}-1}\rho\big)\Big)\sigma^{2_{\mu}^{\ast}-1}-\big(2_{\mu}^{\ast}-1\big)\Big(\frac{1}{|x|^{\mu}}\ast \sigma^{2_{\mu}^{\ast}}\Big)\sigma^{2_{\mu}^{\ast}-2}\rho\bigg]
\\&+\bigg[2_{\mu}^{\ast}
\Big(\frac{1}{|x|^{\mu}}\ast \big(\sigma^{2_{\mu}^{\ast}-1}\rho\big)\Big)\sigma^{2_{\mu}^{\ast}-1}+\big(2_{\mu}^{\ast}-1\big)\Big(\frac{1}{|x|^{\mu}}\ast \sigma^{2_{\mu}^{\ast}}\Big)\sigma^{2_{\mu}^{\ast}-2}\rho\bigg]
\\&
+\bigg[-2_{\mu}^{\ast}
\Big(\frac{1}{|x|^{\mu}}\ast \big((\alpha \mathcal{W})^{2_{\mu}^{\ast}-1}\rho\big)\Big)\big(\alpha \mathcal{W}\big)^{2_{\mu}^{\ast}-1}-\big(2_{\mu}^{\ast}-1\big)\Big(\frac{1}{|x|^{\mu}}\ast \big(\alpha \mathcal{W}\big)^{2_{\mu}^{\ast}}\Big)\big(\alpha \mathcal{W}\big)^{2_{\mu}^{\ast}-2}\rho\bigg]
\\&
\end{split}
    \end{equation}
   \begin{equation*}
     \begin{split}
&+\bigg[\Big(\frac{1}{|x|^{\mu}}\ast\big(\alpha \mathcal{W}+\mathcal{Z}\big)^{2_{\mu}^{\ast}}\Big)\big(\alpha \mathcal{W}+\mathcal{Z}\big)^{2_{\mu}^{\ast}-1}-\Big(\frac{1}{|x|^{\mu}}\ast \big(\alpha \mathcal{W}\big)^{2_{\mu}^{\ast}}\Big)\big(\alpha \mathcal{W}\big)^{2_{\mu}^{\ast}-1}\bigg]\\&+\bigg[-2_{\mu}^{\ast}
\Big(\frac{1}{|x|^{\mu}}\ast \big(\alpha \mathcal{W}\big)^{2_{\mu}^{\ast}-1}\mathcal{Z}\Big)\big(\alpha \mathcal{W}\big)^{2_{\mu}^{\ast}-1}-\big(2_{\mu}^{\ast}-1\big)\Big(\frac{1}{|x|^{\mu}}\ast \big(\alpha \mathcal{W}\big)^{2_{\mu}^{\ast}}\Big)\big(\alpha \mathcal{W}\big)^{2_{\mu}^{\ast}-2}\mathcal{Z}\bigg].
\end{split}
    \end{equation*}

Now we evaluate some terms on the right-hand side of \eqref{a1}. Using Proposition~\ref{fai2}-(2), in the region $\{\Psi>0\}$, we can obtain
\begin{equation*}
\sum\limits_{j\neq i}\alpha_j\Big(\frac{1}{|x|^{\mu}}\ast W_j^{2_{\mu}^{\ast}}\Big)W_j^{2_{\mu}^{\ast}-1}=o\Big[\Big(\frac{1}{|x|^{\mu}}\ast \mathcal{W}^{2_{\mu}^{\ast}}\Big) \mathcal{W}^{2_{\mu}^{\ast}-2}\mathcal{Z}\Big].
\end{equation*}
By a straightforward computation
\begin{equation*}
\begin{split}
2_{\mu}^{\ast}&\bigg[
\Big(\frac{1}{|x|^{\mu}}\ast \big(\sigma^{2_{\mu}^{\ast}-1}\rho\big)\Big)\sigma^{2_{\mu}^{\ast}-1}
-\Big(\frac{1}{|x|^{\mu}}\ast \big((\alpha \mathcal{W})^{2_{\mu}^{\ast}-1}\rho\big)\Big)\big(\alpha \mathcal{W}\big)^{2_{\mu}^{\ast}-1}\bigg]\\&+
\big(2_{\mu}^{\ast}-1\big)\bigg[\Big(\frac{1}{|x|^{\mu}}\ast \sigma^{2_{\mu}^{\ast}}\Big)\sigma^{2_{\mu}^{\ast}-2}\rho
-\Big(\frac{1}{|x|^{\mu}}\ast \big(\alpha \mathcal{W}\big)^{2_{\mu}^{\ast}}\Big)\big(\alpha \mathcal{W}\big)^{2_{\mu}^{\ast}-2}\rho\bigg]\\
=&o\bigg[\Big(\frac{1}{|x|^{\mu}}\ast \big(\sigma^{2_{\mu}^{\ast}-1}|\rho|\big)\Big)\mathcal{W}^{2_{\mu}^{\ast}-1}\bigg]+o\bigg[\Big(\frac{1}{|x|^{\mu}}\ast \big(\mathcal{W}^{2_{\mu}^{\ast}-1}|\rho|\big)\Big)\mathcal{W}^{2_{\mu}^{\ast}-1}\bigg]+
o\bigg[\Big(\frac{1}{|x|^{\mu}}\ast \sigma^{2_{\mu}^{\ast}}\Big)\mathcal{W}^{2_{\mu}^{\ast}-2}\rho\bigg]\\&+o\bigg[\Big(\frac{1}{|x|^{\mu}}\ast \mathcal{W}^{2_{\mu}^{\ast}}\Big)W^{2_{\mu}^{\ast}-2}|\rho|\bigg],
\end{split}
\end{equation*}
and
\begin{equation*}
\begin{split}
\Big(\frac{1}{|x|^{\mu}}&\ast\big(\alpha \mathcal{W}+\mathcal{Z}\big)^{2_{\mu}^{\ast}}\Big)\big(\alpha \mathcal{W}+\mathcal{Z}\big)^{2_{\mu}^{\ast}-1}-\Big(\frac{1}{|x|^{\mu}}\ast \big(\alpha \mathcal{W}\big)^{2_{\mu}^{\ast}}\Big)\big(\alpha \mathcal{W}\big)^{2_{\mu}^{\ast}-1}\\&\hspace{4mm}-2_{\mu}^{\ast}
\Big(\frac{1}{|x|^{\mu}}\ast \big((\alpha \mathcal{W})^{2_{\mu}^{\ast}-1}\mathcal{Z}\big)\Big)\big(\alpha \mathcal{W}\big)^{2_{\mu}^{\ast}-1}-\big(2_{\mu}^{\ast}-1\big)\Big(\frac{1}{|x|^{\mu}}\ast \big(\alpha \mathcal{W}\big)^{2_{\mu}^{\ast}}\Big)\big(\alpha \mathcal{W}\big)^{2_{\mu}^{\ast}-2}\mathcal{Z}\\
=&\bigg[\frac{1}{|x|^{\mu}}\ast\Big((\alpha \mathcal{W}+\mathcal{Z})^{2_{\mu}^{\ast}}-(\alpha \mathcal{W})^{{2_{\mu}^{\ast}}}-{2_{\mu}^{\ast}}(\alpha \mathcal{W})^{{2_{\mu}^{\ast}}-1}\mathcal{Z}\Big)\bigg]\big(\alpha \mathcal{W}+\mathcal{Z}\big)^{2_{\mu}^{\ast}-1}\\&+
\Big(\frac{1}{|x|^{\mu}}\ast\big(\alpha \mathcal{W}\big)^{2_{\mu}^{\ast}}\Big)\Big[(\alpha \mathcal{W}+\mathcal{Z})^{2_{\mu}^{\ast}-1}-(\alpha \mathcal{W})^{{2_{\mu}^{\ast}}-1}-({2_{\mu}^{\ast}}-1)(\alpha \mathcal{W})^{{2_{\mu}^{\ast}}-2}\mathcal{Z}\Big]\\
&+2_{\mu}^{\ast}
\Big(\frac{1}{|x|^{\mu}}\ast \big((\alpha \mathcal{W})^{2_{\mu}^{\ast}-1}\mathcal{Z}\big)\Big)\Big[(\alpha \mathcal{W}+\mathcal{Z})^{2_{\mu}^{\ast}-1}-(\alpha \mathcal{W})^{{2_{\mu}^{\ast}}-1}-(2_{\mu}^{\ast}-1)(\alpha \mathcal{W})^{{2_{\mu}^{\ast}}-2}\mathcal{Z}\Big]\\&+2_{\mu}^{\ast}(2_{\mu}^{\ast}-1)\Big(\frac{1}{|x|^{\mu}}\ast\big((\alpha \mathcal{W})^{2_{\mu}^{\ast}-1}\mathcal{Z}\big)\Big)\big(\alpha \mathcal{W}\big)^{2_{\mu}^{\ast}-2}\mathcal{Z}\\
=&o\bigg[\Big(\frac{1}{|x|^{\mu}}\ast\big(\mathcal{W}^{2_{\mu}^{\ast}-1}\mathcal{Z}\big)\Big) \mathcal{W}^{2_{\mu}^{\ast}-1}\bigg]+
o\bigg[\Big(\frac{1}{|x|^{\mu}}\ast \mathcal{W}^{2_{\mu}^{\ast}}\Big) \mathcal{W}^{2_{\mu}^{\ast}-2}\mathcal{Z}\bigg].
\end{split}
\end{equation*}
Furthermore, by applying an elementary inequality, we have that
\begin{equation*}
\begin{split}
&\Big(\frac{1}{|x|^{\mu}}\ast |\sigma+\rho|^{2_{\mu}^{\ast}}\Big)|\sigma+\rho|^{2_{\mu}^{\ast}-2}(\sigma+\rho
)-\Big(\frac{1}{|x|^{\mu}}\ast \sigma^{2_{\mu}^{\ast}}\Big)\sigma^{2_{\mu}^{\ast}-1}-2_{\mu}^{\ast}
\Big(\frac{1}{|x|^{\mu}}\ast \big(\sigma^{2_{\mu}^{\ast}-1}\rho\big)\Big)\sigma^{2_{\mu}^{\ast}-1}\\&\hspace{4mm}-\big(2_{\mu}^{\ast}-1\big)\Big(\frac{1}{|x|^{\mu}}\ast \sigma^{2_{\mu}^{\ast}}\Big)\sigma^{2_{\mu}^{\ast}-2}\rho\\&
\lesssim
\Big(\frac{1}{|x|^{\mu}}\ast \rho^{2_{\mu}^{\ast}}\Big)\rho^{2_{\mu}^{\ast}-1}+\Big(\frac{1}{|x|^{\mu}}\ast \sigma^{2_{\mu}^{\ast}}\Big)\rho^{2_{\mu}^{\ast}-1}+\Big(\frac{1}{|x|^{\mu}}\ast \sigma^{2_{\mu}^{\ast}}\Big)\sigma^{2_{\mu}^{\ast}-3}\rho^2+\Big(\frac{1}{|x|^{\mu}}\ast \big(\sigma^{2_{\mu}^{\ast}-1}\rho\big)\Big)\sigma^{2_{\mu}^{\ast}-2}\rho\\&
\end{split}
    \end{equation*}
   \begin{equation*}
     \begin{split}
&+\Big(\frac{1}{|x|^{\mu}}\ast\big(\sigma^{2_{\mu}^{\ast}-1}\rho\big)\Big)\sigma^{2_{\mu}^{\ast}-3}\rho^2+
\Big(\frac{1}{|x|^{\mu}}\ast \big(\sigma^{2_{\mu}^{\ast}-1}\rho\big)\Big)\rho^{2_{\mu}^{\ast}-1}+\Big(\frac{1}{|x|^{\mu}}\ast \big(\sigma^{2_{\mu}^{\ast}-2}\rho^2)\Big)\sigma^{2_{\mu}^{\ast}-2}\rho\\&+\Big(\frac{1}{|x|^{\mu}}\ast \rho^{2_{\mu}^{\ast}}\Big)\sigma^{2_{\mu}^{\ast}-3}\rho^2+\Big(\frac{1}{|x|^{\mu}}\ast \big(\sigma^{2_{\mu}^{\ast}-2}\rho^2\big)\Big)\sigma^{2_{\mu}^{\ast}-3}\rho^2+\Big(\frac{1}{|x|^{\mu}}\ast \big(\sigma^{2_{\mu}^{\ast}-2}\rho^2\big)\Big)\rho^{2_{\mu}^{\ast}-1}\\&+\Big(\frac{1}{|x|^{\mu}}\ast \rho^{2_{\mu}^{\ast}}\Big)\sigma^{2_{\mu}^{\ast}-1}+\Big(\frac{1}{|x|^{\mu}}\ast \rho^{2_{\mu}^{\ast}}\Big)\sigma^{2_{\mu}^{\ast}-2}\rho\\&=:\Gamma(\rho).
\end{split}
\end{equation*}
Therefore, combining \eqref{a1} and the preceding estimates, we obtain the estimate
\begin{equation}\label{a2}
     \begin{split}
     &\big(\alpha-\alpha^{2\cdot2_{\mu}^{\ast}-1}\big)\Big(\frac{1}{|x|^{\mu}}\ast \mathcal{W}^{2_{\mu}^{\ast}}\Big)\mathcal{W}^{2_{\mu}^{\ast}-1}-\big(2_{\mu}^*\alpha^{2(2_{\mu^*}-1)}+o(1)\big)\Big(\frac{1}{|x|^{\mu}}\ast \big(\mathcal{W}^{2_{\mu}^{\ast}-1}\mathcal{Z}\big)\Big)\mathcal{W}^{2_{\mu}^{\ast}-1}\\&\hspace{4mm}-\big(\big(2_{\mu}^*-1\big)\alpha^{2(2_{\mu^*}-1)}+o(1)\big)\Big(\frac{1}{|x|^{\mu}}\ast \mathcal{W}^{2_{\mu}^{\ast}}\Big)\mathcal{W}^{2_{\mu}^{\ast}-2}\mathcal{Z}-
\Delta\rho-\left[-\Delta u-\Big(\frac{1}{|x|^{\mu}}\ast |u|^{2_{\mu}^{\ast}}\Big)|u|^{2_{\mu}^{\ast}-2}u\right]\\&
\hspace{4mm}-2_{\mu}^{\ast}
\Big(\frac{1}{|x|^{\mu}}\ast \big((\alpha \mathcal{W})^{2_{\mu}^{\ast}-1}\rho\big)\Big)\big(\alpha \mathcal{W}\big)^{2_{\mu}^{\ast}-1}-\big(2_{\mu}^{\ast}-1\big)\Big(\frac{1}{|x|^{\mu}}\ast \big(\alpha \mathcal{W}\big)^{2_{\mu}^{\ast}}\Big)\big(\alpha \mathcal{W}\big)^{2_{\mu}^{\ast}-2}\rho\\
&\lesssim \Gamma(\rho)+o\Big[\Big(\frac{1}{|x|^{\mu}}\ast \big(\sigma^{2_{\mu}^{\ast}-1}|\rho|\big)\Big)\mathcal{W}^{2_{\mu}^{\ast}-1}\Big]+o\Big[\Big(\frac{1}{|x|^{\mu}}\ast \big(\mathcal{W}^{2_{\mu}^{\ast}-1}|\rho|\big)\Big)\mathcal{W}^{2_{\mu}^{\ast}-1}\Big]\\&\hspace{4mm}+
o\Big[\Big(\frac{1}{|x|^{\mu}}\ast \sigma^{2_{\mu}^{\ast}}\Big)\mathcal{W}^{2_{\mu}^{\ast}-2}|\rho|\Big]+o\Big[\Big(\frac{1}{|x|^{\mu}}\ast \mathcal{W}^{2_{\mu}^{\ast}}\Big)\mathcal{W}^{2_{\mu}^{\ast}-2}|\rho|\Big].
\end{split}
\end{equation}

Now by letting $\eta$ be either $\mathcal{W}$ or $\frac{\partial \mathcal{W}}{\partial\lambda}$, we see that
\begin{equation}\label{a3}
\begin{split}
&\hspace{4mm}\int_{\mathbb{R}^N}\nabla\rho\nabla \eta=0,\\
\hspace{4mm}&2_\mu^*\int_{\mathbb{R}^N}\Big(\frac{1}{|x|^{\mu}}\ast \big(\mathcal{W}^{2_\mu^*-1}\rho\big)\Big)\mathcal{W}^{2_\mu^*-1}\eta
+(2_\mu^*-1)\int_{\mathbb{R}^N}\Big(\frac{1}{|x|^{\mu}} \ast \mathcal{W}^{2_\mu^*}\Big)\mathcal{W}^{2_\mu^*-2}\eta\rho=0,
\end{split}
\end{equation}
by orthogonality conditions. By testing \eqref{a2} with $\eta\Psi$, we obtain
\begin{equation}\label{a4}
\begin{split}
&\bigg|\int_{\mathbb{R}^N}\bigg[\big(\alpha-\alpha^{2\cdot2_{\mu}^{\ast}-1}\big)\Big(\frac{1}{|x|^{\mu}}\ast \mathcal{W}^{2_{\mu}^{\ast}}\Big)\mathcal{W}^{2_{\mu}^{\ast}-1}-\big(2_{\mu}^*\alpha^{2(2_{\mu^*}-1)}+o(1)\big)\Big(\frac{1}{|x|^{\mu}}\ast \big(\mathcal{W}^{2_{\mu}^{\ast}-1}\mathcal{Z}\big)\Big)\mathcal{W}^{2_{\mu}^{\ast}-1}\\&\hspace{4mm}-\big(\big(2_{\mu}^*-1\big)\alpha^{2(2_{\mu^*}-1)}+o(1)\big)\Big(\frac{1}{|x|^{\mu}}\ast \mathcal{W}^{2_{\mu}^{\ast}}\Big)\mathcal{W}^{2_{\mu}^{\ast}-2}V\bigg]\eta\Psi\bigg|\\& \lesssim
\Big|\int_{\mathbb{R}^N}\Gamma(\rho)\eta\Psi\Big|+\Big|\int_{\mathbb{R}^N}\nabla\rho\nabla(\eta\Psi)\Big|+\bigg|\int_{\mathbb{R}^N}\Big[-\Delta u-\Big(\frac{1}{|x|^{\mu}}\ast |u|^{2_{\mu}^{\ast}}\Big)|u|^{2_{\mu}^{\ast}-2}u\Big]\eta\Psi \bigg|\\&\hspace{4mm}
+\bigg|2_{\mu}^*\int_{\mathbb{R}^N}\Big(\frac{1}{|x|^{\mu}}\ast \big( \mathcal{W}^{2_{\mu}^{\ast}-1}\rho\big)\Big) \mathcal{W}^{2_{\mu}^{\ast}-1}\eta\Psi+(2_{\mu}^*-1)\int_{\mathbb{R}^N}\Big(\frac{1}{|x|^{\mu}}\ast  \mathcal{W}^{2_{\mu}^{\ast}}\Big)\mathcal{W}^{2_{\mu}^{\ast}-2}\rho\eta\Psi\bigg|
\\&\hspace{4mm}+o\Big[\int_{\mathbb{R}^N}\Big(\frac{1}{|x|^{\mu}}\ast \big(\sigma^{2_{\mu}^{\ast}-1}|\rho|\big)\Big)\mathcal{W}^{2_{\mu}^{\ast}-1}|\eta|\Psi\Big]+o\Big[\int_{\mathbb{R}^N}\Big(\frac{1}{|x|^{\mu}}\ast \big(\mathcal{W}^{2_{\mu}^{\ast}-1}|\rho|\big)\Big)\mathcal{W}^{2_{\mu}^{\ast}-1}|\eta|\Psi\Big]\\&\hspace{4mm}+
o\Big[\int_{\mathbb{R}^N}\Big(\frac{1}{|x|^{\mu}}\ast \sigma^{2_{\mu}^{\ast}}\Big)\mathcal{W}^{2_{\mu}^{\ast}-2}|\rho||\eta|\Psi\Big]+o\Big[\int_{\mathbb{R}^N}\Big(\frac{1}{|x|^{\mu}}\ast \mathcal{W}^{2_{\mu}^{\ast}}\Big)\mathcal{W}^{2_{\mu}^{\ast}-2}|\rho||\eta|\Psi\Big].
\end{split}
\end{equation}
Let us estimate each term of the righ hand side of \eqref{a4}.
\begin{equation*}
\Big|\int_{\mathbb{R}^N}\nabla\rho\nabla(\eta\Psi)\Big|=\Big|\int_{\mathbb{R}^N}\nabla\rho\nabla\big(\eta(\Psi-1)\big)\Big|\leq
\big\|\nabla\rho\big\|_{L^2}\big\|\nabla\big(\eta(\Psi-1)\big\|_{L^2},
\end{equation*}
\begin{equation*}
\bigg|\int_{\mathbb{R}^N}\Big[-\Delta u-\Big(\frac{1}{|x|^{\mu}}\ast |u|^{2_{\mu}^{\ast}}\Big)|u|^{2_{\mu}^{\ast}-2}u\Big]\eta\Psi \bigg|\leq\Big\|\Delta u+\Big(\frac{1}{|x|^{\mu}}\ast |u|^{2_{\mu}^{\ast}}\Big)|u|^{2_{\mu}^{\ast}-2}u\Big\|_{(\mathcal{D}^{1,2}(\mathbb{R}^N))^{-1}}\big\|\nabla(\eta\Psi)\big\|_{L^2}.
\end{equation*}
Combining \eqref{a3} with HLS and Sobolev inequalities, we get that
\begin{equation*}
\begin{split}
\bigg|2_{\mu}^*\int_{\mathbb{R}^N}&\Big(\frac{1}{|x|^{\mu}}\ast \big( \mathcal{W}^{2_{\mu}^{\ast}-1}\rho\big)\Big) \mathcal{W}^{2_{\mu}^{\ast}-1}\eta\Psi+(2_{\mu}^*-1)\int_{\mathbb{R}^N}\Big(\frac{1}{|x|^{\mu}}\ast  \mathcal{W}^{2_{\mu}^{\ast}}\Big)\mathcal{W}^{2_{\mu}^{\ast}-2}\rho\eta\Psi\bigg|\\
=&\bigg|2_{\mu}^*\int_{\mathbb{R}^N}\Big(\frac{1}{|x|^{\mu}}\ast \big( \mathcal{W}^{2_{\mu}^{\ast}-1}\rho\big)\Big) \mathcal{W}^{2_{\mu}^{\ast}-1}\eta(\Psi-1)+(2_{\mu}^*-1)\int_{\mathbb{R}^N}\Big(\frac{1}{|x|^{\mu}}\ast  \mathcal{W}^{2_{\mu}^{\ast}}\Big)\mathcal{W}^{2_{\mu}^{\ast}-2}\rho\eta(\Psi-1)\bigg|\\
\lesssim&\big\|\mathcal{W}\big\|_{L^{2^{*}}}^{\frac{N-\mu+2}{N-2}}\|\nabla\rho\|_{L^2}\Big(\int_{\{\Psi<1\}}\big(\mathcal{W}^{2_{\mu}^{*}-1}|\eta|\big)^{\frac{2N}{2N-\mu}}\Big)^{\frac{2N-\mu}{2N}}
+\big\|\mathcal{W}\big\|_{L^{2^{*}}}^{\frac{2N-\mu}{N-2}}\|\nabla\rho\|_{L^2}\Big(\int_{\{\Psi<1\}}\big(\mathcal{W}^{2_{\mu}^{*}-2}|\eta|\big)^{\frac{2N}{N-\mu+2}}\Big)^{\frac{N-\mu+2}{2N}},
\end{split}
\end{equation*}
\begin{equation*}
\begin{split}
\int_{\mathbb{R}^N}\Big(\frac{1}{|x|^{\mu}}\ast \big(\sigma^{2_{\mu}^{\ast}-1}|\rho|\big)\Big)\mathcal{W}^{2_{\mu}^{\ast}-1}|\eta|\Psi\lesssim\big\|\sigma\big\|_{L^{2^{*}}}^{\frac{N-\mu+2}{N-2}}\big\|\mathcal{W}^{2_{\mu}^*-1}\eta\big\|_{L^{\frac{2N}{2N-\mu}}}\|\nabla\rho\|_{L^2},
\end{split}
\end{equation*}
\begin{equation*}
\int_{\mathbb{R}^N}\Big(\frac{1}{|x|^{\mu}}\ast \big(\mathcal{W}^{2_{\mu}^{\ast}-1}|\rho|\big)\Big)\mathcal{W}^{2_{\mu}^{\ast}-1}|\eta|\Psi\lesssim\big\|\mathcal{W}\big\|_{L^{2^{*}}}^{\frac{N-\mu+2}{N-2}}\big\|\mathcal{W}^{2_{\mu}^*-1}\eta\big\|_{L^{\frac{2N}{2N-\mu}}}\|\nabla\rho\|_{L^2},
\end{equation*}
\begin{equation*}
\int_{\mathbb{R}^N}\Big(\frac{1}{|x|^{\mu}}\ast \sigma^{2_{\mu}^{\ast}}\Big)\mathcal{W}^{2_{\mu}^{\ast}-2}|\rho||\eta|\Psi\lesssim\big\|\sigma\big\|_{L^{2^{*}}}^{\frac{2N-\mu}{N-2}}\big\|\mathcal{W}^{2_{\mu}^*-2}\eta\big\|_{L^{\frac{2N}{N-\mu+2}}}\|\nabla\rho\|_{L^2},
\end{equation*}
\begin{equation*}
\int_{\mathbb{R}^N}\Big(\frac{1}{|x|^{\mu}}\ast \mathcal{W}^{2_{\mu}^{\ast}}\Big)\mathcal{W}^{2_{\mu}^{\ast}-2}|\rho||\eta|\Psi\lesssim\big\|\mathcal{W}\big\|_{L^{2^{*}}}^{\frac{2N-\mu}{N-2}}\big\|\mathcal{W}^{2_{\mu}^*-2}\eta\big\|_{L^{\frac{2N}{N-\mu+2}}}\|\nabla\rho\|_{L^2}.
\end{equation*}
It remains to estimate $\Big|\int_{\mathbb{R}^N}\Gamma(\rho)\eta\Psi\Big|$. Then by applying HLS inequality and Sobolev inequality to all terms, we have that
\begin{equation*}
\Big|\Big(\frac{1}{|x|^{\mu}}\ast \rho^{2_{\mu}^{\ast}}\Big)\rho^{2_{\mu}^{\ast}-1}\eta\Psi\Big|\lesssim\big\|\nabla\rho\big\|_{L^2}^{\frac{3N-2\mu+2}{N-2}}\big\|\eta\big\|_{L^{2^*}},
\end{equation*}
\begin{equation*}
\begin{split}
&\bigg|\int_{\mathbb{R}^N}\Big(\frac{1}{|x|^{\mu}}\ast \sigma^{2_{\mu}^{\ast}}\Big)\rho^{2_{\mu}^{\ast}-1}\eta\Psi+\int_{\mathbb{R}^N}\Big(\frac{1}{|x|^{\mu}}\ast \sigma^{2_{\mu}^{\ast}}\Big)\sigma^{2_{\mu}^{\ast}-3}\rho^2\eta\Psi\bigg|\\&\lesssim\big\|\sigma\big\|_{L^{2^*}}^{\frac{2N-\mu}{N-2}}\big\|\nabla\rho\big\|_{L^2}^{\frac{N-\mu+2}{N-2}}\big\|\eta\big\|_{L^{2^*}}
+\big\|\sigma\big\|_{L^{2^*}}^{\frac{N-2\mu+6}{N-2}}\big\|\nabla\rho\big\|_{L^2}^{2}\big\|\eta\big\|_{L^{2^*}},
\end{split}
\end{equation*}

\begin{equation*}
\begin{split}
&\bigg|\int_{\mathbb{R}^N}\Big[\Big(\frac{1}{|x|^{\mu}}\ast \rho^{2_{\mu}^{\ast}}\Big)\sigma^{2_{\mu}^{\ast}-3}\rho^2+\Big(\frac{1}{|x|^{\mu}}\ast \rho^{2_{\mu}^{\ast}}\Big)\sigma^{2_{\mu}^{\ast}-1}+\Big(\frac{1}{|x|^{\mu}}\ast \rho^{2_{\mu}^{\ast}}\Big)\sigma^{2_{\mu}^{\ast}-2}\rho\Big]\eta\Psi\bigg|\\&
\lesssim\big\|\sigma\big\|_{L^{2^{*}}}^{2_{\mu}^*-3}\big\|\nabla\rho\big\|_{L^2}^{2_{\mu}^*+2}\big\|\eta\big\|_{L^{2^*}}
+\big\|\sigma\big\|_{L^{2^{*}}}^{2_{\mu}^*-1}\big\|\nabla\rho\big\|_{L^2}^{2_{\mu}^{*}}\big\|\eta\big\|_{L^{2^*}}
+\big\|\sigma\big\|_{L^{2^{*}}}^{2_{\mu}^*-2}\big\|\nabla\rho\big\|_{L^2}^{2_{\mu}^*+1}\big\|\eta\big\|_{L^{2^*}},
\end{split}
\end{equation*}

\begin{equation*}
\begin{split}
&\bigg|\int_{\mathbb{R}^N}\Big[\Big(\frac{1}{|x|^{\mu}}\ast \big(\sigma^{2_{\mu}^{\ast}-1}\rho\big)\Big)\sigma^{2_{\mu}^{\ast}-2}\rho
+\Big(\frac{1}{|x|^{\mu}}\ast\big(\sigma^{2_{\mu}^{\ast}-1}\rho\big)\Big)\sigma^{2_{\mu}^{\ast}-3}\rho^2+
\Big(\frac{1}{|x|^{\mu}}\ast\big(\sigma^{2_{\mu}^{\ast}-1}\rho\big)\Big)\rho^{2_{\mu}^{\ast}-1}\Big]\eta\Psi\bigg|\\&
\lesssim\big\|\sigma\big\|_{L^{2^*}}^{\frac{N-2\mu+6}{N-2}}\big\|\nabla\rho\big\|_{L^2}^{2}\big\|\eta\big\|_{L^{2^*}}
+\big\|\sigma\big\|_{L^{2^*}}^{\frac{8-2\mu}{N-2}}\big\|\nabla\rho\big\|_{L^2}^{3}\big\|\eta\big\|_{L^{2^*}}
+\big\|\sigma\big\|_{L^{2^*}}^{\frac{N-\mu+2}{N-2}}\big\|\nabla\rho\big\|_{L^2}^{\frac{2N-\mu}{N-2}}\big\|\eta\big\|_{L^{2^*}},
\end{split}
\end{equation*}
and
\begin{equation*}
\begin{split}
&\bigg|\int_{\mathbb{R}^N}\Big[\Big(\frac{1}{|x|^{\mu}}\ast \big(\sigma^{2_{\mu}^{\ast}-2}\rho^2)\Big)\sigma^{2_{\mu}^{\ast}-2}\rho+\Big(\frac{1}{|x|^{\mu}}\ast \big(\sigma^{2_{\mu}^{\ast}-2}\rho^2\big)\Big)\sigma^{2_{\mu}^{\ast}-3}\rho^2+\Big(\frac{1}{|x|^{\mu}}\ast \big(\sigma^{2_{\mu}^{\ast}-2}\rho^2\big)\Big)\rho^{2_{\mu}^{\ast}-1}\Big]\eta\Psi\bigg|\\&
\lesssim\big\|\sigma\big\|_{L^{2^*}}^{2(2_{\mu}^*-2)}\big\|\nabla\rho\big\|_{L^2}^{3}\big\|\eta\big\|_{L^{2^*}}
+\big\|\sigma\big\|_{L^{2^*}}^{2\cdot2_{\mu}^{\ast}-5}\big\|\nabla\rho\big\|_{L^2}^{4}\big\|\eta\big\|_{L^{2^*}}
+\big\|\sigma\big\|_{L^{2^*}}^{2_{\mu}^{\ast}-2}\big\|\nabla\rho\big\|_{L^2}^{2_{\mu}^*+1}\big\|\eta\big\|_{L^{2^*}}.
\end{split}
\end{equation*}
Moreover, by Proposition \ref{fai2} and $|\eta|\lesssim \mathcal{W}$, we get that
\begin{equation*}
\big\|\nabla\big(\eta(\Psi-1)\big\|_{L^2}=o(1),\quad\big\|\nabla(\eta\Psi)\big\|_{L^2}\lesssim1,
\end{equation*}
\begin{equation*}
\big\|\mathcal{W}\big\|_{L^{2^{*}}}^{\frac{N-\mu+2}{N-2}}\Big(\int_{\{\Psi<1\}}\big(\mathcal{W}^{2_{\mu}^{*}-1}|\eta|\big)^{\frac{2N}{2N-\mu}}\Big)^{\frac{2N-\mu}{2N}}=o(1),
\end{equation*}
\begin{equation*}
\big\|\mathcal{W}\big\|_{L^{2^{*}}}^{\frac{2N-\mu}{N-2}}\Big(\int_{\{\Psi<1\}}\big(\mathcal{W}^{2_{\mu}^{*}-2}|\eta|\big)^{\frac{2N}{N-\mu+2}}\Big)^{\frac{N-\mu+2}{2N}}=o(1),
\end{equation*}
\begin{equation*}
\big\|\sigma\big\|_{L^{2^{*}}}^{\frac{N-\mu+2}{N-2}}\big\|\mathcal{W}^{2_{\mu}^*-1}\eta\big\|_{L^{\frac{2N}{2N-\mu}}} \lesssim1,\quad
 \big\|\mathcal{W}\big\|_{L^{2^{*}}}^{\frac{N-\mu+2}{N-2}}\big\|\mathcal{W}^{2_{\mu}^*-1}\eta\big\|_{L^{\frac{2N}{2N-\mu}}}\lesssim1,
\end{equation*}
\begin{equation*}
\big\|\sigma\big\|_{L^{2^{*}}}^{\frac{2N-\mu}{N-2}}\big\|\mathcal{W}^{2_{\mu}^*-2}\eta\big\|_{L^{\frac{2N}{N-\mu+2}}}\lesssim1,
\quad\big\|\mathcal{W}\big\|_{L^{2^{*}}}^{\frac{2N-\mu}{N-2}}\big\|\mathcal{W}^{2_{\mu}^*-2}\eta\big\|_{L^{\frac{2N}{N-\mu+2}}}\lesssim1,
\end{equation*}
 \begin{equation*}
\big\|\eta\big\|_{L^{2^*}}\lesssim1, \quad
\big\|\sigma\big\|_{L^{2^*}}^{\frac{2N-\mu}{N-2}}\big\|\eta\big\|_{L^{2^*}}\lesssim1
,\quad\big\|\sigma\big\|_{L^{2^*}}^{\frac{N-2\mu+6}{N-2}}\big\|\eta\big\|_{L^{2^*}}\lesssim1,
\end{equation*}
 \begin{equation*}
\big\|\sigma\big\|_{L^{2^{*}}}^{\frac{6-N-\mu}{N-2}}\big\|\eta\big\|_{L^{2^*}}
\lesssim1,\quad\big\|\sigma\big\|_{L^{2^{*}}}^{\frac{N-\mu+2}{N-2}}\big\|\eta\big\|_{L^{2^*}}
\lesssim1,\quad\big\|\sigma\big\|_{L^{2^{*}}}^{\frac{4-\mu}{N-2}}\big\|\eta\big\|_{L^{2^*}}\lesssim1,
\end{equation*}
\begin{equation*}
\big\|\sigma\big\|_{L^{2^*}}^{\frac{N-2\mu+6}{N-2}}\big\|\eta\big\|_{L^{2^*}}\lesssim1,\quad
\big\|\sigma\big\|_{L^{2^*}}^{\frac{8-2\mu}{N-2}},\big\|\eta\big\|_{L^{2^*}}\lesssim1
,\quad\big\|\sigma\big\|_{L^{2^*}}^{\frac{N-\mu+2}{N-2}}\big\|\eta\big\|_{L^{2^*}}\lesssim1,
\end{equation*}
and
\begin{equation*}
\big\|\sigma\big\|_{L^{2^*}}^{2(2_{\mu}^*-2)}\big\|\eta\big\|_{L^{2^*}}\lesssim1,
\quad\big\|\sigma\big\|_{L^{2^*}}^{2\cdot2_{\mu}^{\ast}-5}\big\|\eta\big\|_{L^{2^*}}\lesssim1
\quad,\big\|\sigma\big\|_{L^{2^*}}^{2_{\mu}^{\ast}-2}\big\|\eta\big\|_{L^{2^*}}\lesssim1.
\end{equation*}
Therefore, combining \eqref{a4} and the above all estimates, we eventually have
\begin{equation}\label{a5}
\begin{split}
&\Bigg|\int_{\mathbb{R}^N}\bigg[\big(\alpha-\alpha^{2\cdot2_{\mu}^{\ast}-1}\big)\Big(\frac{1}{|x|^{\mu}}\ast \mathcal{W}^{2_{\mu}^{\ast}}\Big)\mathcal{W}^{2_{\mu}^{\ast}-1}-\Big(2_{\mu}^*\alpha^{2(2_{\mu}^\ast-1)}+o(1)\Big)\Big(\frac{1}{|x|^{\mu}}\ast \big(\mathcal{W}^{2_{\mu}^{\ast}-1}\mathcal{Z}\big)\Big)\mathcal{W}^{2_{\mu}^{\ast}-1}\\&\hspace{4mm}-\Big(\big(2_{\mu}^*-1\big)\alpha^{2(2_{\mu}^\ast-1)}+o(1)\Big)\Big(\frac{1}{|x|^{\mu}}\ast \mathcal{W}^{2_{\mu}^{\ast}}\Big)\mathcal{W}^{2_{\mu}^{\ast}-2}\mathcal{Z}\bigg]\eta\Psi\Bigg|\\&
\lesssim o(1)\big\|\nabla\rho\big\|_{L^2}+\Big\|\Delta u+\Big(\frac{1}{|x|^{\mu}}\ast |u|^{2_{\mu}^{\ast}}\Big)|u|^{2_{\mu}^{\ast}-2}u\Big\|_{(\mathcal{D}^{1,2}(\mathbb{R}^N))^{-1}}+\big\|\nabla\rho\big\|_{L^2}^{\min\big(2,\frac{N-\mu+2}{N-2}\big)}.
\end{split}
\end{equation}

Now let us denote $\mathcal{Z}=\mathcal{Z}_1+\mathcal{Z}_2$ where $\mathcal{Z}_1:=\sum\limits_{j<i}\alpha_jW_j$ and $\mathcal{Z}_2:=\sum\limits_{j>i}\alpha_jW_j$.
Then by applying induction we may assume that the statement of the Proposition holds for all $j<i$, we are going to show that  Proposition holds for all $j>i$.

By noting that
\begin{equation*}
\int_{\mathbb{R}^N}\Big(\frac{1}{|x|^{\mu}}\ast W_i^{2_{\mu}^{\ast}}\Big)W_i^{2_{\mu}^{\ast}-1}W_j=\int_{\mathbb{R}^N}\nabla W_i\nabla W_j=\int_{\mathbb{R}^N}\Big(\frac{1}{|x|^{\mu}}\ast W_j^{2_{\mu}^{\ast}}\Big)W_j^{2_{\mu}^{\ast}-1}W_i,
\end{equation*}
then we get that
\begin{equation}\label{ZW1}
\begin{split}
\int_{\mathbb{R}^N}&\Big(\frac{1}{|x|^{\mu}}\ast W^{2_{\mu}^{\ast}}\Big)W^{2_{\mu}^{\ast}-2}\mathcal{Z}_1|\eta|\Psi\lesssim\int_{\mathbb{R}^N}\Big(\frac{1}{|x|^{\mu}}\ast W^{2_{\mu}^{\ast}}\Big)W^{2_{\mu}^{\ast}-1}\mathcal{Z}_1\Psi\\&\lesssim \epsilon\big\|\nabla\rho\big\|_{L^2}+\Big\|\Delta u+\Big(\frac{1}{|x|^{\mu}}\ast |u|^{2_{\mu}^{\ast}}\Big)|u|^{2_{\mu}^{\ast}-2}u\Big\|_{(\mathcal{D}^{1,2}(\mathbb{R}^N))^{-1}}+\big\|\nabla\rho\big\|_{L^2}^{\min\big(2,\frac{N-\mu+2}{N-2}\big)},
\end{split}
\end{equation}
by $|\eta|\leq \mathcal{W}$.

Now we are in a position to prove (\ref{www}) of Proposition \ref{estimate2}. If $\alpha=1$, the proof is completed. So we may assume that
 $\alpha\neq1$.
Let us start with the left-hand side of inequality (\ref{a5}).
Notice that, by Proposition \ref{fai2}-(4), we get that
\begin{equation}\label{a6}
\begin{split}
\int_{\mathbb{R}^N}&\bigg[\Big(\frac{1}{|x|^{\mu}}\ast \mathcal{W}^{2_{\mu}^{\ast}}\Big)\mathcal{W}^{2_{\mu}^{\ast}-1}-\big(1+o(1)\big)\hat{\xi}\Big(\frac{1}{|x|^{\mu}}\ast \mathcal{W}^{2_{\mu}^{\ast}}\Big)\mathcal{W}^{2_{\mu}^{\ast}-2}\bigg]\eta\Psi\\&
=\int_{\mathbb{R}^N}\Big(\frac{1}{|x|^{\mu}}\ast \mathcal{W}^{2_{\mu}^{\ast}}\Big)\mathcal{W}^{2_{\mu}^{\ast}-1}\eta-\hat{\xi}\int_{\mathbb{R}^N}\Big(\frac{1}{|x|^{\mu}}\ast \mathcal{W}^{2_{\mu}^{\ast}}\Big)\mathcal{W}^{2_{\mu}^{\ast}-2}\eta+o(1),
\end{split}
\end{equation}
where
\begin{equation*}
\quad \hat{\xi}:=\frac{(2_{\mu}^*-1)\alpha^{2(2_{\mu}^\ast-1)}\mathcal{Z}_2(0)}{\alpha-\alpha^{2\cdot2_{\mu}^{\ast}-1}}. \end{equation*}
Furthermore, choosing $\eta=W$ and $\eta=\partial_{\lambda}\mathcal{W}$, respectively, we have that
\begin{equation*}
\begin{split}
\frac{\int_{\mathbb{R}^N}\Big(\frac{1}{|x|^{\mu}}\ast \mathcal{W}^{2_{\mu}^{\ast}}\Big)\mathcal{W}^{2_{\mu}^{\ast}}}{\int_{\mathbb{R}^N}\Big(\frac{1}{|x|^{\mu}}\ast \mathcal{W}^{2_{\mu}^{\ast}}\Big)\mathcal{W}^{2_{\mu}^{\ast}-1}}\neq\frac{\int_{\mathbb{R}^N}\Big(\frac{1}{|x|^{\mu}}\ast \mathcal{W}^{2_{\mu}^{\ast}}\Big)\mathcal{W}^{2_{\mu}^{\ast}-1}\partial_{\lambda}\mathcal{W}}{\int_{\mathbb{R}^N}\Big(\frac{1}{|x|^{\mu}}\ast \mathcal{W}^{2_{\mu}^{\ast}}\Big)\mathcal{W}^{2_{\mu}^{\ast}-2}\partial_{\lambda}\mathcal{W}},
\end{split}
\end{equation*}
where we have used the fact that
\begin{equation*}
\int_{\mathbb{R}^N}\Big(\frac{1}{|x|^{\mu}}\ast \mathcal{W}^{2_{\mu}^{\ast}}\Big)\mathcal{W}^{2_{\mu}^{\ast}-1}\partial_{\lambda}\mathcal{W}=\int_{\mathbb{R}^N}\nabla \mathcal{W}\nabla\partial_{\lambda}\mathcal{W}=0.
\end{equation*}
By a straight computation,
\begin{equation*}
\int_{\mathbb{R}^N}\Big(\frac{1}{|x|^{\mu}}\ast \mathcal{W}^{2_{\mu}^{\ast}}\Big)\mathcal{W}^{2_{\mu}^{\ast}-2}\partial_{\lambda}\mathcal{W}
=\frac{\mathcal{\widetilde{Q}}}{2^{\ast}-1}\frac{d}{d\lambda}\Big|_{\lambda=1}\Big(\frac{1}{\lambda^{N-2/2}}\int_{\mathbb{R}^N} \mathcal{W}[0,1]^{2^{\ast}-1}\Big)\neq0.
\end{equation*}
The above arguments tell us that equality (\ref{a6}) cannot be very small and can be maximized by choosing $\eta$.
 Therefore, by the fact that $\mathcal{Z}_2(x)\Psi(x)=(1+o(1))\mathcal{Z}_2(0)\Psi(x)$, Proposition~\ref{fai2}-(2) and (\ref{a5})-(\ref{ZW1}) imply that
\begin{equation}\label{a7}
\begin{split}
&\Big|\alpha-\alpha^{2\cdot2_{\mu}^{\ast}-1}\Big|\Bigg|\int_{\mathbb{R}^N}\bigg[\Big(\frac{1}{|x|^{\mu}}\ast \mathcal{W}^{2_{\mu}^{\ast}}\Big)\mathcal{W}^{2_{\mu}^{\ast}-1}-\big(1+o(1)\big)\hat{\xi}\Big(\frac{1}{|x|^{\mu}}\ast \mathcal{W}^{2_{\mu}^{\ast}}\Big)\mathcal{W}^{2_{\mu}^{\ast}-2}\bigg]\eta\Psi\Bigg|\\&
\leq\Bigg|\int_{\mathbb{R}^N}\bigg[\big(\alpha-\alpha^{2\cdot2_{\mu}^{\ast}-1}\big)\Big(\frac{1}{|x|^{\mu}}\ast \mathcal{W}^{2_{\mu}^{\ast}}\Big)\mathcal{W}^{2_{\mu}^{\ast}-1}-\big(2_{\mu}^*-1\big)\alpha^{2(2_{\mu}^\ast-1)}\Big(\frac{1}{|x|^{\mu}}\ast \mathcal{W}^{2_{\mu}^{\ast}}\Big)\mathcal{W}^{2_{\mu}^{\ast}-2}\mathcal{Z}(x)\bigg]\eta\Psi\Bigg|\\&\hspace{4mm}+\Bigg|\big(2_{\mu}^*-1\big)
\alpha^{2(2_{\mu}^{\ast}-1)}\int_{\mathbb{R}^N}\Big(\frac{1}{|x|^{\mu}}\ast W^{2_{\mu}^{\ast}}\Big)W^{2_{\mu}^{\ast}-2}\mathcal{Z}_1|\eta|\Psi\Bigg|
\\&
\lesssim\Bigg|\int_{\mathbb{R}^N}\bigg[\big(\alpha-\alpha^{2\cdot2_{\mu}^{\ast}-1}\big)\Big(\frac{1}{|x|^{\mu}}\ast \mathcal{W}^{2_{\mu}^{\ast}}\Big)\mathcal{W}^{2_{\mu}^{\ast}-1}-\Big(2_{\mu}^*\alpha^{2(2_{\mu}^\ast-1)}+o(1)\Big)\Big(\frac{1}{|x|^{\mu}}\ast \big(\mathcal{W}^{2_{\mu}^{\ast}-1}\sum_{j\neq i}\alpha_jW_j\big)\Big)\mathcal{W}^{2_{\mu}^{\ast}-1}\\&\hspace{4mm}-\Big(\big(2_{\mu}^*-1\big)\alpha^{2(2_{\mu}^\ast-1)}+o(1)\Big)\Big(\frac{1}{|x|^{\mu}}\ast \mathcal{W}^{2_{\mu}^{\ast}}\Big)\mathcal{W}^{2_{\mu}^{\ast}-2}\mathcal{Z}\bigg]\eta\Psi\Bigg|\\&\hspace{4mm}
+\big(2_{\mu}^*-1\big)
\alpha^{2(2_{\mu}^{\ast}-1)}\int_{\mathbb{R}^N}\Big(\frac{1}{|x|^{\mu}}\ast W^{2_{\mu}^{\ast}}\Big)W^{2_{\mu}^{\ast}-2}\mathcal{Z}_1|\eta|\Psi\\&
\lesssim o(1)\big\|\nabla\rho\big\|_{L^2}+\Big\|\Delta u+\Big(\frac{1}{|x|^{\mu}}\ast |u|^{2_{\mu}^{\ast}}\Big)|u|^{2_{\mu}^{\ast}-2}u\Big\|_{(\mathcal{D}^{1,2}(\mathbb{R}^N))^{-1}}+\big\|\nabla\rho\big\|_{L^2}^{\min\big(2,\frac{N-\mu+2}{N-2}\big)}.
\end{split}
\end{equation}
Therefore the first conclusion of Proposition \ref{estimate2} follows from \eqref{a6} and \eqref{a7}.

Now we are in a position to complete the proof of the second conclusion of Proposition \ref{estimate2}. Using Lemma~\ref{fai2}-(2) and choosing $\epsilon=o(1)$, we get by taking $\eta=\mathcal{W}$ in \eqref{a5}, for $j\neq i$,
\begin{equation}\label{a9}
\begin{split}
&\Bigg|\big(\alpha-\alpha^{2\cdot2_{\mu}^{\ast}-1}\big)\int_{\mathbb{R}^N}\Big(\frac{1}{|x|^{\mu}}\ast \mathcal{W}^{2_{\mu}^{\ast}}\Big)\mathcal{W}^{2_{\mu}^{\ast}}\Psi-o(1)\int_{\mathbb{R}^N}\Big(\frac{1}{|x|^{\mu}}\ast \mathcal{W}^{2_{\mu}^{\ast}}\Big)\mathcal{W}^{2_{\mu}^{\ast}}\Psi
\\&\hspace{4mm}-\Big(\big(2_{\mu}^*-1\big)\alpha^{2(2_{\mu}^\ast-1)}+o(1)\Big)\int_{\mathbb{R}^N}\Big(\frac{1}{|x|^{\mu}}\ast \mathcal{W}^{2_{\mu}^{\ast}}\Big)\mathcal{W}^{2_{\mu}^{\ast}-1}\mathcal{Z}\Psi\Bigg|\\&
\lesssim o(1)\big\|\nabla\rho\big\|_{L^2}+\Big\|\Delta u+\Big(\frac{1}{|x|^{\mu}}\ast |u|^{2_{\mu}^{\ast}}\Big)|u|^{2_{\mu}^{\ast}-2}u\Big\|_{(\mathcal{D}^{1,2}(\mathbb{R}^N))^{-1}}+\big\|\nabla\rho\big\|_{L^2}^{\min\big(2,\frac{N-\mu+2}{N-2}\big)}.
\end{split}
\end{equation}
Then from the first conclusion of Proposition \ref{estimate2} and \eqref{a9} imply that, for $j\neq i$,
\begin{equation}\label{wzhao}
\begin{split}
&\Bigg|\int_{\mathbb{R}^N}\Big(\frac{1}{|x|^{\mu}}\ast \mathcal{W}^{2_{\mu}^{\ast}}\Big)\mathcal{W}^{2_{\mu}^{\ast}-1}\mathcal{Z}\Psi\Bigg|\\&
\lesssim o(1)\big\|\nabla\rho\big\|_{L^2}+\Big\|\Delta u+\Big(\frac{1}{|x|^{\mu}}\ast |u|^{2_{\mu}^{\ast}}\Big)|u|^{2_{\mu}^{\ast}-2}u\Big\|_{(\mathcal{D}^{1,2}(\mathbb{R}^N))^{-1}}+\big\|\nabla\rho\big\|_{L^2}^{\min\big(2,\frac{N-\mu+2}{N-2}\big)}.
\end{split}
\end{equation}
Combining the estimate from \cite{FG20},
$$
\int_{\mathbb{R}^N}\mathcal{W}_i^{2^{\ast}-1}\mathcal{W}_j=\int_{\mathbb{R}^N}\mathcal{W}_j^{2^{\ast}-1}\mathcal{W}_i\approx\int_{B(0,1)}\mathcal{W}_i^{2^{\ast}-1}\mathcal{W}_j\quad\mbox{for any}\quad j\neq i,
$$
then we get that
\begin{equation*}
\begin{split}
\int_{\mathbb{R}^N}\Big(\frac{1}{|x|^{\mu}}\ast\mathcal{W}_i^{2_{\mu}^{\ast}}\Big)\mathcal{W}_i^{2_{\mu}^{\ast}-1}\mathcal{W}_j&=\int_{\mathbb{R}^N}\Big(\frac{1}{|x|^{\mu}}\ast\mathcal{W}_j^{2_{\mu}^{\ast}}\Big)\mathcal{W}_j^{2_{\mu}^{\ast}-1}\mathcal{W}_i\\
&\approx\int_{B(0,1)}\Big(\frac{1}{|x|^{\mu}}\ast\mathcal{W}_i^{2_{\mu}^{\ast}}\Big)\mathcal{W}_i^{2^{\ast}-1}\mathcal{W}_j\\&
\lesssim o(1)\big\|\nabla\rho\big\|_{L^2}+\Big\|\Delta u+\Big(\frac{1}{|x|^{\mu}}\ast |u|^{2_{\mu}^{\ast}}\Big)|u|^{2_{\mu}^{\ast}-2}u\Big\|_{(\mathcal{D}^{1,2}(\mathbb{R}^N))^{-1}}+\big\|\nabla\rho\big\|_{L^2}^{\min\big(2,\frac{N-\mu+2}{N-2}\big)},
\end{split}
\end{equation*}
for any $j\neq i$
and imply the second conclusion of Proposition \ref{estimate2} holds for all $j>i$. Hence the conclusion can be deduced by the induction.

  \qed


    \end{document}